\documentclass[10pt]{article}

\usepackage{persobase}
\usepackage{persomath}
\usepackage{tikz-cd}
\tikzcdset{labels={font=\everymath\expandafter{\the\everymath\textstyle}}}
\usetikzlibrary[arrows.meta]
\usepackage{pinlabel}

\newcommand{\qsim}{/\negthickspace\sim}

\title{Ergodicity of the geodesic flow for groups with a contracting element}
\author{Rémi Coulon}
\date{\today}

\setbftrue
\intvaldfrenchtrue

\begin{document}

\maketitle

\begin{abstract}
	In this article we investigate the dynamical properties of the geodesic flow for a proper metric space whose isometry group contains a contracting element.
	We show that the existence of a contracting isometry is a sufficient evidence of negative curvature to carry in this context various results borrowed from hyperbolic geometry.
	In particular, we extend the Hopf-Tsuji-Sullivan dichotomy proving that the geodesic flow is either conservative and ergodic or dissipative.
\end{abstract}

\noindent
{\footnotesize 
\textbf{Primary MSC Classification:}   
20F65, 
20F67 
37A10, 
37A15, 
37A25, 
37D40. 
}

\tableofcontents

\section{Introduction}

\subsection{Background}
Given a hyperbolic manifold $M = \mathbf H^n / G$, the geodesic flow $(\phi_t)$ on its unit tangent bundle $SM$ is a fundamental example of dynamical system with various remarkable applications to geometry, topology, group theory and number theory.
It comes with a flow-invariant measure $m_G$, called the Bowen-Margulis measure (it may have different names in the litterature).
If $m_G$ is finite, this measure is characterized as the unique invariant measure with maximal entropy.
The Hopf-Tsuji-Sullivan dichotomy states that two mutually exclusive cases can arise.
Either the geodesic flow is conservative (almost every orbit is recurrent) and ergodic or it is dissipative (almost every orbit is divergent).

Consider now a general proper geodesic metric space $X$ endowed with a proper action by isometries of a group $G$.
As $X$ does not necessarily carry a riemannian structure, a natural replacement for the unit tangent bundle of $M = X/G$ is the quotient $\mathcal GX / G$, where $\mathcal GX$ is the space of all bi-infinite, unit-speed parametrized geodesics on $X$.
It comes with a flow $(\phi_t)$ which consists in shifting the origin of the geodesics.
It turns out that a Hopf-Tsuji-Sulllivan dichotomy often holds as soon as $X$ carries a certain form of negative curvature.
This is for instance the case if 
\begin{itemize}
	\item $X$ is a tree, see Coornaert-Papadopoulos \cite{Coornaert:1994aa};
	\item $X$ is a Gromov hyperbolic space where every pair of distinct points in the Gromov boundary are joined by a unique bi-infinite geodesic, see Kaimanovich \cite{Kaimanovich:1994wz};
	\item $X$ is CAT(-1), see Burger-Mozes \cite{Burger:1996kc} and Roblin \cite{Roblin:2003vz};
	\item $X$ is a geodesically complete CAT(0) space and $G$ contains a rank-one isometry, see Ricks  \cite{Ricks:2017aa} for proper and co-compact actions,  Link-Picaud \cite{Link:2016dc} for Hadamard manifolds, and Link \cite{Link:2018ue} for the general case.
	\item $X$ is a properly convex domain in $\mathbf P(\R^{d+1})$ and $G \subset {\rm PGL}_{d+1}(\R)$ is a discrete subgroup preserving $X$, see Blayac \cite{Blayac:2025pa}.

\end{itemize}
In a similar spirit, we should also mention that for moduli spaces, the ergodicity of the Teichmüller flow has been proved independently by Veech \cite{Veech:1982aa} and Masur \cite{Masur:1982aa}.

\begin{rema*}
	The Hopf-Tsuji-Sullivan dichotomy also admits generalizations in a higher rank setting, see Burger-Landesberg-Lee-Oh \cite{Burger:2021aa}. 
	However for the purpose of this article we will remain in the context of ``rank-one'' spaces.
\end{rema*}

\subsection{Contracting elements}
In the past decades, there have been many attempts to investigate the marks of negative curvature in groups, beyond hyperbolicity.
Contracting elements is a useful notion in this direction that goes back at least to the work of Minsky \cite{Minsky:1996wx}.
It was later formalized by Algom-Kfir \cite{Algom-Kfir:2011st} and Bestvina-Fujiwara \cite{Bestvina:2009hh}.
Roughly speaking a subset $Y \subset X$ is contracting, if any ball disjoint from $Y$ has a nearest-point projection onto $Y$, whose diameter is uniformly bounded.
An element $g \in G$ is \emph{contracting} if the orbit map $\Z \to X$ sending $n$ to $g^nx$ is a quasi-isometric embedding with contracting image.
Contracting elements can be thought of as hyperbolic directions in the space $X$.
Here are a few examples.
\begin{itemize}
	\item If $X$ is a hyperbolic space (in the sense of Gromov) endowed with a proper action of $G$, then every loxodromic element in $G$ is contracting \cite{Gromov:1987tk}.
	\item Assume that $G$ is hyperbolic relative to $\{P_1, \dots, P_m\}$.
	Suppose that $G$ acts properly, co-compactly on $X$ (in particular, $X$ is not necessarily Gromov hyperbolic).
	Any infinite order element of $G$ which is not conjugated in some $P_i$ is contracting  \cite{Sisto:2012um,Gerasimov:2016uc}.
	\item If $X$ is a proper CAT(0) space endowed then any rank-one isometry is contracting \cite{Bestvina:2009hh}.
	\item Let $\Sigma$ be surface of finite type.
	Assume that $G$ is the mapping class group of $\Sigma$ and $X$ the Teichmüller space of $\Sigma$ endowed with the Teichmüller metric.
	Then any pseudo-Anosov element is contracting \cite{Minsky:1996wx}.
	\item Suppose that $\Omega \subset \mathbf P(\R^{d+1})$ is a properly convex domain endowed with the corresponding Hilbert metric.
	Let $G \subset {\rm PGL}_{d+1}(\R)$ be a discrete group preserving $\Omega$.
	Then any rank-one element of $G$, in the sense of Islam, is contracting. 
	See Islam \cite[Theorem~10.1]{Islam:2025ra}.
\end{itemize}
Groups with a contracting elements are known to be acylindrically hyperbolic, see Sisto \cite{Sisto:2018uc} (the converse direction is still open though).
However one of our motivation to develop this theory is to tackle counting problems.
Thus we do not want to change the metric we are working with.
There exists now a large literature on groups with a (strongly) contracting elements, in particular, extending counting results known to hold in hyperbolic groups (which are of particular interest for our study).
Let us mention for instance the works of Arzhantseva-Cashen-Tao \cite{Arzhantseva:2015cl} and Yang \cite{Yang:2014aa}  studying growth tightness, Legaspi \cite{Legaspi:2024gr} and Arzhantseva-Cashen \cite{Arzhantseva:2020aa} investigating the growth of quasi-convex subgroups and normal subgroups respectively, Yang \cite{Yang:2020ub} proving the genericity of contracting elements, Yang \cite{Yang:2019wa} exploring statistically convex co-compact actions.
Note that, at the time, no ``Patterson-Sullivan theory'' existed in this context.
These authors developed ad hoc methods for each problem.
The first articles laying the foundations of conformal densities for groups with a contracting elements are independently due to the author \cite{Coulon:2022tu} and Yang \cite{Yang:2022aa}.

\subsection{Main result}
The goal of this article is to extend the Hopf-Tsuji-Sullivan dichotomy, whenever $G$ admits a contracting element for its action on $X$, thus providing a common approach for the aforementioned examples.
In order to state precisely the main result, let us introduce first some important objects.
We assume here that $X$ is a proper geodesic metric space and $G$ a group acting properly, by isometries on $X$ with a contracting element.
In addition, we suppose that $G$ is not virtually cyclic.

\paragraph{The coarse unit tangent bundle.}
The \emph{horocompactification} of $X$, denoted by $\bar X$, is the smallest compactification of $X$ such that the map $X \times X \times X \to \R$ sending $(x,y,z)$ to $\dist xz - \dist yz$ extends continuously to a map $X \times X \times \bar X \to \R$.
The horoboundary $\partial X = \bar X \setminus X$ can be seen as a set of cocycles $c \colon X \times X \to \R$, generalizing the notion of Busemann cocycles.

\begin{exam*}
	If $X$ is CAT(0), then its horoboundary coincides with the visual boundary.
	If $X$ is Gromov hyperbolic, then $\partial X$ maps onto the Gromov boundary of $X$ (see below).
\end{exam*}

\noindent
We now consider the sets 
\begin{align*}
	\partial^2 X & = \set{(c,c') \in \partial X \times \partial X}{ \gro c{c'}o < \infty},\\
	SX & = \partial^2 X \times \R,
\end{align*}
(the notation $\gro c{c'}o$ stands for the Gromov product of $c$ and $c'$ at a point $o \in X$ -- see \autoref{sec: gromov product} -- in particular if $(c,c')$ belongs to $\partial^2X$, then $c$ and $c'$ are distinct).
Both spaces come with a natural action of $G$, induced by the one on $X$.
One also defines a flow $(\phi_t)$ on $SX$ by translating the coordinate in the $\R$ factor.
This flow commutes with the action of $G$.
In addition, there is a natural $G$-invariant, flow-invariant map $\mathcal GX \to SX$.

\begin{rema*}
	If $X$ is a geodesically complete, simply-connected manifold with negative sectional curvature, then $SX$ is homeomorphic to the unit tangent bundle of $X$.
	More generally, if $X$ is CAT(-1) then the map $\mathcal GX \to SX$ is a homeomorphism.
	This identification is known as the Hopf coordinates.
\end{rema*}

In our settings, the map $\mathcal GX \to SX$ need not be onto (not any two points in the boundary are joined by a geodesic) nor one-to-one (see \autoref{rem: not proper action}).
Note also that $X$ is not assumed to be geodesically complete.
Nevertheless it turns out that the flow $(\phi_t)$ on $SX/G$ is an efficient analogue of the geodesic flow on a hyperbolic manifold.
In practice, the definition of the quotient space $SX/G$ requires some caution. 
Indeed, in general the action of $G$ on $SX$ is not properly discontinuous (see \autoref{rem: not proper action} and the discussion below).

\paragraph{Bowen-Margulis measure.}
Now that we have described the space, let us explain what measure we will consider.
We follow a classical construction due to Patterson and Sullivan \cite{Patterson:1976hp,Sullivan:1979vs} which has been extended by Knieper in the CAT(0) setting \cite{Knieper:1998vj}.
Fix a base point $o \in X$.
In a previous work \cite{Coulon:2022tu} we built $G$-invariant, conformal densities $\nu = (\nu_x)_{x \in X}$ supported on $\partial X$ (see \autoref{sec: conformal densities}).
First we form a $G$-invariant current $\mu$ on $\partial^2X$ which is absolutely continuous with respect to the product measure $\nu_o \otimes \nu_o$.
Then we endow $SX$ with the measure $m = \mu \otimes ds$ where $ds$ is the Lebesgue measure on $\R$.
Finally we consider the sub-$\sigma$-algebra $\mathfrak L_G$ which consists of all $G$-invariant Borel subsets.
As we mentioned before, the action of $G$ on $SX$ is not necessarily properly discontinuous.
However its restriction to a sufficiently large ``squeezing'' subset $S_0X$ is (see \autoref{def: squeezing}).
We use this subspace to derive from $m$ a measure $m_G$ on $\mathfrak L_G$ which plays the role of the measure on the quotient space $SX/G$.
The latter is our Bowen-Margulis measure.

However, it turns out that the space $SX$ is too large for our purpose.
Assume for instance that the product $X = \H^2 \times \intval 01$ is endowed with the $L^1$-metric and $G$ is a group acting properly by isometries on $\H^2$ and trivially on $\intval 01$.
For every $u \in \intval 01$, the ``layer'' $\H^2/ G \times \{u\} \subset SX/G$ is flow invariant.
Thus if our measure gives positive mass to several layers, one cannot expect to prove that the geodesic flow is ergodic.
Already the horoboundary $\partial X$ is not suitable for the $G$-invariant, conformal densities to behave nicely -- see the discussion in \cite{Coulon:2022tu}.
Motivated by the analogy with Gromov hyperbolic spaces, one could be tempted to replace everywhere $\partial X$ by $\partial X\qsim$, where $\sim$ is the equivalence relation identifying two cocycles $c_1, c_2 \in \partial X$ whenever $\norm[\infty]{c_1 -c_2} < \infty$.
Indeed, in this situation $\partial X \qsim$ is homeomorphic to the Gromov boundary of $X$.
However, as a topological space, the quotient $\partial X\qsim$ can in general be rather wild (it does not have to be Hausdorff for instance).
For this reason, we adopt a \emph{measurable} point of view.
We denote by $\mathfrak B$ the Borel $\sigma$-algebra on $\bar X$ and by $\mathfrak B^+$ the sub-$\sigma$-algebra which consists of all Borel subsets which are saturated for the equivalence relation $\sim$.
We think of the \emph{reduced horoboundary} as the measurable space $(\partial X, \mathfrak B^+)$.
When restricted to $\mathfrak B^+$, $G$-invariant, conformal densities display many features which are typical from hyperbolic geometry, see Coulon  \cite{Coulon:2022tu}.
A similar phenomenon arises for the geodesic flow.
More precisely we endow $SX$ with a suitable $G$-invariant $\sigma$-algebra $\mathfrak L^+$ and then define $\mathfrak L_G^+$ as the collection of all $G$-invariant subsets in $\mathfrak L^+$.
The dynamical system we are really interested in is the geodesic flow $(\phi_t)$ acting on $(SX, \mathfrak L_G^+, m_G)$.

\paragraph{The Hopf-Tsuji-Sullivan dichotomy.}
We are now in a position to state our main result.
Recall that the \emph{Poincaré series} of the group $G$ is
\begin{equation*}
	\mathcal P_G(s) = \sum_{g \in G}e^{-s\dist o{go}}.
\end{equation*}
Its \emph{critical exponent} is denoted by $\omega_G$.
It has the property that $\mathcal P_G(s)$ converges (\resp diverges) whenever $s > \omega_G$ (\resp $s < \omega_G$).
We say that the action of $G$ on $X$ is \emph{divergent} if $\mathcal P_G(s)$ diverges at $s = \omega_G$.
The \emph{radial limit set} of $G$ roughly consists of all ``directions'' $c \in \partial X$ for which there is a sequence $(g_no)$ in the orbit of $G$ that radially converges to $c$, see \autoref{exa: radial limit sets}\ref{enu: radial limit sets - radial} for a precise definition.

\begin{theo}
	Let $X$ be a proper, geodesic, metric space.
	Fix a base point $o \in X$.
	Let $G$ be a group acting properly, by isometries on $X$.
	We assume that $G$ is not virtually cyclic and contains a contracting element.
	Let $\omega \in \R_+$.
	Let $\nu  = (\nu_x)$ be a $G$-invariant, $\omega$-conformal density supported on $\partial X$.
	Let $m_G$ be the associated Bowen-Margulis measure on $(SX, \mathfrak L_G)$.
	Then one of the following two cases holds.
	Moreover, for each one all the stated properties are equivalent.
	\begin{itemize}
		\item \emph{\bfseries Convergent case.}
		\begin{enumerate}[label=(c\arabic{*}), ref=(c\arabic{*})]
			\item \label{eqn: main - hopf-tsuji-sullivan - conv - def}
			The Poincaré series $\mathcal P_G(s)$ converges at $s = \omega$.
			\item \label{eqn: main - hopf-tsuji-sullivan - conv - radial}
			The measure $\nu_o$ gives zero measure to the radial limit set.
			\item \label{eqn: main - hopf-tsuji-sullivan - conv - bdy}
			The diagonal action of $G$ on $(\partial X \times \partial X, \mathfrak B \otimes \mathfrak B, \nu_o \otimes \nu_o)$ is not conservative.
			\item \label{eqn: main - hopf-tsuji-sullivan - conv - dissipative}
			The geodesic flow on $(SX, \mathfrak L_G, m_G)$ is dissipative.
		\end{enumerate}
		
		\item \emph{\bfseries Divergent case.}
		\begin{enumerate}[label=(d\arabic{*}), ref=(d\arabic{*})]
			\item \label{eqn: main - hopf-tsuji-sullivan - div - def}
			The Poincaré series $\mathcal P_G(s)$ diverges at $s = \omega$, hence $\omega = \omega_G$.
			\item \label{eqn: main - hopf-tsuji-sullivan - div - radial}
			The measure $\nu_o$ gives full measure to the radial limit set.
			\item \label{eqn: main - hopf-tsuji-sullivan - div - bdy}
			The diagonal action of $G$ on $(\partial X \times \partial X, \mathfrak B \otimes \mathfrak B, \nu_o \otimes \nu_o)$ is conservative.
			\item \label{eqn: main - hopf-tsuji-sullivan - div - conservative}
			The geodesic flow on $(SX, \mathfrak L_G, m_G)$ is conservative.
			\item \label{eqn: main - hopf-tsuji-sullivan - div - double ergo}
			The diagonal action of $G$ on $(\partial X \times \partial X, \mathfrak B^+ \otimes \mathfrak B^+, \nu_o \otimes \nu_o)$ is ergodic.
			\item \label{eqn: main - hopf-tsuji-sullivan - div - ergo flow}
			The geodesic flow on $(SX, \mathfrak L^+_G, m_G)$ is ergodic.
		\end{enumerate}
	\end{itemize}
\end{theo}

If $X$ is Gromov hyperbolic, the ergodicity of the geodesic flow was proved by Kaimanovich \cite{Kaimanovich:1994wz} under the assumption that every pair of distinct points in the Gromov boundary are joined by a unique bi-infinite geodesic.
This assumption has been removed by Bader-Furman \cite{Bader:2017te} when the action of $G$ on $X$ is proper and co-compact (actually their argument works if the Bowen-Margulis measure is finite).
The general case is covered in Coulon-Dougall-Schapira-Tapie \cite{Coulon:2018va}.
Nevertheless, to the best of our knowledge, even in this rather classical settings, the full dichotomy does not appear in the litterature.

If $X$ is a CAT(0) space, then the (regular) horoboundary and the reduced horoboundary coincide.
Hence we recover the work of Link \cite{Link:2018ue}.
Actually, in the divergent case, we even get a slightly more general result (see below).

Let $\Sigma$ be an orientable surface of finite type, whose mapping class group $G = \mcg \Sigma$  is not virtually cyclic.
Assume that $X$ is the Teichmüller space of $\Sigma$ endowed with the Teichmüller metric.
Recall that every pseudo-Anosov element in $G$ is contracting for its action on $X$.
Thus our work applies in this situation as well.
Arzhantseva, Cashen and Tao \cite[Section~10]{Arzhantseva:2015cl} also observed that the work of Eskin, Mirzakani, and Rafi \cite[Theorem~1.7]{Eskin:2019uw} implies that the action of $G$ on $X$ is strongly positively recurrent (a.k.a. statistically convex co-compact).
Consequently the Poincaré series $\mathcal P_G(s)$ diverges at the critical exponent $s = \omega_G$ -- see Yang \cite{Yang:2019wa} for a geometric proof or Coulon \cite{Coulon:2022tu} for a more dynamical proof inspired by Schapira-Tapie \cite{Schapira:2021ti}.
This provides an alternative point of view on the ergodicity of the Teichmüller flow.
See also Gekhtman \cite{Gekhtman:2013dy}.

We noticed previously that our framework applies for a discrete group acting on a properly convex domain (endowed with the corresponding Hilbert geometry) with a rank-one element.
We recover the ergodicity of the geodesic flow that was proved in this context by Blayac \cite{Blayac:2025pa}.
Actually, it is worth mentioning that Blayac makes an intensive use of the horocompactification to build his invariant conformal densities, just as we do.

\subsection{Strategy}

Let us now say a few words about our strategy.
Although the main scheme follows classical arguments, there are various (sometimes subtle) difficulties that needed be overcome.
In our previous work \cite{Coulon:2022tu} we proved the equivalences \ref{eqn: hopf-tsuji-sullivan - div - def} $\Leftrightarrow$ \ref{eqn: hopf-tsuji-sullivan - div - radial} and  \ref{eqn: hopf-tsuji-sullivan - conv - def} $\Leftrightarrow$ \ref{eqn: hopf-tsuji-sullivan - conv - radial}. 
This part has been obtained independently by Yang \cite{Yang:2022aa}.
All the other equivalences are new with this level of generality.
The relation with conservativity / dissipativity is rather standard --- there are some technicalities that need be addressed though, such as the study of atoms in $(\partial X, \mathfrak B^+, \nu_o)$.

\paragraph{The Hopf argument.}
Our main contribution is the ergodicity of the geodesic flow.
The proof follows the famous Hopf argument \cite{Hopf:1937kk}.
If $X = \mathbf H^n$ is the hyperbolic space, this approach can be summarized as follows.
Fix a positive summable function $\rho \in L^1(SX/G,m_G)$.
Let $f \in L^1(SX/G, m_G)$.
If the geodesic flow is conservative, then the Hopf ergodic Theorem states that for almost every vector $v \in SX/G$, the limit
\begin{equation*}
	\lim_{T \to \pm\infty} 
	\frac{\displaystyle\int_0^T f \circ \phi_t(v)dt}{\displaystyle \int_0^T \rho \circ \phi_t(v) dt} 
\end{equation*}
exists and defines a flow-invariant function $f_\infty$ on $SX/G$.
Geodesics in $\H^n$ satisfy a \emph{convergence property}: if $\gamma_1$ and $\gamma_2$ are two geodesic rays whose endpoints at infinity coincide, then there is $u \in \R$ such that
\begin{equation*}
	\lim_{t \to \infty} \dist {\gamma_1(t)}{\gamma_2(t+u)} = 0
\end{equation*}
(the convergence is even exponentially fast).
This feature can be used to prove that if $f$ is continuous with compact support, then $f_\infty(v)$ actually only depends on the ``future'' of $v$ and by symmetry only on its ``past''.
As $m_G$ is built from a product measure on $\partial^2X$, one deduces with a Fubini argument that $f_\infty$ is essentially constant.
Since continuous functions with compact support are dense in $L^1(SX/G,m_G)$ we can conclude that the geodesic flow is ergodic.

When $X$ is Gromov hyperbolic or CAT(0) this convergence property does not holds anymore.
This is one reason why Kaimanovich \cite[Theorem~2.6]{Kaimanovich:1994wz} and Link \cite[Theorem~A]{Link:2018ue} require an additional assumption.
Kaimanovich asks that any two points in the boundary are joined by a unique geodesic (with some convergence property), while Link assumes that the space $X$ contains a zero-width geodesic (i.e. a geodesic not bounding a flat strip).
Our statement does not make any such hypothesis.
To bypass this difficulty we extend the approach used by Bader-Furman \cite{Bader:2017te} for Gromov hyperbolic spaces.
See also Coulon-Dougall-Schapira-Tapie \cite{Coulon:2018va}.
The idea is to run the Hopf argument on a different dense subset of functions which are well behaved for another measurement of the convergence.
If $X$ is Gromov hyperbolic, the convergence property can roughly be stated as follows.
Consider two bi-infinite geodesics $\gamma_1$ and $\gamma_2$ starting at $\eta_1, \eta_2 \in \partial X$ and ending both at $\xi \in \partial X$.
As $t$ tends to infinity, the visual metric between $\eta_1$ and $\eta_2$ seen from $\gamma_i(t)$ decays exponentially.

In our context, there is probably no hope to build a similar global visual metric on $\partial X$.
(Recall that if $G$ is a non-trivial free product, then the action on its Cayley graph $X$ admits a contracting element.
The space $X$ can therefore be pretty uncontrolled.)
However we know that the measure $\nu_o$ on $\partial X$ gives full measure to the radial limit set.
An important part of the article consists in defining a ``visual metric'' on $\Lambda\qsim$, where $\Lambda$ is a suitable recurrent subset of the radial limit set (see \autoref{res: reduced bdy standard}).
This allows us to run the Hopf argument with $\mathcal E \otimes L^1(\R, ds)$, where $\mathcal E \subset L^1(\Lambda \times \Lambda, \mathfrak B^+ \otimes \mathfrak B^+, \mu)$ consists of functions which are, roughly speaking, locally constant for the visual metric (see \autoref{sec: geodesic flow}).

\begin{rema}
	Another strategy for this last step would have been to use the theory of projection complexes developed by Bestvian-Bromberg-Fujiwara \cite{Bestvina:2015tv} to build a coarsely Lipschitz map from $X$ to a (non locally finite) Gromov hyperbolic graph and then exploit the features of hyperbolic geometry.
	This is for instance the path chosen by Yang \cite{Yang:2022aa} to prove the equivalence \ref{eqn: hopf-tsuji-sullivan - div - def} $\Leftrightarrow$ \ref{eqn: hopf-tsuji-sullivan - div - radial}.
	Nevertheless, we preferred to use visual metrics.
	Indeed, although it is a bit technical, it is fairly simple and only relies on elementary geometric properties of contracting elements.
	Moreover working in a hyperbolic space will not be useful to bypass the subtle difference between the Gromov boundary and the horoboundary, which is central in our work.
\end{rema}

\paragraph{Acknowledgment.}
The author is grateful to the \emph{Centre Henri Lebesgue} ANR-11-LABX-0020-01 for creating an attractive mathematical environment.
He acknowledges support from the Agence Nationale de la Recherche under the grant GOFR (ANR-22-CE40-0004).
The IMB receives support from the EIPHI Graduate School (ANR-17-EURE-0002).
This work started during the trimester program \emph{Groups Acting on Fractals, Hyperbolicity and Self-Similarity} hosted at  the Institut Henri Poincaré (UAR 839 CNRS-Sorbonne Université, LabEx CARMIN, ANR-10-LABX-59-01) in Spring 2022 and was completed at the Centre de Recherche Mathématiques de Montréal (IRL3457) during the \emph{Geometric Group Theory} semester in Spring 2023.
The author thanks these two institutes for their hospitality and support.
He also thanks Ilya Gekhtman, Barbara Schapira and Samuel Tapie for enlightening discussions as well as the anonymous referee for their careful reading and useful comments.

%
\section{Groups with a contracting element}
\label{sec: contracting}
%

\paragraph{Notations and vocabulary.}
In this article $(X,d)$ is always a proper geodesic metric space.
The open ball of radius $r \in \R_+$ centered at $x \in X$ is denoted by $B(x,r)$.
Let $Y$ be a closed subset of $X$.
Given $x \in X$, a point $y \in Y$, is a \emph{projection} of $x$ on $Y$ if $\dist xy = d(x,Y)$.
The \emph{projection} of a subset $Z\subset X$ onto $Y$ is
\begin{equation*}
	\pi_Y(Z) = \set{y \in Y}{ y \ \text{is the projection of a point}\ z \in Z}.
\end{equation*}
Let $\gamma \colon I \to X$ be a continuous path intersecting $Y$.
The \emph{entry point} and \emph{exit point} of $\gamma$ in $Y$ are $\gamma(t)$ and $\gamma(t')$ where
\begin{equation*}
	t = \inf \set{s \in I}{\gamma(s) \in Y} 
	\quad \text{and} \quad
	t' = \sup \set{s \in I}{\gamma(s) \in Y} .
\end{equation*}
If $I$ is a compact interval, such points always exist (the subset $Y$ is closed).
Given $d\in \R_+$, we denote by $\mathcal N_d(Y)$ the \emph{$d$-neighborhood} of $Y$, that is the set of points $x \in X$ such that $d(x,Y) \leq d$.
The distance between two subsets $Y, Y'$ of $X$ is 
\begin{equation*}
	d(Y,Y') = \inf_{(y,y') \in Y \times Y'} \dist y{y'}.
\end{equation*}
We adopt the convention that the diameter of the empty set is zero.

\paragraph{Contracting set.}

\begin{defi}[Contracting set]
	Let $\alpha \in \R_+^*$ and $Y \subset X$ be a closed subset.
	We say that $Y$ is \emph{$\alpha$-contracting} if for any geodesic $\gamma$ with $d(\gamma, Y)\geq \alpha$, we have $\diam{ \pi_Y(\gamma)} \leq \alpha$.
	The set $Y$ is \emph{contracting}, if $Y$ is $\alpha$-contracting for some $\alpha \in \R^*_+$.	
\end{defi}

The next statements are direct consequences of the definition.
Their proofs are left to the reader.

\begin{lemm}
\label{res: qc contracting set}
	Let $Y$ be an $\alpha$-contracting subset.
	If $\gamma$ is a geodesic joining two points of $\mathcal N_\alpha(Y)$, then $\gamma$ lies in the $(5\alpha/2)$-neighborhood of $Y$.
\end{lemm}

\begin{lemm}
\label{res: proj contracting set}
	Let $Y$ be an $\alpha$-contracting subset.
	Let $x,x' \in X$ and $\gamma$ be a geodesic from $x$ to $x'$.
	Let $p$ and $p'$ be respective projections of $x$ and $x'$ onto $Y$.
	If $d(x, Y) < \alpha$ or $\dist p{p'} > \alpha$, then the following holds:
	\begin{enumerate}
		\item $d(\gamma,Y) < \alpha$;
		\item the entry point (\resp exit point) of $\gamma$ in $\mathcal N_\alpha(Y)$ is $2\alpha$-closed to $p$ (\resp $p'$);
		\item $\dist x{x'} \geq \dist xp + \dist p{p'} +  \dist {p'}{x'} - 8\alpha$.
	\end{enumerate}
\end{lemm}

\begin{rema}
\label{rem: Lipschitz proj}
	The last inequality can be improved if $x$ or $x'$ already belongs to $Y$.
	It follows from the above statement that the nearest-point projection onto $Y$ is large-scale $1$-Lipschitz.
	More precisely, if $Z \subset X$ is a subset then $\diam {\pi_Y(Z)} \leq \diam Z + 4\alpha$.
\end{rema}

\begin{lemm}
\label{res: intersection vs projection}
	Let $Y, Y' \subset X$ be two $\alpha$-contracting sets.
	For every $A \in \R_+$, the following inequalities hold:
	\begin{align*}
		\diam{\mathcal N_A(Y) \cap \mathcal N_A(Y')}
		& \leq \diam{\mathcal N_{5\alpha/2}(Y) \cap \mathcal N_{5\alpha/2}(Y') }+ 2A + 4\alpha, \\
		\diam { \mathcal N_{5\alpha/2}(Y) \cap \mathcal N_{5\alpha/2}(Y')}
		& \leq \diam{ \pi_Y(Y')} + 14\alpha.
	\end{align*}
\end{lemm}


\paragraph{Contracting element.}
Let $o \in X$ be a base point.
Consider a group $G$ acting properly by isometries on $X$.
For every $r \in \R_+$ we denote by $B_G(r)$ the ball of radius $r$ in $G$ for the pseudo-metric induced by $X$, that is 
\begin{equation}
\label{eqn: induced ball}
	B_G(r) = \set{g \in G}{\dist o{go} \leq r}.
\end{equation}
Similarly, given $r, a \in \R_+$, the induced sphere of radius $r$ in $G$ is 
\begin{equation}
\label{eqn: induced sphere}
	S_G(r, a) = \set{g \in G}{r - a \leq \dist o{go} < r + a}.
\end{equation}

\begin{defi}[Contracting element]
	An element $g \in G$ is \emph{contracting}, for its action on $X$, if the orbit map $\Z \to X$ sending $n$ to $g^nx$ is a quasi-isometric embedding with contracting image.
\end{defi}

Let $g \in G$ be a contracting element.
We call \emph{axis} of $g$ any closed subset $Y \subset X$ such that the Hausdorff distance between $Y$ and some (hence any) orbit of $\group g$ is finite.
By definition it is quasi-isometric to $\R$.   
Define $E(g)$ as the set of elements $u \in G$ such that the Hausdorff distance between $Y$ and $uY$ is finite.
It follows from the definition that $E(g)$ is a subgroup of $G$ that does not depend on the axis $Y$.
It is the maximal, virtually cyclic subgroup of $G$ containing $\group g$.
Moreover, $E(g)$ is almost-malnormal, that is $uE(g)u^{-1} \cap E(g)$ is finite, for every $u \in G \setminus E(g)$.
There is $C \in \R_+$ such that for every $u \in G$,
\begin{itemize}
	\item $\diam{\pi_Y(uY)} \leq C$, if $u \notin E(g)$;
	\item the Hausdorff distance between $Y$ and $uY$ is at most $C$, if $u \in E(g)$;
\end{itemize}
(of course the parameter $C$ does depend on the chosen axis $Y$).
See Yang \cite[Lemma~2.11]{Yang:2019wa}.

\begin{prop}
\label{res: suff condition contracting}
	Let $g \in G$ be a contracting element and $Y$ an axis of $g$.
	There is $D\in \R_+$ with the following property: if $H \subset G$ is a subgroup that does not contain a contracting element, then the projection on $Y$ of 
	\begin{equation*}
		\bigcup_{h \in H \setminus E(g)} hY 
	\end{equation*}
	has diameter at most $D$.
\end{prop}

\begin{proof}
	Since the statement is of qualitative nature, we only sketch the proof. 
	In particular, all our estimates will be very generous.
	We refer the reader to Arzhantseva-Cashen-Tao \cite[section~3]{Arzhantseva:2015cl} or Yang \cite[Section~2.3]{Yang:2019wa} where similar arguments are run with precise computations.
	Without loss of generality we can assume that $Y$ is a bi-infinite geodesic $\gamma \colon \R \to X$.
	We fix $\alpha \in \R_+^*$ such that $\gamma$ is $\alpha$-contracting and $\diam{\pi_\gamma(u\gamma)} \leq \alpha$, for every $u \in G \setminus E(g)$.
	We begin with the following claim.
	
	\begin{clai}
		There is $L\in\R_+$ such that for every $h\in G\setminus E(g)$, if
		\begin{equation*}
			d\left(\pi_\gamma(h^{-1}\gamma), \pi_\gamma(h\gamma)\right) > L,
		\end{equation*}
		then $h$ is contracting.
	\end{clai}
	
	Up to changing the parametrization of $\gamma$ we can assume that the projection on $\gamma$ of $h^{-1}\gamma$ and $h \gamma$ are respectively contained in $\gamma(I_-)$  and $\gamma(I_+)$ where $I_- = \intval {-\alpha}0$ and $I_+ =\intval \ell{\ell + \alpha}$ for some $\ell \in [L, \infty)$.
	Let $m = \gamma(\ell /2)$.
	By construction any projection of $h^{-1}m$ (\resp $hm$) on $\gamma$ belongs to $\gamma(I_-)$ (\resp $\gamma(I_+)$).
	Hence if $L$ is sufficiently large then $m$ is $5\alpha$-close to any geodesic from $h^{-1}m$ to $hm$.
	Denote by $\nu \colon \intval 0T \to X$ a geodesic from $h^{-1}m$ to $m$.
	Note that the concatenation of $\nu$ and $h\nu$ is ``almost'' a geodesic from $h^{-1}m$ to $hm$ (see \autoref{fig: path nu}).
	We extend $\nu$ to a bi-infinite $\group h$-invariant path, still denoted by $\nu \colon \R \to X$, which is characterized as follows: $\nu(t + kT) = h^k \nu(t)$ for every $t \in \R$ and $k \in \Z$.
	One proves that $\nu$ is a quasi-geodesic which  fellow-travels for a long time with $h^k\gamma$, for every $k \in \Z$.
	This suffices to show that $h$ is contracting and completes the proof of our claim.

	\begin{figure}[htbp]
		\begin{center}
			\labellist
			\small\hair 2pt
			 \pinlabel {$h^{-1}m$} [ ] at 175 125
			 \pinlabel {$m$} [ ] at 273 47
			 \pinlabel {$hm$} [ ] at 495 127
			 \pinlabel {$\pi_\gamma(h\gamma)$} [t] at 425 12
			 \pinlabel {$\pi_\gamma(h^{-1}\gamma)$} [t] at 136 10
			 \pinlabel {$\nu$} [ ] at 163 61
			 \pinlabel {$h\gamma$} [ ] at 365 152
			 \pinlabel {$h^{-1}\gamma$} [ ] at 37 99
			 \pinlabel {$\gamma$} [ ] at 489 27
			\endlabellist
			\includegraphics[width=\linewidth]{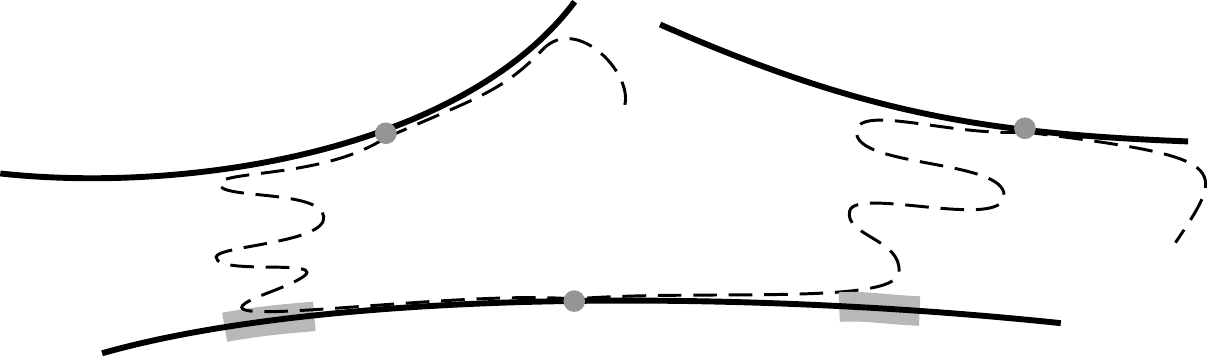}
			\vskip 5mm
			\caption{The dashed line represents the map $\nu$}
			\label{fig: path nu}
		\end{center}
	\end{figure}

	\medskip
	We now fix a set $S \subset H$ of representatives of $H / H \cap E(g)$.
	Recall that the Hausdorff distance between $\gamma$ and $u\gamma$ is uniformly bounded as $u$ runs over $E(g)$.
	Hence it suffices to bound the diameter of the projection on $\gamma$ of the set
	\begin{equation*}
		Z = \bigcup_{h \in S \setminus \{1\}} h\gamma
	\end{equation*}
	Let $h_1,h_2 \in S \setminus\{1\}$ be two distinct elements.
	Without loss of generality we can assume that
	\begin{equation*}
		d\left( \pi_\gamma(h_1\gamma), \pi_\gamma(h_2\gamma)\right)  \geq 100\alpha.
	\end{equation*}
	Set $h = h_2^{-1}h_1$.
	By construction,
	\begin{equation*}
		d\left(  \pi_{h_1^{-1}\gamma}\left(\gamma\right),  \pi_{h_1^{-1}\gamma}\left(h^{-1}\gamma\right)\right)
		\geq d\left(  \pi_{\gamma}\left(h_1\gamma\right),  \pi_{\gamma}\left(h_2\gamma\right)\right)
		 \geq 100\alpha
	\end{equation*}
	One shows using \autoref{res: proj contracting set} that the projection of $h^{-1}\gamma$ on $\gamma$ lies in the $10\alpha$-neihgbourhood of $\pi_\gamma(h_1^{-1}\gamma)$ (see \autoref{fig: axes}).
	\begin{figure}[htbp]
		\begin{center}
			\vskip 5mm
			\labellist
			\tiny\hair 2pt
			 \pinlabel {$\pi_\gamma(h_2\gamma)$} [ ] at 655 50
			 \pinlabel {$\gamma$} [ ] at 690 15
			 \pinlabel {$\pi_\gamma\left(h_2^{-1}\gamma\right)$} [ ] at 540 40
			 \pinlabel {$\pi_\gamma\left(h_1^{-1}\gamma\right)$} [ ] at 338 50
			 \pinlabel {$\pi_\gamma\left(h_1\gamma\right)$} [ ] at 220 60
			 \pinlabel {$h_2^{-1}\gamma$} [ ] at 510 210
			 \pinlabel {$h_1^{-1}\gamma$} [ ] at 420 205
			 \pinlabel {$\pi_{h_2^{-1}\gamma}\left(\gamma\right)$} [ ] at 620 180
			 \pinlabel {$\pi_{h_2^{-1}\gamma}\left(h\gamma\right)$} [ ] at 760 180
			 \pinlabel {$h\gamma$} [ ] at 890 128
			 \pinlabel {$\pi_{h_1^{-1}\gamma}\left(\gamma\right)$} [ ] at 310 185
			 \pinlabel {$\pi_{h_1^{-1}\gamma}\left(h^{-1}\gamma\right)$} [bl] at 140 200
			 \pinlabel {$h^{-1}\gamma$} [ ] at 27 138
			\endlabellist
			\vskip 5mm
			\includegraphics[width=\linewidth]{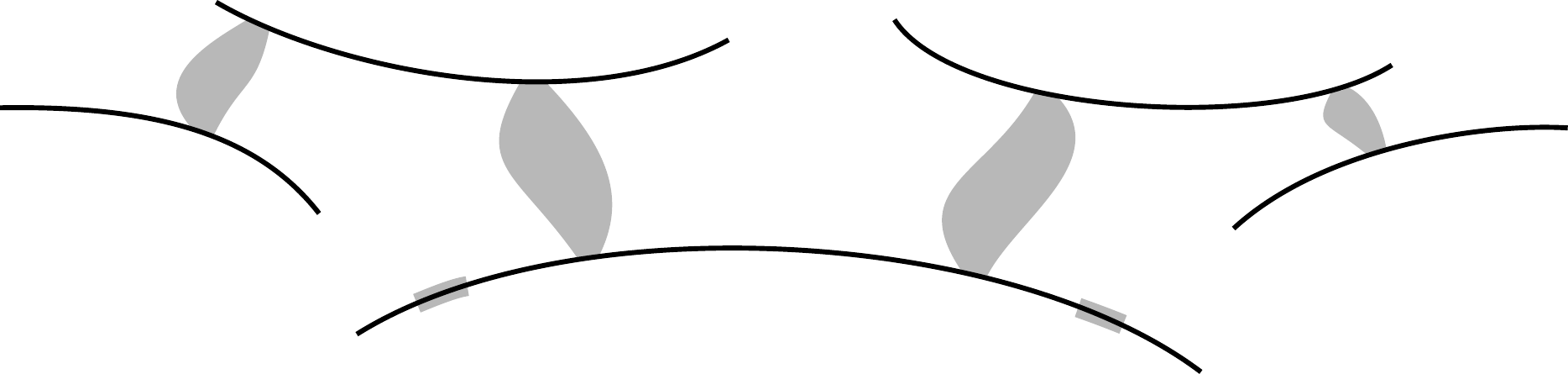}
			\caption{Projections of axes}
			\label{fig: axes}
		\end{center}
	\end{figure}
	Similarly one observes that 
	\begin{equation*}
		d\left(  \pi_{h_2^{-1}\gamma}\left(\gamma\right),  \pi_{h_2^{-1}\gamma}\left(h\gamma\right)\right)
		\geq d\left(  \pi_{\gamma}\left(h_2\gamma\right),  \pi_{\gamma}\left(h_1\gamma\right)\right)
		 \geq 100\alpha
	\end{equation*}
	and proves that the projection of $h\gamma$ on $\gamma$ lies in the $10\alpha$-neihgbourhood of $\pi_\gamma(h_2^{-1}\gamma)$.
	In view of our previous claim
	\begin{equation*}
		d\left(\pi_\gamma(h_i^{-1}\gamma), \pi_\gamma(h_i\gamma)\right) \leq L, 
		\quad \forall i \in \{1, 2\}.
	\end{equation*}
	Consequently
	\begin{align*}
		d\left(\pi_\gamma(h^{-1}\gamma), \pi_\gamma(h\gamma)\right)
		& \geq d\left(\pi_\gamma(h_1^{-1}\gamma), \pi_\gamma(h_2^{-1}\gamma)\right)  - 50\alpha \\
		& \geq d\left( \pi_\gamma(h_1\gamma), \pi_\gamma(h_2\gamma)\right)  - 2L - 100\alpha.
	\end{align*}
	Since $h_1$ and $h_2$ are distinct, $h$ does not belong to $E(g)$.
	It is neither contracting by assumption.
	It follows from our previous claim that the projection of $Z$ on $\gamma$ has diameter at most $3L + 200\alpha$.
\end{proof}

\section{Compactification of $X$}

\subsection{Horocompactification}
Let $C(X)$ be the set of all continuous functions on $X$ endowed with the topology of uniform convergence on every compact subset.
We denote by $C^*(X)$ the quotient of $C(X)$ by the subspace consisting of all constant functions.
It is endowed with the quotient topology.
Alternatively $C^*(X)$ is the set of continuous cocycles $c \colon X \times X \to \R$.
By \emph{cocycle} we mean that $c(x,z) = c(x,y) + c(y,z)$, for every $x,y,z \in X$.
For example, given $z \in X$, we define the cocycle $b_z \colon X \times X \to \R$ by 
\begin{equation*}
	b_z(x,y) = \dist xz - \dist yz, \quad \forall x,y \in X.
\end{equation*}
Since $X$ is geodesic, the map
\begin{equation*}
	\begin{array}{rccc}
		\iota \colon & X & \to & C^*(X) \\
			& z & \mapsto & b_z
	\end{array}
\end{equation*}
is a homeomorphism from $X$ onto its image.

\begin{defi}[Horoboundary]
	The \emph{horocompactification} $\bar X$ of $X$ it the closure of $\iota(X)$ in $C^*(X)$.
	The \emph{horoboundary} of $X$ is the set $\partial X = \bar X \setminus \iota(X)$.
\end{defi}

From now on, we often identify $X$ and its image in $C^*(X)$.

\begin{exam}
\label{exa: busemann cocycle}
	Any geodesic ray $\gamma \colon \R_+\to X$ defines a \emph{Busemann cocycle} $c'_\gamma  \in \partial X$ given by 
	\begin{equation*}
		c'_\gamma(x,y) = \lim_{t \to +\infty} \dist x{\gamma(t)} - \dist y{\gamma(t)},
		\quad \forall x,y \in X.
	\end{equation*}
	If $\gamma \colon \R \to X$ is a bi-infinite geodesic, we also denote by $c_\gamma$ the Busemann cocycle associated to the reverse path $\bar \gamma$, that is such that $\bar \gamma(t) = \gamma(-t)$, for all $t \in \R$.
	We think of $\gamma$ as a bi-infinite geodesic going from $c_\gamma$ to $c'_\gamma$.
\end{exam}

By construction, every cocycle $c \in \bar X$ is $1$-Lipschitz, or equivalently $\abs{c(x,x')} \leq \dist x{x'}$, for every $x,x' \in X$.
Recall that $X$ is proper.
It is a consequence of the Azerla-Ascoli theorem, that the horocompactification $\bar X$ is indeed a compact set.
We denote by $\mathfrak B$ the Borel $\sigma$-algebra on $\bar  X$.
Any isometry $g$ of $X$ extends to a homeomorphism of $\bar X$ as follows: for every $c \in \bar X$, the cocycle $gc$ is defined by 
\begin{equation*}
	gc(x,y) = c(g^{-1}x,g^{-1}y), \quad \forall x,y \in X.
\end{equation*}

\begin{defi}
	Let $c \in \bar X$.
	A \emph{gradient arc} for $c$ is a path $\nu \colon I \to X$ parametrized by arc length such that $c(\nu(s),\nu(t)) = t - s$, for every $s,t \in I$.
	If $I$ has the form $[a, \infty)$ or $I = \R$, we use the word \emph{gradient ray} and \emph{gradient line} respectively.
\end{defi}

Given a gradient arc $\nu \colon I \to X$ for $c$, we say that $\nu$ \emph{ends at $c$} if 
\begin{itemize}
	\setlength\itemsep{0em}
	\item either $c = \iota(z)$ for some $z \in X$ and $\nu$ ends at $z$,
	\item or $c \in \partial X$ and $I$ is not bounded from above.
\end{itemize}

\begin{defi}
\label{res: complete gradient arc}
	Let $c, c' \in \bar X$.
	A \emph{complete gradient arc} from $c$ to $c'$ is a gradient arc $\nu \colon I \to X$ for $c'$, ending at $c'$ such that the reverse path $\bar \nu$ given by $\bar \nu(t) = \nu(-t)$ is a gradient arc for $c$, ending at $c$.
\end{defi}

\noindent
If both $c$ and $c'$ belong to $X$, then any complete gradient arc from $c$ to $c'$ is just a geodesic from $c$ to $c'$.
If $c$ or $c'$ belongs to $X$, a complete gradient arc from $c$ to $c'$ always exists, see \cite[Lemma~3.4]{Coulon:2022tu}.
We will discuss later the existence of complete gradient arcs joining two points in the horoboundary, see \autoref{res: existence bi-infinite gradient line}.

Note that existence of gradient lines implies that $\iota(X)$ is open in $\bar X$.
Indeed, let $z\in X$ and set $b_z = \iota(z)$.
Fix $\epsilon\in \R_+^*$. 
Denote by $K$ the closed ball of radius $\epsilon$ centered at $z$.
The set $U$ of elements $c \in \bar X$ such that $\norm[K]{b_z -c} < \epsilon$ is an open neighborhood of $b_z$ in $\bar X$.
Consider now $c \in \partial X$ and choose a gradient ray $\nu \colon \R_+ \to X$ for $c$ starting at $z$.
By construction,
\begin{equation*}
	b_z(\nu(0), \nu(t)) = - t
	\quad \text{and} \quad 
	c(\nu(0), \nu(t)) = t, \quad \text{for all}\ t \in \R_+, 
\end{equation*}
so that $\norm[K]{b_z - c} \geq 2\epsilon$.
Consequently $U \cap \partial X = \emptyset$, hence $U$ is contained in $\iota(X)$ which completes the proof of our claim.

%
\subsection{Reduced horoboundary}
\label{sec: reduced horoboundary}
%

Given a cocycle $c \in C^*(X)$ we write $\norm[\infty]c$ for its uniform norm, i.e.
\begin{equation*}
	\norm[\infty]c = \sup_{x,x' \in X} \abs{c(x,x')}.
\end{equation*}
Similarly, if $K \subset X$ is compact, then $\norm[K]c$ is the uniform norm of $c$ restricted to $K \times K$.
Note that $\norm[\infty]c$ can be infinite.
We endow $\bar X$ with a binary relation: two cocycles $c,c' \in \bar X$ are \emph{equivalent}, and we write $c \sim c'$, if one of the following holds:
\begin{itemize}
	\item either $c,c'$ lie in the image of $\iota \colon X \to C^*(X)$ and $c = c'$,
	\item or $c,c' \in \partial X$ and $\norm[\infty] {c-c'} < \infty$.
\end{itemize}
If $c \in \bar X$, we write $[c]$ for its equivalence class.

Given a subset $B \subset \bar X$, the \emph{saturation} of $B$, denoted by $B^+$, is the union of all equivalence classes intersecting $B$.
We say that $B$ is \emph{saturated} if $B^+ = B$.
Note that the collection of saturated subsets is closed under complement and (uncountable) union and intersection.
The \emph{reduced algebra}, denoted by $\mathfrak B^+$, is the sub-$\sigma$-algebra  of $\mathfrak B$ which consists of all saturated Borel subsets.
The saturation of any closed set belongs to $\mathfrak B^+$, see \cite[Lemma~3.7]{Coulon:2022tu}.

\begin{rema}
	In this article we often make the following abuse of notations:
	If $(\Omega, \mathfrak A)$ is a measurable space, and $\Omega_0$ a subset of $\Omega$, we write $(\Omega_0, \mathfrak A)$ for the measurable space $\Omega_0$ endowed with the $\sigma$-algebra $\mathfrak A_0 = \set {A \cap \Omega_0}{A \in \mathfrak A}$.
\end{rema}

\begin{defi}
\label{def: reduced horoboundary}
	The measurable space $(\partial X, \mathfrak B^+)$ is called the \emph{reduced horoboundary} of the space  $X$.
 \end{defi}

%
\subsection{Gromov product.}
\label{sec: gromov product}
%

Given $c,c' \in \bar X$, and $x \in X$, we define the following Gromov product 
\begin{equation}
\label{eqn: extension Gromov product}
	\gro c{c'}x =  \frac 12 \sup_{z \in X} \left[c(x,z) + c'(x,z)\right],
\end{equation}
taking values in $\R_+ \cup \{ \infty\}$.
The action of $G$ on $\bar X$ preserves Gromov products.
Since cocycles are $1$-Lispchitz, if $c' = \iota(y)$ for some $y \in X$, then
\begin{equation*}
\label{eqn: gromov product}
	\gro cyx = \frac 12 \left[c(x,y) + \dist  xy\right].
\end{equation*}
Moreover $0 \leq \gro cyx \leq \dist xy$.
In particular, if both $c$ and $c'$ belong to $X$, we recover the usual definition of the Gromov product.
If $c,c' \in \partial X$ are equivalent, i.e. $\norm[\infty]{c-c'} < \infty$, then $\gro c{c'}x = \infty$.
The converse does not always hold true though.
Note that for every, $c,c' \in \bar X$, for every $x,x' \in X$, 
\begin{equation}
\label{eqn: gromov product - conf}
	\gro c{c'}x = \frac 12 \left[ c(x,x') + c'(x,x')\right] + \gro c{c'}{x'}.
\end{equation}
In addition, for all $c,c' \in \bar X$, for all $x,x',y,y' \in X$, we have,
\begin{equation}
\label{eqn: gromov product - lip}
	\abs{ \gro cyx - \gro {c'}{y'}{x'}} \leq  \frac 12 \norm[K]{c - c'}+  \dist x{x'} + \dist y{y'},
\end{equation}
where $K\subset X$ is any compact set containing $x$, $x'$, $y$,  and $y'$.

\begin{lemm}
\label{res: existence bi-infinite gradient line}
	Let $c,c' \in \partial X$.
	Assume that there is $z \in X$ such that $\gro c{c'}z = 0$.
	Then there exists a complete gradient line from $c$ to $c'$ going through $z$. 
	\end{lemm}

\begin{rema}
\label{rem: gromov product achieved}
	In view of (\ref{eqn: gromov product - conf}) the assumption is equivalent to the following fact: 
	there is $z \in X$, such that for some (hence any) $x \in X$.
	\begin{equation*}
		\gro c{c'}x = \frac 12 \left[ c(x,z)+ c'(x,z)\right]. \qedhere
	\end{equation*}
\end{rema}

\begin{proof}
	Denote by $\gamma \colon \R_+ \to X$ (\resp $\gamma' \colon \R_+ \to X$) a complete gradient arc from $z$ to $c$ (\resp to $c'$).
	Such lines exist by \cite[Lemma~3.4]{Coulon:2022tu}.
	Build a path $\nu \colon \R \to X$ as follows 
	\begin{equation*}
		\begin{cases}
			\gamma'(t) & \text{if} \ t \geq 0, \\
			\gamma(-t) & \text{if} \ t \leq 0.
		\end{cases}
	\end{equation*}
	We claim that $c'(\gamma(t),z) = t$, for every $t \in \R_+$.
	Since $\gamma$ is a gradient line for $c$, we observe that
	\begin{equation*}
		t + c'(z, \gamma(t)) 
		\leq c(z, \gamma(t)) + c'(z, \gamma(t))
		\leq 2 \gro c{c'}z
		\leq 0.
	\end{equation*}
	Hence  $c'(\gamma(t),z) \geq t$.
	Recall that $c'$ is $1$-Lipchitz,  thus it forces $c'(\gamma(t),z) = t$, as we announced.
	Consequently $c'(\nu(t),z) = -t$, for every $t \in \R$.
	It follows from the cocycle property that $c'(\nu(s),\nu(t)) = t-s$, for every $s,t \in \R$.
	In other words, $\nu$ is a gradient line for $c'$.
	By symmetry we get that the reverse path $t \mapsto \nu(-t)$ is also a gradient line for $c$.
\end{proof}

\medskip
The map 
\begin{equation*}
	\begin{array}{ccc}
		\bar X \times X \times X & \to & \R_+ \\
		(c,y, x) & \mapsto & \gro cyx
	\end{array}
\end{equation*}
is continuous.
However its extension to $\bar X \times \bar X \times X \to \R_+ \cup\{\infty\}$ may not be continuous anymore.
We claim that it is lower semi-continuous though, hence Borel.
Indeed for every $N \in \R_+$, the set 
\begin{equation*}
	A_N = \set{(c,c',x) \in \bar X \times \bar X \times X}{\gro c{c'}x \leq N}.
\end{equation*}
can be written
\begin{equation*}
	A_N = \bigcap_{z \in X}\set{(c,c',x) \in \bar X \times \bar X \times X}{ c(x,z) + c'(x,z)\leq 2N}.
\end{equation*}
As an intersection of closed subsets, $A_N$ is closed as well, whence our claim.

\begin{prop}
\label{res: continuous extension gromov product}
	Let $(c,c') \in \bar X \times \bar X$.
	Assume that there exist a neighborhood $U \subset \bar X \times \bar X$ of $(c,c')$ and a compact subset $K \subset X$ with the following property: for every pair $(y,y') \in (X\times X) \cap U$, there is a geodesic from $y$ to $y'$ intersecting $K$.
	Then, there is a point $z_0 \in K$ such that the following holds.
	\begin{enumerate}
		\item \label{enu: continuous extension gromov product - gromov product}
		$\gro c{c'}{z_0} = 0$.
		\item \label{enu: continuous extension gromov product - gradient line}
		There is a complete gradient line from $c$ to $c'$ going through $z_0$.
		\item \label{enu: continuous extension gromov product - continuity}
		For every $x \in X$, the map $\bar X \times \bar X  \times X \to \R_+ \cup\{ \infty\}$ sending $(b,b',z)$ to $\gro b{b'}z$ is continuous at $(c,c',x)$.
	\end{enumerate}
\end{prop}

\begin{proof}
	Fix $x \in X$.
	For simplicity we let 
	\begin{equation*}
		\ell = \sup_{z \in K}\ \frac 12 \left[c(x,z) +  c'(x,z) \right] .
	\end{equation*}
	The set $K$ is compact, while $c$ and $c'$ are continuous.
	Hence the upper bound $\ell$ is achieved at some point $z_0 \in K$.
	Let $\epsilon \in \R_+^*$.
	Recall that $\bar X$ is endowed with the topology of uniform convergence on every compact subset.
	There is an open set $V \subset U$ of $(c, c')$ such that for every $(b,b') \in V$, the cocycle $b$ (\resp $b'$) restricted to $\{x \} \times K$ differs from $c$ (\resp $c'$) by at most $\epsilon$.
	
	Let $z \in X$ and set $K_0 = K \cup \{z\}$.
	Let $(b,b') \in V$.
	Since $V$ is open, there are $(y,y') \in (X \times X) \cap V$ whose corresponding cocycles $\tilde b$ and $\tilde b'$, when restricted to $\{ x \} \times K_0$, differ from $b$ and $b'$ respectively by at most $\epsilon$.
	By construction $(y,y') \in U$.
	Thus there is a point $p \in K$ on a geodesic from $y$ to $y'$.
	Using the triangle inequality we observe that 
	\begin{equation*}
		 \frac 12 \left[\tilde b(x,z) + \tilde b'(x,z) \right] 
		 \leq \gro y{y'}x
		\leq  \frac 12 \left[\tilde b(x,p) + \tilde b'(x,p) \right].
	\end{equation*}
	According to our choices, $b$ and $\tilde b$ differs by at most $\epsilon$ on $(x,z)$ while $c$ and $\tilde b$ differs by at most $\epsilon$ on $(x,p)$.
	A similar statement holds for $b'$, $\tilde b'$ and $c'$.
	Consequenlty, the previous inequality yields
	\begin{equation*}
		 \frac 12 \left[ b(x,z) + b'(x,z) \right] - \epsilon
		\leq  \frac 12 \left[c(x,p) +  c'(x,p) \right]  + \epsilon
	\end{equation*}
	Consequently
	\begin{equation*}
		 \frac 12 \left[ b(x,z) + b'(x,z) \right] - 2 \epsilon
		 \leq \ell
		 \leq \gro c{c'}x.
	\end{equation*}
	This inequality holds for every $z \in X$ so that 
	\begin{equation*}
		 \gro b{b'}x - 2\epsilon
		 \leq \ell		 
		 \leq \gro c{c'}x.
	\end{equation*}
	Note that $\ell$ does not depend on the choice of $b$ and $b'$, thus $(b,b') \mapsto \gro b{b'}x$ is uniformly bounded on $V$.
	The previous inequality shows in particular that,
	\begin{equation*}
		\gro c{c'}x 
		= \ell
		= \frac 12 \left[ c(x,z_0) + c'(x,z_0)\right].
	\end{equation*}
	As we noticed in \autoref{rem: gromov product achieved}, it means in particular that $\gro c{c'}{z_0} = 0$,  which prove \ref{enu: continuous extension gromov product - gromov product}.
	Thus there is complete gradient line from $c$ to $c'$ (\autoref{res: existence bi-infinite gradient line}) whence \ref{enu: continuous extension gromov product - gradient line}.

	We are left to prove \ref{enu: continuous extension gromov product - continuity}.
	We use the same notations as above.
	Recall that $b$ and $c$ (\resp $b'$ and $c'$) differ  by at most $\epsilon$ on $\{x\} \times K$.
	Hence
	\begin{equation*}
		\gro c{c'}x \leq \frac 12 \left[c(x,z_0) + c'(x,z_0) \right] \leq \frac 12 \left[b(x,z_0) + b'(x,z_0) \right]  + \epsilon \leq \gro b{b'}x + \epsilon.
	\end{equation*}
	To summarize we have proved that for every $\epsilon \in \R_+^*$ there is a neighborhood $V$ of $(c,c')$, such that for every $(b,b') \in V$, we have
	\begin{equation*}
		\abs{ \gro b{b'}x - \gro c{c'}x} \leq 2\epsilon.
	\end{equation*}
	Recall that cocycles in $\bar X$ are $1$-Lipschitz.
	Using (\ref{eqn: gromov product - conf}) we can vary the base point.
	More precisely for every $y \in B(x,\epsilon)$, the previous inequality yields
	\begin{equation*}
		\abs{ \gro b{b'}y - \gro c{c'}x} \leq 3\epsilon,
	\end{equation*}
	whence the result.
\end{proof}

%
\subsection{Extended boundary of a contracting set}
%

\begin{defi}
	Let $Y$ be a closed subset of $X$.
	Let $c \in \bar X$.
	A \emph{projection} of $c$ on $Y$ is a point $q \in Y$ such that for every $y \in Y$, we have $c(q,y) \leq 0$.
\end{defi}

Given $z \in X$, the projection of $b = \iota(z)$ on $Y$ coincides with the definition of the nearest-point projection of $z$.
If $c$ belongs to $\partial X$, a projection of $c$ on $Y$ may exist or not.
This leads to the following definition.

\begin{defi}
	Let $Y$ be a closed subset of $X$.
	The \emph{extended boundary} of $Y$, denoted by $\partial ^+ Y$ is the set of all cocycles $c \in \partial X$ for which there is no projection of $c$ on $Y$.
\end{defi}

If $Y$ is contracting, then $\partial^+Y$ is saturated (as the notation suggests) \cite[Proposition~3.9]{Coulon:2022tu}.
The next statements extend \autoref{res: proj contracting set} for a gradient ray joining two points in $\bar X$.

\begin{lemm}[Coulon {\cite[Corollary~3.13]{Coulon:2022tu}}]
\label{res: proj cocycle contracting set}
	Let $\alpha \in \R_+^*$.
	Let $Y$ an $\alpha$-contracting set.
	Let $x \in X$ and $c \in \bar X \setminus \partial^+Y$.
	Let $p$ and $q$ be respective projections of $x$ and $c$ on $Y$.
	Let $\nu$ be a complete gradient arc from $x$ to $c$.
	If $d(x,Y) < \alpha$ or $\dist pq > 4\alpha$, then the following holds
	\begin{itemize}
		\item $d(\nu, Y)< \alpha$;
		\item  the entry point (\resp exit point) of $\nu$ in $\mathcal N_\alpha(Y)$ is $2\alpha$-closed (\resp $5\alpha$-closed) to $p$ (\resp $q$);
		\item $\dist xp + \dist pq - 14\alpha \leq c(x,q) \leq \dist xp + \dist pq$.
	\end{itemize}
\end{lemm}

\begin{lemm}
\label{res: proj two cocycle contracting set}
	Let $\alpha \in \R_+^*$.
	Let $Y$ an $\alpha$-contracting set and $c,c' \in \bar X \setminus \partial^+Y$.
	Let $q$ and $q'$ be respective projections of $c$ and $c'$ on $Y$.
	If $\dist q{q'} > 7\alpha$, then the following holds.
	\begin{enumerate}
		\item \label{enu: proj two cocycle contracting set - existence}
		There is a complete gradient arc $\nu_0 \colon I \to X$ from $c$ to $c'$ such that for some (hence every) point $z$ on $\nu_0$, we have $\gro c{c'}z = 0$.
		\item \label{enu: proj two cocycle contracting set - entry/exit point}
		For any complete gradient arc $\nu \colon I \to X$ from $c$ to $c'$, we have $d(\nu, \gamma) < \alpha$.
		Moreover the entry point (\resp exit point) of $\nu$ in $\mathcal N_\alpha(\gamma)$ is $5\alpha$-close to $q$ (\resp $q'$).
	\end{enumerate}
\end{lemm}

\begin{proof}
	We start with the existence of a complete gradient line.
	Denote by $K$ the closed ball of radius $4\alpha$ centered at $q$.
	In view of \autoref{res: continuous extension gromov product} it suffices to prove that there is an neighborhood $U$ of $(c,c')$ in $\bar X \times\bar X$ such that for every $(z,z') \in (X \times X) \cap U$, any geodesic from $z$ to $z'$ intersects $K$.
	Assume on the contrary that this fact does not hold.
	We can find two sequences $(z_n)$ and $(z'_n)$ of points in $X$ converging respectively to $c$ and $c'$ and a geodesic from $z_n$ to $z'_n$ that does not cross $K$.
	Denote by $q_n$ and $q'_n$ respective projections of $z_n$ and $z'_n$ on $Y$.
	Up to passing to a subsequence, we can assume that $(q_n)$ and $(q'_n)$ respectively converges to $q_*,q'_* \in Y$ which are also respective projections of $c$ and $c'$ on $Y$, see \cite[Proposition~3.11]{Coulon:2022tu}.
	According to \cite[Corollary~3.12]{Coulon:2022tu} we have  $\dist q{q^*} \leq \alpha$ and $\dist {q'}{q'_*} \leq \alpha$.
	Consequently $\dist {q_n}{q'_n} > \alpha$, provided $n$ is sufficiently large.
	It follows from \autoref{res: proj contracting set} that $q_n$ (and $q'_n$) are $2\alpha$-close to any geodesic from $z_n$ to $z'_n$.
	Since $\dist q{q_*} \leq \alpha$, this contradicts our assumption, provided $n$ is sufficiently large.
	
	We now prove \ref{enu: proj two cocycle contracting set - entry/exit point}.
	Let $\nu \colon I \to X$ be an arbitrary complete gradient arc from $c$ to $c'$.
	Let 
	\begin{equation*}
		T = \inf\set{t \in I}{d(\nu(t), Y) \leq \alpha}
		\quad \text{and} \quad 
		T' = \sup\set{t \in I}{d(\nu(t), Y) \leq \alpha}
	\end{equation*}
	with the convention that $T = \sup I$ and $T' = \inf I$, whenever $\nu$ does not intersect $\mathcal N_\alpha(\gamma)$.
	Set 
	\begin{equation*}
		Q = \pi_Y\left(\set{\nu(t)}{t \leq T}\right)
		\quad \text{and} \quad
		Q' = \pi_Y\left(\set{\nu(t)}{t \geq T'}\right)
	\end{equation*}
	According to \cite[Proposition~3.11]{Coulon:2022tu}, the point $q_*$ (\resp $q'_*$) lies in the $\alpha$-neighborhood of $Q$ (\resp $Q'$) while $Q$ and $Q'$ have diameter at most $\alpha$.
	It follows from the triangle inequality combined with our assumption that
	\begin{equation*}
		d(Q,Q') 
		\geq \dist {q_*}{q'_*} - 4\alpha
		\geq \dist q{q'} - 6\alpha
		> \alpha.
	\end{equation*}
	Since $\gamma$ is contracting, $d(\nu, Y) < \alpha$.
	In particular, $\nu(T)$ and $\nu(T')$ are the entry and exit point of $\nu$ in $Y$.
	Moreover $d(\nu(T),q) \leq 5\alpha$ and $d(\nu(T'),q') \leq 5\alpha$, which completes the proof of \ref{enu: proj two cocycle contracting set - entry/exit point}.
\end{proof}


\begin{lemm}
\label{res: proj equivalent cocycles}
	Let $\alpha \in \R^*_+$.
	Let $Y$ be an $\alpha$-contracting set.
	Let $c,c' \in \bar X \setminus \partial^+Y$.
	Let $q$ and $q'$ be respective projections of $c$ and $c'$ on $Y$.
	If  $\norm[\infty]{c - c'} < \infty$, then $\dist q{q'} \leq 4\alpha$.
\end{lemm}

\begin{proof}
	Let $\nu  \colon \R_+ \to X$ be gradient line for $c$ starting at $q$.
	Note that $q$ is a projection of $\nu(t)$ on $Y$ for any $t \in \R_+$  \cite[Lemma~3.9]{Coulon:2022tu}.
	Assume now that contrary to our claim $\dist q{q'} > 4\alpha$.
	Let $t \in \R_+$.
	It follows from \autoref{res: proj cocycle contracting set} applied with $x = \nu(t)$ and $c'$ that 
	\begin{equation*}
		c'(\nu(t), q) \geq \dist q{\nu(t)} - 4\alpha \geq t - 4\alpha.
	\end{equation*}
	However since $\nu$ is a gradient line for $c$, we also have $c(\nu(t),q) = -t$.
	These two estimates hold for every $t \in \R_+$ and contradict the fact that $\norm[\infty]{c - c'} < \infty$.
\end{proof}

\begin{prop}
\label{res: proj orbit subgroup without contracting}
	Let $g \in G$ be a contracting element and $Y$ an axis of $g$.
	Let $c \in \bar X$.
	Let $H \subset G$ be a subgroup that does not contain a contracting element.
	Assume that the intersection $Hc \cap \partial^+Y$ is empty.
	Then the projection on $Y$ of $Hc$ is bounded.
\end{prop}

\begin{proof}
	We fix $\alpha \in \R_+^*$ such that $Y$ is $\alpha$-contracting and $\diam {\pi_Y(uY)} \leq \alpha$, for every $u \in G \setminus E(g)$.
	Denote by $Z$ the set 
	\begin{equation*}
		Z = \bigcup_{h \in H \setminus E(g)} hY.
	\end{equation*}
	According to \autoref{res: suff condition contracting} the projection of $Z$ on $Y$ has bounded diameter.
	Without loss of generality, we can assume that there is $h_0 \in H$ such that $d(\pi_Y(h_0c), \pi_Y(Z)) > 4\alpha$.
	Up to replacing $c$ by $h_0c$, we can also assume that $h_0 = 1$.
	Let $h \in H \setminus E(g)$.
	According to our previous choices
	\begin{equation*}
		d(\pi_{hY}(hc), \pi_{hY}(Y)) 
		\geq d(\pi_Y(c), \pi_Y(h^{-1}Y)) 
		\geq d(\pi_Y(c), \pi_Y(Z))
		> 4\alpha.
	\end{equation*}	
	One proves using \autoref{res: proj cocycle contracting set} that any projection of $hc$ on $Y$ lies in the $15\alpha$-neighborhood of $\pi_Y(hY)$ hence of $\pi_Y(Z)$, see \autoref{fig: proj orbit subgroup without contracting}.
	\begin{figure}[htbp]
		\begin{center}
			\labellist
			\small\hair 2pt
			 \pinlabel {$c$} [ ] at -5 147
			 \pinlabel {$hc$} [ ] at 70 212
			 \pinlabel {$\pi_{h\gamma}\left(hc\right)$} [l] at 225 263
			 \pinlabel {$\pi_{\gamma}\left(c\right)$} [t] at 82 8
			 \pinlabel {$hY$} [ ] at 529 150
			 \pinlabel {$Y$} [ ] at 543 49
			 \pinlabel {$\pi_{hY}\left(Y\right)$} [b] at 420 130
			 \pinlabel {$\pi_{Y}\left(hY\right)$} [ ] at 377 35
			\endlabellist
			\includegraphics[width=0.8\linewidth]{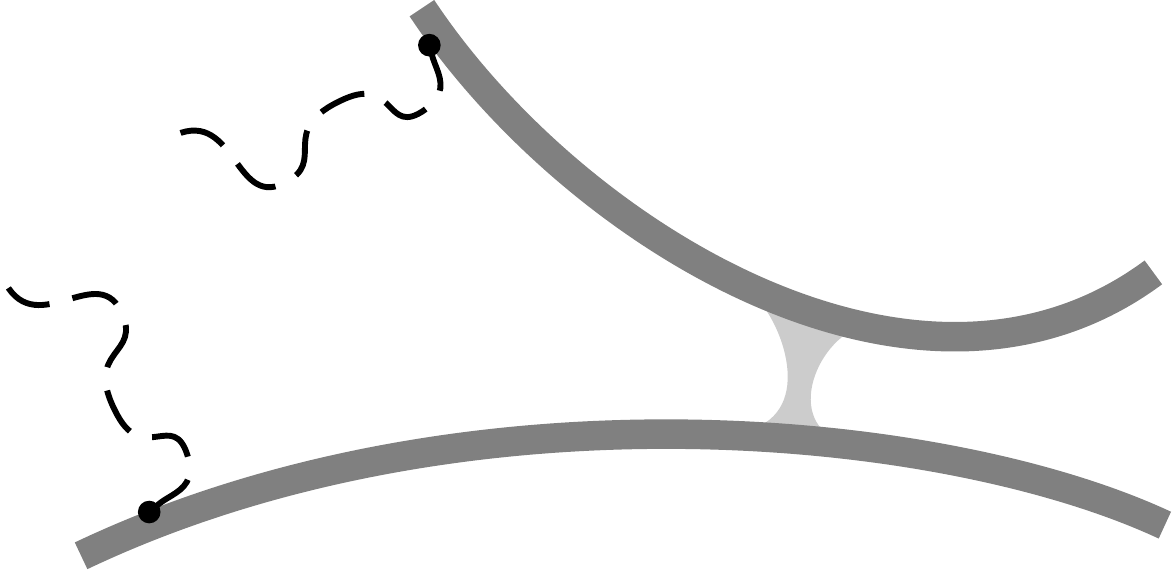}
			\vskip 5mm
			\caption{Projections of $c$ and $hc$ on $Y$}
			\label{fig: proj orbit subgroup without contracting}
		\end{center}
	\end{figure}
	Since the later set is bounded, we get that the projection on $Y$ of $(H \setminus E(g))c$ is bounded.
	Recall that $H$ does not contain any contracting element.
	Consequently $H \cap E(g)$ is finite.
	Therefore the projection of $Hc$ on $Y$ has bounded diameter at well.
\end{proof}

%
\subsection{Shadows}
\label{sec: shadows}
%

We now recall the definition of shadows which will be an essential tool to study conformal measures supported on the horocompactification.

\begin{defi}
	Let $x,y \in X$.
	Let $r \in \R_+$.
	The \emph{$r$-shadow} of $y$ seen from $x$, is the set 
	\begin{equation*}
		\mathcal O_x(y,r) = \set{ c \in \bar X}{ \gro xcy \leq r}.
	\end{equation*}
\end{defi}

By construction, $\mathcal O_x(y,r)$ is a closed subset of $\bar X$.
It follows from (\ref{eqn: gromov product - lip}) that for every $x,x',y,y' \in X$ and $r \in \R_+$,
\begin{equation}
\label{eqn: inclusion shadows triangle inequality}
	\mathcal O_x(y,r) \subset \mathcal O_{x'}(y',r'), \quad \text{where} \quad r' = r + \dist x{x'} + \dist y{y'}.
\end{equation}

\begin{defi}
\label{def: contracting tail}
	Let $x,y \in X$.
	Let $\alpha \in \R_+^*$ and $L \in \R_+$.
	We say that the pair $(x,y)$ has an \emph{$(\alpha, L)$-contracting tail} if there is an $\alpha$-contracting geodesic $\gamma$ ending at $y$ and a projection $p$ of $x$ on $\gamma$ satisfying $\dist py \geq L$.
	The path $\gamma$ is called a \emph{tail} of $(x,y)$ (the parameters $\alpha$ and $L$ should be clear from the context).
\end{defi}

\begin{lemm}
\label{res: wysiwyg shadow}
	Let $\alpha \in \R_+^*$ and  $r, L \in \R_+$ with $L > r + 13\alpha$.
	Let $x,y \in X$.
	Assume $(x,y)$ has an $(\alpha, L)$-contracting tail, say $\gamma$.
	Let $p$ be a projection of $x$ on $\gamma$.
	Let $c\in \mathcal O_x(y,r)$.
	Let $q$ be a projection of $c$ onto $\gamma$.	
	Then the following holds.
	\begin{enumerate}
		\item $\dist yq \leq r+7\alpha$.
		\item 
		For any complete gradient arc $\nu \colon I \to X$ from $x$ to $c$, we have $d(\nu, \gamma) < \alpha$.
		Moreover the entry point (\resp exit point) of $\nu$ in $\mathcal N_\alpha(\gamma)$ is $2\alpha$-close to $p$ (\resp $5\alpha$-close to $q$).
		\item \label{enu: wysiwyg shadow - double gromov product}  
		$\gro xcz \leq 4\alpha$, for every point $z$ on $\gamma$ with $r + 7\alpha \leq \dist zy \leq L$.
		\item \label{enu: wysiwyg shadow - simple gromov product}
		$\gro {z'}cz \leq 2 \alpha$, for every points $z,z'$ on $\gamma$ with $r + 7\alpha \leq \dist zy \leq \dist {z'}y$.
	\end{enumerate}
\end{lemm}

\begin{proof}
	The first two points correspond to \cite[Lemma~4.14]{Coulon:2022tu}.
	By definition of contracting tail, there is a projection $p'$ of $x$ on $\gamma$ such that $\dist {p'}y \geq L$.
	Let $(b_n)$ be a sequence of cocycles in $\iota(X)$ converging to $c$.
	For every $n \in \N$, we write $q_n$ for a projection of $b_n$ on $\gamma$.
	Up to passing to a subsequence, we can assume that $(q_n)$ converges to a point $q^*$ which is a projection of $c$ on $\gamma$ \cite[Lemma~3.11]{Coulon:2022tu}.
	Hence $\dist y{q_*} \leq r + 7\alpha$.
	Consider now a point $z$ on $\gamma$ such that $r + 7\alpha < \dist zy \leq L$.
	For all but finitely many $n \in \N$, $z$ lies on $\gamma$ between $p'$ and $q_n$, while $\dist {p'}{q_n} > \alpha$. 
	Hence $\gro x{b_n}z \leq 4\alpha$ (\autoref{res: proj contracting set}).
	Passing to the limit, we get $\gro xcz \leq 4\alpha$.
	The map $X \to \R$ sending $z$ to $\gro xcz$ is continuous, hence \ref{enu: wysiwyg shadow - double gromov product} is proved.
	Point~\ref{enu: wysiwyg shadow - simple gromov product} is obtained in the same way.
\end{proof}

\begin{lemm}
\label{res: saturation shadow}
	Let $\alpha \in \R_+^*$ and  $r, L \in \R_+$ with $L > r + 15\alpha$.
	Let $x,y \in X$.
	Assume $(x,y)$ has an $(\alpha, L)$-contracting tail.
	Then the saturation $\mathcal O^+_x(y,r)$ of $\mathcal O_x(y,r)$ is a Borel set contained in $\mathcal O_x(y,r + 16\alpha)$.
\end{lemm}

\begin{proof}
	As the saturation of a closed subset, $\mathcal O^+_x(y,r)$ is a Borel set, see \cite[Lemma~3.7]{Coulon:2022tu}.
	Let $\gamma$ be a contracting tail of $(x,y)$.
	By assumption there is a projection $p$ of $x$ on $\gamma$ such that $\dist py \geq L$.
	Let $c,c'\in \partial X$ be two equivalent cocycles with $c \in \mathcal O_x(y,r)$.
	Denote by $q$ and $q'$ respective projections of $c$ and $c'$ on $\gamma$.
	According to \autoref{res: wysiwyg shadow}, $\dist yq \leq r + 7\alpha$.
	Combined with \autoref{res: proj equivalent cocycles} we get $\dist y{q'} \leq r + 11\alpha$.
	Consequently
	\begin{equation*}
		\dist p{q'} \geq \dist py - \dist {q'}y \geq L - (r+11\alpha) > 4\alpha.
	\end{equation*}
	Combining \autoref{res: proj cocycle contracting set} with the triangle inequality we see that 
	\begin{equation*}
		\gro x{c'}y \leq \gro x{c'}{q'} + \dist {q'}y
		\leq r + 16\alpha.
	\end{equation*}
	Hence $c' \in \mathcal O_x(y, r +16\alpha)$.
\end{proof}


\begin{defi}
\label{def: separation}
	Let $r \in \R_+$.
	Let $y,y' \in X$.
	Two subsets $Z,Z' \subset \bar X$ are \emph{$r$-separated by the pair $(y,y')$}, if 
	\begin{equation*}
		Z \subset \mathcal O_{y'}(y,r)
		\quad \text{and} \quad
		Z' \subset \mathcal O_y(y',r)
	\end{equation*}
	If $Z = \{z\}$ and $Z' = \{z'\}$ are reduced to a point, we simply that $z$ and $z'$ are $r$-separated by $(y,y')$.
\end{defi}

Recall that shadows are closed subsets of $\bar X$.
Hence being $r$-separated by $(y,y')$ is a closed condition on $\bar X \times \bar X$.

\begin{prop}
\label{res: bi-gradient line}
	Let $\alpha \in \R_+^*$ and $r \in \R$.
	Let $y,y' \in X$ such that $\dist y{y'} > 2r + 21\alpha$.
	Assume that there exists an $\alpha$-contracting geodesic $\gamma \colon \intval a{a'} \to X$ joining $y$ to $y'$.
	Let $c,c' \in \bar X$.
	Let $q$ and $q'$ be respective projections of $c$ and $c'$ on $\gamma$.
	If $c$ and $c'$ are $r$-separated by $(y,y')$ then the following holds.
	\begin{enumerate}
		\item \label{enu: bi-gradient line - proj}
		$\max\{ \dist qy, \dist{q'}{y'} \} \leq r + 7\alpha$.
		\item \label{enu: bi-gradient line - existence}
		There is a complete gradient arc $\nu_0 \colon I \to X$ from $c$ to $c'$ such that for some (hence every) point $z$ on $\nu_0$, we have $\gro c{c'}z = 0$.
		\item \label{enu: bi-gradient line - entry/exit point}
		For any complete gradient arc $\nu \colon I \to X$ from $c$ to $c'$, we have $d(\nu, \gamma) < \alpha$.
		Moreover the entry point (\resp exit point) of $\nu$ in $\mathcal N_\alpha(\gamma)$ is $5\alpha$-close to $q$ (\resp $q'$).
		\item \label{enu: bi-gradient line - simple gromov product}
		$\gro c{\gamma(s')}{\gamma(s)} \leq 2\alpha$, for every $s,s' \in \intval{a+ (r + 7\alpha)}{a'}$ with  $s' \geq s$.
		\item \label{enu: bi-gradient line - double gromov product}
		$\gro c{c'}{\gamma(s)} \leq 4\alpha$, for every $s \in \intval{a+ (r + 7\alpha)}{a'-(r +7\alpha)}$.
		\item \label{enu: bi-gradient line - neighborhood}
		 The path $\gamma$ restricted to $\intval{a + (r + 13\alpha)}{a'- (r + 13\alpha)}$ lies in the $5\alpha$-neighborhood of any complete gradient arc from $c$ to $c'$.
	\end{enumerate}
\end{prop}

\begin{proof}
	Note that both $(y,y')$ and $(y',y)$ have an $(\alpha, L)$-contracting tail.
	Hence \ref{enu: bi-gradient line - proj} is just a consequence of \autoref{res: wysiwyg shadow}.
	Moreover 
	\begin{equation*}
		\dist q{q'} \geq \dist y{y'} - 2r - 14\alpha > 7\alpha.
	\end{equation*}
	Hence \ref{enu: bi-gradient line - existence} and \ref{enu: bi-gradient line - entry/exit point} follow from \autoref{res: proj two cocycle contracting set}.
	Point~\ref{enu: bi-gradient line - simple gromov product} is a particular case of \autoref{res: wysiwyg shadow} while \ref{enu: bi-gradient line - double gromov product} is proved exactly in the same way.
	We are left to prove \ref{enu: bi-gradient line - neighborhood}.
	For simplicity let $b = a + (r + 13\alpha)$ and $b' = a' - (r + 13\alpha)$.
	Let $\nu$ be a complete gradient arc from $c$ to $c'$.
	Denote by $p = \nu(t)$ and $p' = \nu(t')$ the respective entry and exit points of $\nu$ in $\mathcal N_\alpha(\gamma)$.
	In view of \ref{enu: bi-gradient line - proj} and \ref{enu: bi-gradient line - entry/exit point}, $p$ and $p'$ are respectively $\alpha$-close to $m = \gamma(s)$ and $m' = \gamma(s')$ for some $(s,s') \in \intval ab \times \intval {b'}{a'}$.
	Thus $\nu$ restricted to $\intval t{t'}$ lies in the $(5\alpha/2)$-neighborhood of $\gamma$.
	It follows then from \cite[Lemma~5.1]{Coulon:2022tu} that $\gamma$ restricted to $\intval s{s'}$ lies in the $5\alpha$-neighborhood of $\nu$, whence the result.
\end{proof}


%
\subsection{Conformal densities}
\label{sec: conformal densities}
%

\begin{defi}[Density]
\label{def: density}
	Let $G$ be a group acting properly by isometries on $X$.
	Fix a base point $o \in X$.
	Let $\omega \in \R_+$.
	A \emph{density} is a collection $\nu = (\nu_x)$ of positive finite measures on $(\bar X, \mathfrak B)$ indexed by $X$ such that  $\nu_x \ll \nu_y$, for every $x,y \in X,$ and normalized by $\norm{\nu_o} = 1$.
	Such a density is 
	\begin{enumerate}
		\item  \emph{$G$-invariant}, if $g_\ast \nu_x = \nu_{gx}$, for every $g \in G$ and $x \in X$;
		\item \emph{$\omega$-conformal}, if for every $x,y \in X$,
		\begin{equation*}
			\frac{d\nu_x}{d\nu_y} (c) = e^{-\omega c(x,y)}, \quad \nu_y-\text{a.s.}
		\end{equation*}
	\end{enumerate}
\end{defi}

Existence of $G$-invariant, $\omega_G$-conformal densities supported on $\partial X$ follows from the Patterson construction, see \cite[Proposition~4.3]{Coulon:2022tu}.
The Shadow Lemma stated below is a key tool in the study of invariant, conformal densities, see \cite[Corollary~4.10]{Coulon:2022tu}.

\begin{prop}[Shadow Lemma]
\label{res: shadow lemma}
	Assume that $G$ is not virtually cyclic and has a contracting element.
	Let $\omega \in \R_+$.
	There exist $\epsilon, r_0  \in \R_+^*$ such that for every $r \geq r_0$, for every $G$-invariant, $\omega$-conformal density $\nu = (\nu_x)$, for every $g \in G$, we have
	\begin{equation*}
		\epsilon e^{-\omega \dist o{go}}
		\leq \nu_o\left( \mathcal O_o(go,r)\right) 
		\leq e^{2\omega r}  e^{-\omega \dist o{go}}.
	\end{equation*}
\end{prop}

\paragraph{Atoms.}
Let $\nu = (\nu_x)$ be a $G$-invariant density on $\bar X$.
We collect here a few facts about the atoms of the measure space $(\bar X, \mathfrak B^+, \nu_o)$.
Recall that an element $A \in \mathfrak B^+$ is an \emph{atom} of $\nu_o$ if it has positive measure and for every $B \in \mathfrak B^+$ contained in $A$, either $\nu_o(B) = 0$ or $\nu_o(A \setminus B) = 0$.
For such an atom $A$, its \emph{stabilizer} is the set
\begin{equation*}
	\stab A = \set{h \in G}{\nu_o(A \cap hA) = \nu_o(A)}.
\end{equation*}
As the terminology suggests, it is a subgroup of $G$.
Indeed, if $h_1$ and $h_2$ belong to $\stab A$, then,
\begin{equation*}
	\nu_o(h_1 A \cap h_2A) \geq \nu_o\left((A \cap h_1A) \cap (A \cap h_2A)\right) \geq \nu_o(A) > 0.
\end{equation*}
As $\nu$ is $G$-invariant, its yields $\nu_o(A \cap h_1^{-1}h_2A) > 0$.
However $A$ is an atom, thus $h_1^{-1}h_2 \in \stab A$.

\begin{lemm}
\label{res: subatom with bded proj}
	Let $g \in G$ be a contracting element and $Y$ an axis of $g$.
	There is $D \in \R_+$ with the following property:
	for every atom $A$ of $(\bar X, \mathfrak B^+, \nu_o)$, there is an atom $B \in \mathfrak B^+$ contained in $A \setminus \partial^+Y$ such that the projection of $B$ on $Y$ has diameter at most $D$.
\end{lemm}

\begin{proof}
	Whithout loss of generality, we can take for $Y$ the $\group g$-orbit of some point $x$.
	There is $\alpha\in\R_+^*$ such that $Y$ is $\alpha$-contracting.
	Up to increasing the value of $\alpha$, we can assume that any geodesic joining two points of $Y$ is $\alpha$-contracting; moreover for every $p,q,n \in \Z$, with $p\leq n \leq q$, the point $g^nx$ lies in the $\alpha$-neighborhood of any geodesic from $g^px$ to $g^qx$.
	According to \autoref{res: saturation shadow}, for every $r \in \R_+$, there is $N(r) \in \N$ with the following property: for every $n \in \Z$, with $\abs n \geq N(r)$ we have
	\begin{equation*}
		\mathcal O_o(g^no, r) 
		\subset \mathcal O_o^+(g^no, r) \subset \mathcal O_o(g^no, r + 16\alpha).
	\end{equation*}
	Recall that $\mathcal O_o^+(g^no, r)$ stands for the saturation of $\mathcal O_o(g^no, r)$.
	In view of the Shadow Lemma (\autoref{res: shadow lemma}) there is $r \geq 100\alpha$ and $n \geq N(r)$ such that 
	\begin{equation*}
		\nu_o(\mathcal O_o(g^{-n}o,r+16\alpha)) < \nu_o(A)
		\quad \text{and} \quad
		\nu_o(\mathcal O_o(g^no,r+16\alpha)) < \nu_o(A)
	\end{equation*}
	Since $A$ is an atom, it forces the set
	\begin{equation*}
		A_0 = A \setminus \left(\mathcal O_o^+(g^{-n}o, r) \cup \mathcal O_o^+(g^no, r)\right)
	\end{equation*}
	to have full measure in $A$. 
	In particular, it is also an atom.
	For every $k \in \Z$, we write $A_0(k)$ for the set of all cocycles $c \in A_0$ such that $g^kx$ is a projection of $c$ on $Y$.
	Recall that $r \geq 100\alpha$.
	Since $g$ is contracting, $\partial^+Y$ is contained in $\mathcal O_o(g^{-n}o, r) \cup \mathcal O_o(g^no, r)$, therefore $A_0$ is covered by the union of all $A_0(k)$.
	By construction $A_0(k)$ is a closed subset, hence its saturation $A_0^+(k)$ is Borel, see  \cite[Lemma~3.7]{Coulon:2022tu}.
	As $A_0$ is an atom of $(\bar X, \mathfrak B^+, \nu_o)$, there is $k \in \Z$ such that $B = A^+_0(k)$ has full measure in $A_0$ (hence in $A$) and thus is an atom.
	According to \autoref{res: proj equivalent cocycles}, the projection $p$ of any cocycle in $B$ satisfies $\dist p{g^ko} \leq 4\alpha$.
	Hence $\pi_Y(B)$ has diameter at most $8\alpha$.
\end{proof}

\begin{lemm}
\label{res: no contracting stabilizing atom}
	Let $A$ be an atom of $(\bar X, \mathfrak B^+, \nu_o)$.
	Then $\stab A$ does not contain a contracting element.
\end{lemm}

\begin{proof}
	Let $g \in G$ be a contracting element and $Y$ an axis of $g$.
	According to \autoref{res: subatom with bded proj}, we can find an atom $B \subset A$ and $m \in \N$, such that the projections on $Y$ of $B$ and $g^m B$ are disjoint.
	Hence so are $B$ and $g^mB$.
	Recall that $\nu$ is $G$-invariant.
	Hence $B$ and $g^mB$ have full measure in $A$ and $g^mA$ respectively.
	This forces $\nu_o(A \cap g^m A) = 0$, i.e. $g^m$, and thus $g$, does not belong to $\stab A$.
\end{proof}


%
\section{Radial limit sets}
\label{sec: radial limit sets}
%

In this section, $G$ is still a group acting properly by isometries on a proper geodesic metric space $(X, \distV)$.
By analogy with hyperbolic geometry, we define the radial limit set for the action of $G$ on $X$.
We give here a unified approach for the various limit sets studied in \cite{Coulon:2022tu} and in the present article.

\subsection{Definition and examples}

Denote by $2^G$ the collection of all subsets of $G$.
Consider a non-increasing map $U \colon \R_+ \to 2^G$, that is such that $U(L') \subset U(L)$ whenever $L' \geq L$.
Let $r, L \in \R_+$.
The set $\Lambda_U(G,o,r,L)$ consists of all cocycles $c \in \partial X$ with the following property:
for every $T \in \R_+$, there is $g \in U(L)$ such that $\dist o{go} \geq T$ and $c \in \mathcal O_o(go,r)$.
We also let
\begin{equation*}
	\Lambda_U(G,o,r) 
	= \bigcap_{L \geq 0} \Lambda_U(G,o,r,L)
	= \bigcap_{L , T\geq 0}  \bigcup_{\substack{g \in U(L) \\ \dist {go}o \geq T}} \mathcal O_o(go,r).
\end{equation*}

\begin{defi}
\label{def: general limit set}
	The \emph{$U$-radial limit set} of $G$ is 
	\begin{equation*}
		\Lambda_U(G) 
		= \bigcup_{r \geq 0}G \Lambda_U(G,o,r). 
	\end{equation*}
\end{defi}

The set $\Lambda_U(G)$ is a $G$-invariant, saturated, Borel subset.
Invariance follows from the definition, while saturation is a consequence of (\ref{eqn: gromov product - lip}).
It follows also from (\ref{eqn: gromov product - lip}) that $\Lambda_U(G)$ does not depend on the choice of the base point $o$.

\begin{rema}
	Let $U_1, U_2 \colon \R_+ \to 2^G$ be two non-increasing maps.
	Denote by $U \colon \R_+ \to 2^G$ the map given by $U(L) = U_1(L) \cup U_2(L)$ for every $L \in \R_+$.
	One checks that 
	\begin{equation*}
		\Lambda_U(G) = \Lambda_{U_1}(G) \cup \Lambda_{U_2}(G).
	\end{equation*}
	In particular, if $U_1(L) \subset U_2(L)$ for every $L \in \R_+$, then $\Lambda_{U_1}(G) \subset \Lambda_{U_2}(G)$.
\end{rema}

\begin{exam}\
\label{exa: radial limit sets}
\begin{enumerate}
	\item \label{enu: radial limit sets - radial}
	If $U$ is the constant map such that $U(L) = G$, then we recover the \emph{radial limit set}, which we denote by $\Lambda_{\rm rad}(G)$, see \cite[Definition~4.20]{Coulon:2022tu}.
	According to the above remark, it contains all other $U$-radial limit sets.
	\item \label{enu: radial limit sets - contracting}
	Let $\alpha \in \R^*_+$.
	Denote by $\mathcal T_\alpha(L)$ the set of elements $g \in G$ such that $(o,go)$ has an $(\alpha, L)$-contracting tail (see \autoref{def: contracting tail}).
	Then the union 
	\begin{equation*}
		\Lambda_{\rm ctg}(G) = \bigcup_{\alpha \geq 0} \Lambda_{\mathcal T_\alpha}(G)
	\end{equation*}
	is the \emph{contracting limit set} studied in \cite[Section~5.2]{Coulon:2022tu}.
\end{enumerate}
Another example is detailed in \autoref{res: radial limit set contracting elt}.
\end{exam}

\begin{prop}
\label{res: convergence criterion - prelim}
	Let $c \in \Lambda_{\rm  ctg}(G)$.
	Let $(x_n)$ be a sequence of points in $X$.
	Assume that there is $M \in \R$, such that for every $n \in \N$, we have $c(o,x_n) \geq M$.
	Then any accumulation point of $(x_n)$ in $\partial X$ is equivalent to $c$.
\end{prop}

\begin{proof}
	By definition, there are $\alpha, r \in \R_+^*$, and  $g \in G$, such that $g^{-1}c \in \Lambda_{\mathcal T_\alpha}(G, o, r)$.
	For simplicity we write $b_n = \iota(x_n)$ for the cocycle associated to $x_n$.
	Up to passing to a subsequence we assume that $(b_n)$ converges to a point $b^* \in \partial X$.
	In particular, $(x_n)$ leaves every bounded subset.
		
	We fix now $R \in \R_+$ and denote by $K$ the closed ball of radius $R$ centered at $o$.
	Let $L > r + 28\alpha$.
	There is $u \in \mathcal T_\alpha(L)$ such that $\dist o{uo} >  R + r + 28\alpha$ and $g^{-1}c \in \mathcal O_o(uo, r)$.
	By definition of $\mathcal T_\alpha$, there is an $\alpha$-contracting geodesic $\tau$ ending at $uo$ and a projection $p$ of $o$ on $\tau$ satisfying $\dist p{uo} \geq L$.
	Denote by $q$ a projection of $g^{-1}c$ on $\tau$ and $q_n$ a projection of $g^{-1}x_n$ on $\tau$.
	Observe that $\dist {q_n}q \leq 4\alpha$, for all but finitely many $n \in \N$.
	Indeed, if this is not the case, then it follows from \autoref{res: proj cocycle contracting set} that 
	\begin{align*}
		M
		\leq c(o,x_n) 
		\leq c(o,gq) + c(gq, x_n)
		& \leq c(o,gq) + [g^{-1}c]\left(q, g^{-1}x_n\right) \\
		& \leq c(o,gq) - \dist {x_n}{gq} + 14\alpha, 
	\end{align*}
	which contradicts the fact that $(x_n)$ leaves every bounded subset.
	By \autoref{res: wysiwyg shadow} $\dist {uo}q \leq r + 7\alpha$, hence $\dist {uo}{q_n} \leq r + 11\alpha$.
	Recall that $p$ is a projection of $o$ on $\tau$.
	Moreover 
	\begin{equation*}
		\dist p{q_n} \geq \dist p{uo} - \dist {uo}{q_n} \geq L - (r + 11\alpha) > \alpha
	\end{equation*}
	It follows them from \autoref{res: proj contracting set} that
	\begin{equation*}
		\gro o{g^{-1}x_n}{uo}
		\leq \gro o{g^{-1}x_n}{q_n} + \dist {q_n}{uo}
		\leq r + 15\alpha.
	\end{equation*}
	Thus $g^{-1}x_n$ belongs to $\mathcal O_o(uo, r+ 15\alpha)$ as $g^{-1}c$ does.
	We assumed that $\dist o{uo} > R + r + 28\alpha$.
	Hence by \cite[Lemma~5.7]{Coulon:2022tu} we have $\norm[gK]{c - b_n} \leq 20\alpha$, for all but finitely many $n \in \N$.
	In particular, since $b^*$ is an accumulation point of $(b_n)$, we have $\norm[gK]{c - b^*} \leq 20\alpha$.
	The above estimate holds for every closed ball $K$ centered at $o$, thus $c$ and $b^*$ are equivalent.
\end{proof}


\begin{coro}
\label{res: convergence criterion - cocycle}
	Let $c \in \Lambda_{\rm  ctg}(G)$.
	Let $(x_n)$ be a sequence of points in $X$ such that $c(o,x_n)$ diverges to infinity.
	Then any accumulation point of $(x_n)$ in $\bar X$ is equivalent to $c$.
\end{coro}

\begin{proof}	
	Since cocycles in $\bar X$ are $1$-Lipschitz, we observe that no subsequence of $(x_n)$ is bounded. 
	Hence, any accumulation point of $(x_n)$ belongs to $\partial X$.
	The result now follows from \autoref{res: convergence criterion - prelim}
\end{proof}


%
\subsection{Measure of radial limit sets}
%

Assume now that $G$ contains a contracting element for its action on $X$.
If the action of $G$ on $X$ is divergent (i.e. the Poincaré series of $G$ diverges at the critical exponent $s = \omega_G$), it is a ``well-known'' fact that any $G$-invariant, $\omega_G$-conformal density is supported on the radial limit set.
More precisely we proved the following statement, which can be  understood as a partial form of the Hopf-Tsuji-Sullivan dichotomy. See Corollaries~4.25 and~5.19 as well as Proposition~5.22 in \cite{Coulon:2022tu}.

\begin{prop}
\label{res: quasi-conf + ergo}
	Let $G$ be a group acting properly by isometries on a proper geodesic space $X$.
	Assume that $G$ is not virtually cyclic and contains a contracting element.
	Let $\omega \in \R_+$ and $\nu = (\nu_x)$ be a $G$-invariant, $\omega$-conformal density.
	The followings are equivalent.
	
	\begin{enumerate}
		\item The Poincaré series $\mathcal P_G(s)$ diverges at $s = \omega$ (hence $\omega = \omega_G$).
		\item \label{enu: quasi-conf + ergo - support}
		$\nu_o$ gives positive measure to the radial limit set
		\item \label{enu: quasi-conf + ergo - support}
		 $\nu_o$ is supported on the contracting limit set $\Lambda_{\rm ctg}(G)$, hence on the radial limit set $\Lambda_{\rm rad}(G)$.
	\end{enumerate}
	In this situation, the measure $\nu_o$ restricted to $(\bar X, \mathfrak B^+)$ is non-atomic.
\end{prop}

The same line of arguments can be used to prove that several variations of radial limit sets have full measure, see \autoref{res: gal limit set full measure}.
We extract below the main ingredients used in \cite{Coulon:2022tu}.
Recall that $S_G(r, a)$ stands for the ``sphere'' in $G$ of radius $r$ centered at $a$, see (\ref{eqn: induced sphere}).

\begin{defi}
	A non-decreasing map $U \colon  \R_+ \to 2^G$ has the \emph{extension property}, if there exists $r \in \R_+$ such that for every $L \in \R_+$ and $g \in G$, there is $u \in S_G(L, r)$ with the following properties:  $gu \in U(L)$ while $go$ is $r$-close to some geodesic from $o$ to $guo$.
\end{defi}

\begin{defi}
	A non-decreasing map  $U \colon  \R_+ \to 2^G$ is \emph{almost $G$-invariant}, if there exists $r_1 \in \R_+$ such that  for every $r, L \in \R_+$ with $L > r + r_1$, for every $u \in G$, there is $T \in \R_+$ such that for every $g \in U(L+r_1)$ with $\dist o{go} \geq T$, the following holds
	\begin{enumerate}
		\item $ug \in U(L)$,
		\item $u\mathcal O_o(go,r) \subset \mathcal O_o(ugo,r+r_1)$.
	\end{enumerate}
\end{defi}

The last definition is designed to control the $G$-invariance of the limit set.
More precisely, if  $U \colon  \R_+ \to 2^G$ is an almost $G$-invariant, non-decreasing map, then for every $r, L \in \R_+$ with $L > r + r_1$,
\begin{equation*}
	G\Lambda_U(G,o,r,L + r_1) \subset \Lambda_U(G, o, r+r_1, L),
\end{equation*}
(it suffices to unwrap all definitions) and thus
\begin{equation*}
	G\Lambda_U(G,o,r) \subset \Lambda_U(G, o, r+r_1).
\end{equation*}

\begin{prop}
\label{res: gal limit set full measure}
	Let $\nu = (\nu_x)$ be a $G$-invariant $\omega_G$ conformal density.
	Let $U \colon \R_+ \to 2^G$ be an almost $G$-invariant, non-decreasing map with the extension property.
	Assume that there is $\alpha \in \R^*_+$ such that for all $L \in \R_+$, we have $U(L) \subset \mathcal T_\alpha(L)$.
	If $G$ is not virtually cyclic and the action of $G$ on $X$ is divergent then there exists $r \in \R_+$ such that $\nu_o(\Lambda_U(G, o,r)) = 1$.
	In particular, $\nu_o$ is supported on $\Lambda_U(G)$.
\end{prop}

\begin{proof}
	The proof goes verbatim as in \cite[Section~5]{Coulon:2022tu}.
\end{proof}

%
\subsection{Radial limit set associated to a contracting element}
\label{res: radial limit set contracting elt}
%

We provide here another useful variation of the contracting limit set defined previously.
Fix a contracting element $h \in G$.
Choose a bi-infinite geodesic $\gamma \colon \R_+ \to X$ which is an axis of $h$ and set $o = \gamma(0)$.
The path $\gamma$ yields two Busemann cocycles $c_\gamma, c'_\gamma \in \partial X$ (see \autoref{exa: busemann cocycle}).
We assume that $\gamma$ is oriented so that $c'_\gamma$ is equivalent to the some (hence any) accumulation point of $(h^no)_{n \in \N}$ in $\partial X$.
There is $\alpha \in \R_+^*$ such that any subpath of $\gamma$ is $\alpha$-contracting.

We define a map $U_\gamma \colon \R_+ \to 2^G$ as follows: given $L \in \R_+$, the subset $U_\gamma(L) \subset G$ consists of all elements $g \in G$ for which there is $t \in [L, \infty)$ such that $g\gamma(-t)$ is a projection of $o$ onto $g\gamma$.
Note that $U_\gamma(L) \subset \mathcal T_\alpha(L)$, for every $L \in \R_+$.
Hence the corresponding limit set $\Lambda_{U_\gamma}(G)$ is contained in the contracting limit set.

\begin{lemm}
\label{res: projection between fellow-traveling contracting geo}
	Let $\gamma_1, \gamma_2 \colon \R \to X$ be two bi-infinite geodesic.
	Assume that there exists $\epsilon \in \{\pm 1\}$, $T \in \R$ and $D \in \R_+$ such that 
	\begin{equation*}
		\dist{\gamma_1(t)}{\gamma_2(\epsilon t + T)} \leq D, \quad \forall t \in \R.
	\end{equation*}
	Let $I_1 \subset \R$ be a closed interval and $I_2 = \set{\epsilon t +T}{t \in I_1}$.
	Denote by $\gamma'_i$, the path $\gamma_i$ restricted to $I_i$.
	Let $x \in X$.
	Let $p_i = \gamma_i(t_i)$ be a projection of $x$ on $\gamma'_i$.
	If $\gamma'_2$ is $\alpha$-contracting, then $\abs{ t_2 - (\epsilon t_1 + T)} \leq 2D + 4\alpha$.
\end{lemm}

\begin{proof}
	On the one hand, \autoref{res: proj contracting set} gives
	\begin{align*}
		\dist x{\gamma_2(\epsilon t_1 + T)}
		& \geq \dist x{p_2} + \dist {p_2}{\gamma_2(\epsilon t_1 + T)} - 4\alpha \\
		& \geq d(x,\gamma'_2)+ \abs{t_2 - (\epsilon t_1 + T)} - 4\alpha.
	\end{align*}
	On the other hand, it follows from our assumption that
	\begin{equation*}
		\dist x{\gamma_2(\epsilon t_1 + T)}
		\leq \dist x{p_1} + D
		\leq \dist x{\gamma'_1} + D
		\leq \dist x{\gamma'_2} + 2D.
	\end{equation*}
	The result follows from the combination of these two observations.
\end{proof}


Combining \autoref{res: projection between fellow-traveling contracting geo} with (\ref{eqn: inclusion shadows triangle inequality}) one checks that the limit set $\Lambda_{U_\gamma}(G)$ only depends on $h$ and not on the choice of $\gamma$.
For simplicity, we denote this set by $\Lambda_h$.

\begin{rema}
	Note that the orientation of $\gamma$ matters in the definition of $U_\gamma$.
	Therefore $\Lambda_h$ is a priori distinct from $\Lambda_{h^{-1}}$, which would be defined using the map $U_{\bar \gamma} \colon \R_+ \to 2^G$ associated to the reverse path $\bar \gamma$, that is the path given by $\bar \gamma(t) = \gamma(-t)$, for every $t \in \R$.
\end{rema}

The set $\Lambda_h$ is a $G$-invariant, saturated, Borel subset of $\partial X$.
We now claim that any $G$-invariant, $\omega_G$-conformal density is supported on $\Lambda_h$.
To that end it suffices to check the assumptions of \autoref{res: gal limit set full measure}.
The first one is a variation on \cite[Lemma~4.16]{Coulon:2022tu}.

\begin{lemm}
	If $G$ is not virtually cyclic, then the map $U_\gamma \colon \R_+ \to 2^G$ satisfies the extension property.
\end{lemm}

\begin{proof}
	Since $G$ is not virtually cyclic, there is $s \in G\setminus E(h)$.
	In particular, $\pi_\gamma(s\gamma)$ has diameter at most $D_0$ for some $D_0 \in \R_+$.
	Let $L \in \R_+$ and $g \in G$.
	In order to simplify the exposition, in the remainder of this proof we denote by $D_i$ parameters which do not depend on $L$ and $g$.
	Their value could be estimated in terms of $\alpha$, $\dist o{ho}$, $\dist o{so}$, etc, but these computations will not improve the qualitative nature of the statement.
	
	Since $\gamma$ and $s\gamma$ are both $\alpha$-contracting, with a small ``overlap'', one checks that if $\gamma(t_0)$ and $\gamma'(t'_0)$ are projections of $g^{-1}o$ on $\gamma$ and $s\gamma$ respectively, then
	\begin{equation*}
		\max \{t_0, t'_0\} \leq D_1.
	\end{equation*}
	Assume that $t'_0 \leq D_1$ (the other case works in the same way).
	Since $h$ is contracting, there is $D_2$ and $u \in s \group h$ such that 
	\begin{equation*}
		\dist{u\gamma(t)}{s\gamma(t+L)} \leq D_2, \quad \forall t \in \R.
	\end{equation*}
	In particular,
	\begin{equation*}
		\abs{\dist o{uo} - L}
		\leq \dist o{so} + \dist{s\gamma(L)}{u\gamma(0)} 
		\leq D_3,
	\end{equation*}
	Hence $u \in S_G(L, D_3)$.
	In addition the projection $u\gamma(t'_1)$ of $g^{-1}o$ on $u\gamma $ is such that $\abs{t'_1 - (t'_0 - L)} \leq D_4$ (\autoref{res: projection between fellow-traveling contracting geo}).
	Hence $gu \in U_\gamma(L - D_5)$.
	Since $\gamma$ is contracting, $u\gamma(t'_1)$ is $D_6$-close to any geodesic from $g^{-1}o$ to $uo$.
	The point $u\gamma(t'_1)$ is $D_7$-close to $s\gamma(t'_0)$ while $t'_0 \leq D_1$.
	We deduce that $o$ is $D_8$-close to any geodesic from $g^{-1}o$ to $uo$.
	Hence  $go$ is $D_8$-close to any geodesic from $o$ to $guo$.
	The argument works for any $g \in G$ and $L \in \R_+$, whence the result.
\end{proof}


\begin{lemm}
\label{res: preferred map almost invariant}
	The map $U_\gamma \colon \R_+ \to 2^G$ is almost $G$-invariant.
\end{lemm}

\begin{proof}
	In this proof, we denote by $\gamma_-$ the path $\gamma$ restricted to $(-\infty, 0]$.
	We let $r_1 = 14\alpha$.
	Let $r, L \in \R_+$ with $L > r + r_1$ and $u \in G$.
	We fix $T > \dist o{uo} + L +  6\alpha$. 
	Let $g \in U_\gamma(L + r_1)$ such that $\dist o{go} \geq T$.
	By definition there a projection $p$ of $o$ on $g\gamma_-$ such that $\dist p{go} \geq L + r_1$.
	Let $p'$ be a projection of $u^{-1}o$ on $g\gamma_-$.
	We claim that $\dist {p'}{go} \geq L$.
	Assume on the contrary that it is not the case.
	In particular, $\dist p{p'} > \alpha$.
	Since $g\gamma_-$ is $\alpha$-contracting we get from \autoref{res: proj contracting set}
	\begin{align*}
		\dist o{u^{-1}o}
		\geq \dist op + \dist p{p'}  - 6\alpha
		& \geq \dist op + \left[\dist p{go} - \dist{p'}{go}\right] - 6\alpha \\
		& \geq \dist o{go} - L - 6\alpha.
	\end{align*}
	This contradicts the fact that $\dist o{go} \geq T$, and complete the proof of our claim.
	In particular, $ug$ belongs to $U_\gamma(L)$.
	
	Let $c \in \mathcal O_o(go,r)$.
	Denote by $q$ a projection of $c$ on $g\gamma_-$.
	Since $L + r_1 > r + 13\alpha$, we know by \autoref{res: wysiwyg shadow} that $\dist q{go} \leq r + 7\alpha$.
	In particular, $\dist {p'}q > 4\alpha$.
	Recall that $p'$ and $q$ are respective projections of $u^{-1}o$ and $c$ on $g\gamma_-$.
	It follows from \autoref{res: proj cocycle contracting set} that $\gro {u^{-1}o}cq \leq 5\alpha$.
	The triangle inequality now yields $\gro {u^{-1}o}c{go} \leq r + 12\alpha$, i.e. $\gro o{uc}{ugo} \leq r + 12\alpha$
	We just proved that 	$u\mathcal O_o(go,r) \subset \mathcal O_o(ugo,r + r_1)$, whence the result.
\end{proof}


The next statement is now just a particular case of \autoref{res: gal limit set full measure}.

\begin{prop}
\label{res: preferred limit set full measure}
	Assume that $G$ is not virtually cyclic and the action of $G$ on $X$ is divergent.
	Then for every $G$-invariant, $\omega_G$-conformal density $\nu = (\nu_x)$ we have $\nu_o(\Lambda_h) = 1$.
\end{prop}

We complete this section with additional properties of $\Lambda_h$.

\begin{prop}
\label{res: preferred limit set - fixed shadow width}
	There is $r_0 \in \R_+$ such that the set $\Lambda_h$ coincides with $\Lambda_{U_\gamma}(G, o, r_0)$.
\end{prop}

\begin{proof}
	By construction $\Lambda_{U_\gamma}(G, o, r)$ is contained in $\Lambda_h$ for every $r \in \R_+$.
	Hence it suffices to prove the converse inclusion for a suitable value of $r$.
	Since $h$ is contracting, there is $D \in \R_+$ such that for every $T \in \R_+$, there is $u \in \group h$, with the following property:
	\begin{equation*}
		\dist {u \gamma(t)}{\gamma(t -T)} \leq D, \quad \forall t \in \R.
	\end{equation*}
	Set $r_0 = D + 4\alpha$.
	Since $U_\gamma$ is almost $G$-invariant (\autoref{res: preferred map almost invariant}), we know that $\Lambda_h$ is contained in 
	\begin{equation*}
		\bigcup_{r \geq 0} \Lambda_{U_\gamma}(G, o, r) = \bigcup_{r \geq 0}\bigcap_{L \geq 0} \Lambda_{U_\gamma}(G, o, r, L).
	\end{equation*}
	Let $c \in \Lambda_h$.
	There is $r \in \R_+$ such that $c \in  \Lambda_{U_\gamma}(G, o, r)$.
	We fix $T = r + 7\alpha$.
	According to our choice of $D$, there is $u \in \group h$ such that $\dist {u \gamma(t)}{\gamma(t -T)} \leq D$, for every $t \in \R_+$.
	Let $L \in \R_+$ and $L' > \max\{L + T + 2D + 4\alpha, r + 13\alpha\}$.
	Since $c$ belongs to $ \Lambda_{U_\gamma}(G, o, r, L')$, there is a sequence $g_n \in {U_\gamma}(L')$ such that $\dist o{g_no}$ diverges to infinity and $c \in \mathcal O_o(g_no, r)$, for every $n \in \N$.
	According to \autoref{res: wysiwyg shadow}, $c \in \mathcal O_o(g_n\gamma(-T), 4\alpha)$.
	It follows from the triangle inequality that $c \in \mathcal O_o(g_nuo, r_0)$.
	In addition, $g_nu \in U_\gamma(L)$ by \autoref{res: projection between fellow-traveling contracting geo}.
	Moreover, $\dist o{g_nuo}$ diverges to infinity.
	Hence $c$ belongs to $ \Lambda_U(G, o, r_0, L)$.
	This fact holds for every $L \in \R_+$, whence the result.
\end{proof}


\begin{prop}
\label{res: separation triangle cocycles}
	Let $c' \in \Lambda_h$ and $\Lambda_0\subset \bar X$ be a closed subset such that $c'$ is not equivalent to any cocycle in $\Lambda_0$.
	For all $\ell, M \in \R_+$, there is $u \in G$ such that 
	\begin{enumerate}
		\item $c'(o,uo) \geq M$,
		\item $\Lambda_0$ and $c'$ are $2\alpha$-separated by $(uy, uy')$, where $y = \gamma(-\ell)$ and $y' = \gamma(\ell)$.
	\end{enumerate}
\end{prop}

\begin{rema}
\label{rem: separation triangle cocycles}
	Since $o$ lies on $\gamma$ between $y$ and $y'$, we observe that 
	\begin{equation*}
		\gro c{uy'}{uy} = \gro c{uo}{uy} + \gro c{uy'}{uo} - \gro{uy}{uy'}{uo} = \gro c{uo}{uy} + \gro c{uy'}{uo}, \quad \forall c \in \partial X.
	\end{equation*}
	In particular, $\gro c{uo}{uy} \leq \gro c{uy'}{uy}$, hence $\Lambda_0 \subset \mathcal O_{uo}(uy, 2\alpha)$.
	Similarly $c'$ also belongs to $\mathcal O_{uo}(uy', 2\alpha)$.
\end{rema}

\begin{proof}
	Since $h$ is contracting, there is $D \in \R_+$, such that for every $T \in \R$, there is $n \in \Z$, such that 
	\begin{equation}
	\label{eqn: separation pair cocycles}
		\dist {h^n\gamma(t)}{ \gamma(t-T)} \leq D, \quad \forall t \in \R.
	\end{equation}
	In addition, we write $r_0$ for the parameter given by \autoref{res: preferred limit set - fixed shadow width} and let $r_1 = 2D + \max\{r_0 + 12\alpha, 23\alpha\}$.
	Let $\ell, M \in \R_+$.
	Set $T = \ell + r_1 + 7\alpha$.
	As explained above, there is $n \in \Z$ such that (\ref{eqn: separation pair cocycles}) holds.
	Fix $L \in \R_+$ such that
	\begin{equation*}
		L > \max \{ 2r_1 + 21 \alpha, \ell + T + r_1 + 7\alpha\}.
	\end{equation*}
	According to \autoref{res: preferred limit set - fixed shadow width}, there is a sequence $(g_k)$ of elements in $U_\gamma(L)$ such that $\dist o{g_ko}$ diverges to infinity, while $c' \in \mathcal O_o(g_ko, r_0)$, for every $k \in \N$.
	In particular, 
	\begin{equation*}
		c'(o, g_ko) \geq \dist o{g_ko} - 2\gro o{c'}{g_ko} 
		\geq \dist o{g_ko} - 2r_0.
	\end{equation*}
	Consequently $c'(o, g_ko)$ diverges to infinity, hence any accumulation point of $(g_ko)$ is equivalent to $c'$ (\autoref{res: convergence criterion - cocycle}).
	
	We let $u_k = g_kh^n$ for every $k \in \N$, and claim that if $k$ is large enough, then $u_k$ satisfies the conclusion of the statement.
	Since $g_k$ belongs to $U_\gamma(L)$, we get from \autoref{res: wysiwyg shadow} that 
	\begin{equation*}
		\gro {g_k\gamma(-L)}{c'}{g_ko} \leq r_0 + 12\alpha.
	\end{equation*}
	Set $z = \gamma(-L+T)$ and $z' = \gamma(T)$.
	Combined with (\ref{eqn: separation pair cocycles}) and the triangle inequality, the previous estimates yield $c' \in \mathcal O_{u_kz}(u_kz',r_1)$.
	Note that $\dist {u_ko}{g_ko}$ equals $\dist o{h^no}$ and thus does not depends on $k$.
	Since $c'(o, g_ko)$ diverges to infinity, so does $c'(o, u_ko)$.
	In particular, $c'(o, u_ko) \geq M$ provided $k$ is sufficiently large.

	We are now going to prove that $\Lambda_0 \subset \mathcal O_{u_kz'}(u_kz, r_1)$ for all but finitely many $k \in \N$.
	Since $c'$ belongs to $\Lambda_{\mathcal T_\alpha}(G)$, its equivalence class is 
	\begin{equation*}
		[c'] = \set{b' \in \partial X}{\norm[\infty]{b' - c'} \leq 20\alpha},
	\end{equation*}
	see \cite[Proposition~5.10]{Coulon:2022tu}.
	In particular, it is a closed subset of $\bar X$.
	Thus for any $c$ in $\partial X \setminus [c']$, there is a neighborhood $V_c \subset \bar X$ of $c$ whose closure does not intersect $[c']$.
	We claim that for all but finitely many $k \in \N$, for all $b \in V_c$, any projection $q$ of $b$ onto $g_k\gamma$ restricted to $\intval{-L}0$, satisfies $\dist {g_ko}q \geq L - 18\alpha$.
	Suppose on the contrary that it is not the case.
	Up to passing to a subsequence, we can assume that for every $k \in \N$, there is $b_k \in V_c$ a projection $q_k$ is a projection of $b_k$ onto $g_k\gamma$ restricted to $\intval{-L}0$ such that $\dist {g_ko}{q_k} < L - 18\alpha$.
	Moreover we can suppose that $(b_k)$ converges to a cocycle $b$ in the closure of $V_c$.
	Recall that $g_k$ belongs to $U_\gamma(L)$.
	Hence there is $t_k \geq L$ such that $p_k = g_k \gamma(-t_k)$ is a projection of $o$ on $g_k\gamma$.
	In particular, $\dist{p_k}{q_k} > 4\alpha$.
	Using \autoref{res: proj cocycle contracting set} we observe that $\gro o{b_k}{q_k} \leq 5\alpha$.
	Hence $\gro o{b_k}{g_ko} \leq r_k$,  where $r_k = \dist {g_ko}{q_k} + 5\alpha$.
	In particular, both $g_ko$ and $b_k$ belong to $\mathcal O_o(g_ko,r_k)$.
	Note that $L > r_k + 13\alpha$.
	In view of \cite[Lemma~5.7]{Coulon:2022tu} there is a exhaustion of $X$ by a non-decreasing collection of compact subsets $K_k \subset X$ such that for every $k \in \N$, we have $\norm[K_k]{b_k - \iota(g_ko)} \leq 20\alpha$.
	Recall that any accumulation point $b'$ of $(g_ko)$ is equivalent to $c'$.
	Passing to the limit we see that $\norm[\infty]{b - b'} \leq 20\alpha$, hence $b$ is equivalent to $c'$ as well.
	This contradicts the fact that the closure of $V_c$ does not intersect $[c']$ and completes the proof of our claim.

	It follows from \autoref{res: proj cocycle contracting set}, that for all but finitely many $k \in \N$, for all $b \in V_c$, we have 
	\begin{equation*}
		\gro b{g_ko}{g_k\gamma(-L)} \leq \gro b{g_ko}{q_k} + \left[L- \dist {q_k}{g_ko}\right] \leq 23\alpha,
	\end{equation*}
	where $q_k$ is a projection of $b$ on $g_k \gamma$ restricted to $\intval{-L}0$.
	Combined with (\ref{eqn: separation pair cocycles}) and the triangle inequality, it yields  $V_c \subset \mathcal O_{u_kz'}(u_kz,r_1)$.
	Since $\Lambda_0$ is compact it can be covered by finitely many open subsets of the form $V_c$ where $c \in \partial X \setminus[c']$.
	Hence $\Lambda_0 \subset \mathcal O_{u_kz'}(u_kz,r_1)$ for all but finitely many $k \in \N$.
	In other words $\Lambda_0$ and $c'$ are $r_1$-separated by $(u_kz,u_kz')$.
	We conclude with \autoref{res: bi-gradient line} that  $\Lambda_0$ and $c'$ are $2\alpha$-separated by $(u_ky,u_ky')$.
\end{proof}


%
\section{The reduced horoboundary as a standard measure space}
\label{res: reduced bdy standard}
%

The goal of this section is to prove (among other things) the following statement.

\begin{prop}
\label{res: standard measure space}
	Let $G$ be a group acting properly, by isometries on a proper, geodesic space $X$, with a contracting element.
	There is a $G$-invariant, saturated Borel subset $\Lambda \subset \Lambda_{\rm ctg}(G)$ satisfying the following properties.
	\begin{enumerate}
		\item For any $G$-invariant, $\omega_G$-conformal density $\nu = (\nu_x)$ the subset $\Lambda_{\rm ctg}(G) \setminus \Lambda$ has zero $\nu_o$-measure.
		\item There is a distance $\distV[\infty]$ on $\Lambda \qsim$ such that the metric topology on $\Lambda\qsim$ coincides with the quotient topology induced from $\partial X$.
		\item The quotient $(\Lambda\qsim, \distV[\infty])$ is a separable complete metric space.
	\end{enumerate}
\end{prop}

\begin{rema}
\label{rem: standard measure space}
	Assume that $G$ is not virtually cyclic and the action of $G$ on $X$ is divergent.
	Write $\mathfrak C$ for the Borel $\sigma$-algebra on $(\Lambda \qsim, \distV[\infty])$.
	Let $\nu$ be a $G$-invariant, $\omega_G$-conformal density.
	Denote by $\mu_o$ the measure obtained by first restricting $\nu_o$ to $\Lambda$ and then pushing this restriction to $\Lambda \qsim$.
	Then $(\Lambda\qsim, \mathfrak C, \mu_o)$ is a standard probability space.
	Moreover, $L^p(\partial X, \mathfrak B^+, \nu_o)$ and $L^p(\Lambda\qsim, \mathfrak C, \mu_o)$ are isomorphic, for every $p \in [1, \infty]$.
	In this sense $(\partial X, \mathfrak B^+, \nu_o)$ can be treated as a standard measure space.
\end{rema}

%
\subsection{Initial data}
\label{sec: initial data}
%

We assume that $G$ contains a contracting element, say $h \in G$.
We choose a bi-infinite geodesic $\gamma \colon \R \to X$ which is an axis of $h$.
We suppose that the base point of $X$ is $o = \gamma(0)$.
The path $\gamma$ yields two Busemann cocycles $c_\gamma, c'_\gamma \in \partial X$, as in \autoref{exa: busemann cocycle}.
We assume that $\gamma$ is oriented so that $c'_\gamma$ is equivalent to the some (hence any) accumulation point of $(h^no)_{n \in \N}$ in $\partial X$.

Denote by $\bar \gamma \colon \R \to X$ the reverse path given by $\bar \gamma(t) = \gamma(-t)$ for every $t \in \R$.
For simplicity we write $U_-, U_+ \colon \R_+ \to 2^G$ for the map $U_\gamma$ and $U_{\bar \gamma}$ used in \autoref{res: radial limit set contracting elt}.
We define a map $U \colon \R_+ \to 2^G$ by 
\begin{equation*}
	U(L) = U_-(L) \cup U_+(L), \quad \forall L \in \R_+.
\end{equation*}
In addition, we let
\begin{equation*}
	\Lambda_- = \Lambda_{h^{-1}}, \quad
	\Lambda_+ = \Lambda_h, 
	\quad \text{and} \quad
	\Lambda = \Lambda_- \cup \Lambda_+.
\end{equation*}
(the need of ``symmetrizing'' the set $\Lambda$ as well as the visual distance defined below will appear later, see for instance Lemmas~\ref{res: small distance with three cocycles} and~\ref{res: complete metric space - prelim}).
There is $\alpha \in \R_+^*$ such that the following holds.
\begin{enumerate}[label=(H\arabic{*}), ref=(H\arabic{*})]
	\item \label{enu: alpha - contracting path}
	Any subpath of $\gamma$ is $\alpha$-contracting.
	\item \label{enu: alpha - disp h}
	$\dist o{ho} \leq \alpha$.
	\item \label{enu: alpha - diam proj}
	$\diam { \pi_\gamma(g\gamma) } \leq \alpha$, for every $g \in G \setminus E(h)$;

	\item \label{enu: alpha - fellow travel}
	For every $g \in E(h)$, there is $\epsilon \in \{\pm 1\}$ and $T \in \R$ such that  for every $t \in \R_+$, we have 
	\begin{equation*}
		\dist{g \gamma(t)}{\gamma(\epsilon t + T)}\leq \alpha.
	\end{equation*}
	In particular, the Hausdorff distance between $\gamma$ and $g\gamma$ if at most $\alpha$.
	\item \label{enu: alpha - translation}
	Conversely, for every $T \in \R$, there is $g \in \group h$ such that  for every $t \in \R_+$, we have $\dist{g \gamma(t)}{\gamma(t + T)}\leq \alpha$.
\end{enumerate}
These properties follow from the fact that $h$ is contracting.
The proofs work as if $h$ was a loxodromic element in a hyperbolic space and are left to the reader.
Consider a $G$-invariant, $\omega$-conformal density $\nu = (\nu_x)$.
If $G$ is not virtually cyclic and the action of $G$ on $X$ is divergent, then $\nu$ gives full measure to $\Lambda_-$, $\Lambda_+$, and thus $\Lambda$.
If the action of $G$ is convergent then $\Lambda_{\rm ctg}(G)$ has zero $\nu_o$-measure.
If $G$ is virtually cyclic, then $\Lambda = \partial X$.
In all cases $\Lambda_{\rm ctg}(G) \setminus \Lambda$ has zero $\nu_o$-measure.

\paragraph{Separation.}

\begin{voca}
	In order to shorten the statements, we adopt the following terminology.
	Let $Z$ and $Z'$ be two subsets of $\bar X$.
	Let $r, L \in \R_+$ and $\epsilon \in \{\pm 1\}$.
	Set $x = \gamma(-\epsilon L)$ and $x' = \gamma(\epsilon L)$.
	We say that $Z$ and $Z'$ are \emph{$(L,r,\epsilon)$-separated at an element $u \in G$}, if $Z$ and $Z'$ are $r$-separated by $(ux,ux')$, that is
	\begin{equation*}
		Z \subset \mathcal O_{ux'}(ux,r) 
		\quad \text{and} \quad
		Z' \subset \mathcal O_{ux}(ux',r).
	\end{equation*}

\end{voca}

\begin{rema}
\label{rem: separation at an element}
	Let $r,M \in \R_+$ with $M > r + 11\alpha$.
	Let $\epsilon \in \{\pm 1\}$.
	According to \autoref{res: bi-gradient line}, if $Z$ and $Z'$ are $(M, r, \epsilon)$-separated at some $u \in G$, then they are also $(L, 2\alpha, \epsilon)$-separated at $u$, for any $L \leq M - (r + 7\alpha)$.
\end{rema}

\begin{lemm}
\label{res: simultaneous separation}
	Let $r_1, r_2, L_1, L_2 \in \R_+$ with $L_i > r_i + 32\alpha$.
	Let $\epsilon_1, \epsilon_2 \in \{\pm 1\}$.
	Let $g_1, g_2 \in G$.
	Let $c,c' \in \bar X$.
	Assume that $c$ and $c'$ are $(L_i, r_i, \epsilon_i)$-separated at $g_i$.
	Let $p$ be a projection of $g_1o$ on $g_2\gamma$.
	If $\dist p{g_2o} < L_2 -(r_2 + 20\alpha)$, then $g_1^{-1}g_2 \in E(h)$.
\end{lemm}

\begin{proof}
	Let $x_i = g_i\gamma(-\epsilon_i L_i)$ and $x_i' = g_i \gamma(\epsilon_i L_i)$.
	Denote by $\nu \colon I \to X$ a complete gradient arc from $c$ to $c'$ (such an arc exists by \autoref{res: bi-gradient line}).
	For simplicity we write $\gamma_i$ for the path $g_i\gamma$ restricted to $\intval{-L_i}{L_i}$, so that $\gamma_i$ joins $x_i$ to $x'_i$.
	Let $q_i$ and $q'_i$ be respective projections of $c$ and $c'$ on $\gamma_i$.
	It follows from  \autoref{res: bi-gradient line}~\ref{enu: bi-gradient line - entry/exit point} that $d(\nu, \gamma_i) < \alpha$.
	Moreover if $\nu(a_i)$ and $\nu(a'_i)$ stands for the entry / exit point of $\nu$ in $\mathcal N_\alpha(\gamma_i)$ then 
	\begin{equation}
	\label{eqn: simultaneous separation}
		\dist{\nu(a_i)}{q_i} \leq 5\alpha
		\quad \text{and} \quad
		\dist{\nu(a'_i)}{q'_i} \leq 5\alpha.
	\end{equation}
	In addition, $\nu$ restricted to $\intval{a_i}{a'_i}$ lies in the $(5\alpha/2)$-neighborhood of $\gamma_i$, hence of $g_i \gamma$.
	Recall that $\diam {\pi_\gamma (g\gamma)} \leq \alpha$ whenever $g \notin E(h)$, see \ref{enu: alpha - diam proj}.
	In view of \autoref{res: intersection vs projection} is suffices to show that the diameter of $\intval {a_1}{a'_1} \cap \intval {a_2}{a'_2}$ is larger than $15\alpha$.
	
	We first claim that $a'_i - a_i > 15\alpha$.
	Indeed, it follows from \autoref{res: bi-gradient line}~\ref{enu: bi-gradient line - proj} that $\dist {x_i}{q_i} \leq r_i + 7\alpha$ and $\dist {x'_i}{q'_i} \leq r_i + 7\alpha$.
	Thus the triangle inequality combined with (\ref{eqn: simultaneous separation}) yields 
	\begin{equation*}
		a'_i - a_i 
		\geq \dist {q_i}{q'_i} - 10\alpha 
		\geq \dist {x_i}{x'_i} - 2r_i - 24\alpha
		> 15\alpha.
	\end{equation*}
	Suppose now that the diameter of $\intval {a_1}{a'_1} \cap \intval {a_2}{a'_2}$ is at most $15\alpha$.
	Either $a'_1 \leq a_2 + 15\alpha$ or $a'_2 \leq a_1 + 15\alpha$.
	In the remainder of the proof, we assume that $a'_1 \leq a_2 + 15\alpha$ (the other case works in the same way).
	As we already observed, $\nu$ restricted to $\intval{a_1}{a'_1}$ lies in the $(5\alpha/2)$-neighborhood of $\gamma_1$, which contains $g_1o$.
	It follows from \cite[Lemma~5.1]{Coulon:2022tu} that $g_1o$ is $5\alpha$-close to the point $\nu(b_1)$ for some $b_1 \in \intval{a_1}{a'_1}$.
	However the triangle inequality gives
	\begin{equation*}
		a'_1 - b_1 
		\geq \dist {q'_1}{g_1o} - 10\alpha
		\geq \dist {x'_1}{g_1o} - (r_1 +17\alpha)
		\geq L_1 - (r_1 + 17\alpha)
		> 15\alpha
	\end{equation*}
	Consequently
	\begin{equation*}
		b_1 < a'_1 - 15\alpha \leq a_2.
	\end{equation*}
	Recall that the distance between $\nu$ restricted to $\intval {b_1}{a_2}$ and $\gamma_2$ is at least $\alpha$.
	It follows from \autoref{res: proj cocycle contracting set} that the projection of $\nu(b_1)$ on $\gamma_2$ is $4\alpha$-close to $q_2$.
	Recall that $p$ is a projection of $g_1o$ on $g_2\gamma$ (and not just on $\gamma_2$).
	However since $\dist p{g_2o} \leq L_2$, it belongs to $\gamma_2$.
	Hence $p$ is also a projection of $g_1o$ on $\gamma_2$.
	Since projection on $\gamma_2$ is large-scale Lipschitz (\autoref{rem: Lipschitz proj}) the point $p$ is $13\alpha$-close to $q_2$.
	It follows now from the triangle inequality that 
	\begin{align*}
		L_2 
		\leq \dist {x_2}{g_2o} 
		& \leq \dist {x_2}{q_2} + \dist{q_2}p + \dist p{g_2o} \\
		& \leq  \dist p{g_2o}  +  r_2 + 20\alpha.
	\end{align*}
	This contradicts our assumption and the proof is complete.	
\end{proof}


%
\subsection{Admissible sequences}
%

In the previous section, we defined a preferred radial limit set $\Lambda$ (with full measure provided  $G$ is not virtually cyclic and the action of $G$ on $X$ is divergent).
We now begin our work to build a complete metric $\distV[\infty]$ on $\Lambda \qsim$.
Inspired by the case of hyperbolic groups, we start by defining a collection of ultra-metrics $(\distV[L])$ on $\Lambda \qsim$, which can be thought of as a ``visual metrics on $\Lambda$ seen from $o$''.

\begin{voca}
	We say that a subset $I \subset \Z$ is \emph{connected} if for all integers $i,j \in I$ with $i \leq j$, the set $\intvald ij$ is contained in $I$.
\end{voca}

\begin{defi}
\label{def: admissible sequence}
	Let $L \in \R_+$.
	A sequence $(u_i)$ of elements of $G$ indexed by a connected set $I \subset \Z$ is \emph{$L$-admissible} if for every $i \in \Z$  for which $i-1, i, i+1 \in I$, there are $\epsilon_i \in \{\pm 1\}$ as well as $z \in B(u_{i-1}o, 10\alpha)$ and $z' \in B(u_{i+1}o, 10\alpha)$ which are $(L,2\alpha, \epsilon_i)$-separated at $u_i$.
	The number $\epsilon_i$ is called the \emph{orientation} of $u_i$ in the sequence.
	
	If $I$ is finite (\resp $I = \N$) we say that the sequence $(u_i)$ is \emph{finite} (\resp \emph{infinite}).
\end{defi}

\begin{rema*}
	Just as working with hyperbolic spaces involves many estimates in terms of the hyperbolicity constant $\delta$, our approach requires numerous computations dealing with errors in terms of $\alpha$.
	In all the results below, we do not claim to have given the optimal estimates. 
	We sometimes rounded our calculations to get statements easier to digest.
	What matters the most are the orders of magnitude, i.e. which quantities are large / small compare to $\alpha$.
\end{rema*}

Being admissible is a $G$-invariant property: if $(u_i)$ is an $L$-admissible sequence, then so is $(gu_i)$ for every $g \in G$.
The definition is local (it only involves the indices $i-1$, $i$ and $i+1$).
The next statement can be thought of as a local-to-global phenomenon.

\begin{lemm}
\label{res: admissible sequence - global projection}
	Let $L >200\alpha$.
	Assume that $(u_0, \dots, u_n)$ is a finite $L$-admissible sequence. 
	Let $i \in \intvald 1{n-1}$.
	Let $\epsilon_i$ be the orientation of $u_i$ in the sequence.
	The points $u_0o$ and $u_no$ are $(L,16\alpha, \epsilon_i)$-separated at $u_i$.
	In particular, $\gro {u_0o}{u_no}{u_io} \leq 4\alpha$, for every $i \in \intvald 0n$.
	Moreover if $i \in \intvald 1{n-1}$, then $u_i$ belongs to $U(L - 18\alpha)$.
\end{lemm}

\begin{proof}
	To make it easier to follow the computations we let $r = 16\alpha$.
	The first part of the statement is proved by induction on $n$.
	If $n \leq 2$, it is a consequence of (\ref{eqn: gromov product - lip}).
	Let $n \geq 2$ for which the statement holds.
	Let $(u_0, \dots, u_{n+1})$ be finite $L$-admissible sequence.
	For simplicity, for every $i \in \intvald 0{n+1}$, we let $y_i = u_io$.
	In addition, if $i \in \intvald 1n$, we write $\gamma_i \colon \R \to X$ for the path given by $\gamma_i(t) = u_i \gamma(\epsilon_i t)$ and let $x_i = \gamma_i(-L)$ and $x'_i = \gamma_i(L)$.
	According to the induction hypothesis applied with $(u_0, \dots, u_n)$ and $(u_1, \dots, u_{n+1})$ we know that 
	\begin{enumerate}
		\item $y_0 \in \mathcal O_{x'_i}(x_i, r)$ for every $i \in \intvald 1{n-1}$, and
		\item $y_{n+1}\in \mathcal O_{x_i}(x'_i, r)$ for every $i \in \intvald 2n$.
	\end{enumerate}
	We are left to prove that $y_0 \in \mathcal O_{x'_n}(x_n, r)$ and $y_{n+1}\in \mathcal O_{x_1}(x'_1, r)$.
	We focus on the latter property. 
	The other one follows by symmetry.
	Denote by $p_2 = \gamma_2(t_2)$, $q_2 = \gamma_2(s_2)$, and $q'_2 = \gamma_2(s'_2)$ respective projections of $x_1$, $y_1$, and $y_{n+1}$ on $\gamma_2$.
	Since $y_{n+1} \in \mathcal O_{x_2}(x'_2, r)$, we observe using \autoref{res: proj contracting set} that $s'_2 \geq L - (r + 2\alpha)$.
	Similarly, $s_2 \leq -L +  (r + 2\alpha)$.
	We now claim that $t_2 \leq - L + (r + 6\alpha)$.
	Suppose on the contrary that it is not the case.
	In particular, $t_2 - s_2 > 4\alpha$, hence $\dist {p_2}{q_2} > 4\alpha$.
	Note that $p_2$ and $q_2$ are both projections onto $\gamma_2$ of points from $\gamma_1$.
	According to \ref{enu: alpha - diam proj}, $u_1^{-1}u_2 \in E(h)$.
	By \ref{enu: alpha - fellow travel}, there is $\eta \in \{\pm 1\}$ and $T \in \R$ such that
	\begin{equation*}
		\dist {\gamma_1(t)}{\gamma_2(\eta t + T)} \leq \alpha, \quad \forall t \in \R.
	\end{equation*}
	In particular, $\dist {p_2}{x_1} \leq \alpha$ and $\dist {q_2}{y_1} \leq \alpha$, hence $\abs{t_2 - (T - \eta L)} \leq 2 \alpha$ and $\abs{s_2 - T} \leq 2\alpha$.
	We observed previously that $t_2 - s_2 > 4\alpha$, which forces $\eta = -1$.
	Consequently $x'_1$ is $\alpha$-close to the point $p'_2 = \gamma_2(t'_2)$ where $t'_2 = - L + T$.
	In particular, $t'_2 \leq s_2 \leq t_2$, while $s_2 \leq 0$.
	It follows that 
	\begin{equation*}
		\gro {x_1}{y_2}{x'_1}
		\geq \gro {p_2}{y_2}{p'_2} - 2\alpha
		\geq \dist {p'_2}{q_2} - 2\alpha
		\geq \dist {x'_1}{y_1} - 4\alpha
		\geq L - 4\alpha
		> r .
	\end{equation*}
	This contradicts the fact that $y_2$ also belongs to $\mathcal O_{x_1}(x'_1,r)$ and completes the proof of our claim.
	Observe now, that $x_1$ and $y_{n+1}$ are $2\alpha$-separated by $(p_2,q'_2)$.
	According to \autoref{res: bi-gradient line}~\ref{enu: bi-gradient line - double gromov product}, $\gro{x_1}{y_{n+1}}{y_2} \leq 4\alpha$.
	Thus
	\begin{equation*}
		\gro {x_1}{y_{n+1}}{x'_1}
		\leq \gro{x_1}{y_2}{x'_1} + \gro {x_1}{y_{n+1}}{y_2}
		\leq 16\alpha
		\leq r,
	\end{equation*}
	(the middle inequality above follows from the definition of admissible chain).
	This completes the proof of the induction step for $n + 1$, and the first part of our statement.
	The second part regarding Gromov products follows from \autoref{res: bi-gradient line}~\ref{enu: bi-gradient line - double gromov product}.
	The conclusion on $U$ is an application of \autoref{res: proj contracting set}.
\end{proof}


The next two statements will help us defining the endpoints of an infinite admissible sequence.

\begin{lemm}
\label{res: cylinder small gromov product - prelim}
	Let $L > 200\alpha$.
	Let $n \in \N \setminus\{0, 1\}$.
	Assume that $(u_0, \dots, u_n)$ is a finite $L$-admissible sequence. 
	Then 
	\begin{equation*}
		\dist {u_no}{u_0o} 
		\geq \sum_{i = 0}^{n-1} \dist {u_io}{u_{i+1}o} - 8(n-1)\alpha
		\geq n(L - 24\alpha).
	\end{equation*}
	\end{lemm}

\begin{proof}
	It follows from \autoref{res: admissible sequence - global projection} that $\gro{u_{i+1}o}{u_0o}{u_io} \leq 4\alpha$, for every $i \in \intvald 1{n-1}$, that is
	\begin{equation*}
		\dist {u_{i+1}o}{u_0o} \geq \dist {u_{i+1}o}{u_io} +  \dist {u_io}{u_0o}  - 8\alpha.
	\end{equation*}
	Hence the first inequality.
	We are left to prove that $\dist {u_{i+1}o}{u_io}  \geq L -16\alpha$, for every $i \in \intvald 0{n-1}$.
	Assume first that $i \in \intvald 1{n-1}$.
	Denote by $\epsilon_i$ the orientation of $u_i$ in the sequence.
	Let $x_i = u_i\gamma(-\epsilon_i L)$ and $x'_i = u_i \gamma(-\epsilon_i L)$.
	Combining \autoref{res: admissible sequence - global projection} and \autoref{rem: separation triangle cocycles} we observe that 
	\begin{equation*}
		L - \dist{u_io}{u_{i+1}o}
		\leq \gro {u_io}{u_{i+1}o}{x'_i}
		\leq \gro {x_i}{u_{i+1}o}{x'_i} 
		\leq 16\alpha,
	\end{equation*}
	whence our claim.
	If $i = 0$, the argument works in the same way switching the roles of $u_i$ and $u_{i+1}$.
\end{proof}


\begin{lemm}
\label{res: equivalent accumulation points}
	Let $L > 200\alpha$.
	Let $(u_n)$ be an infinite $L$-admissible sequence.
	If $c_1$ and $c_2$ are two accumulation points of $(u_no)$ then $\norm[\infty]{c_1 - c_2} \leq 20\alpha$.
	Moreover any cocycle $c \in \partial X$ equivalent to $c_i$ satisfies  $\norm[\infty]{c - c_i} \leq 20\alpha$.
\end{lemm}

\begin{rema}
	Note that we are not claiming that $c_i$ belongs to $\Lambda$.
	Indeed, it certainly belongs to $\Lambda_U(o,4\alpha, L-18\alpha)$ (see the proof below) but may fail to be in $\Lambda_U(o, 4\alpha, L')$ for every $L' \in \R_+$.
\end{rema}

\begin{proof}
	Let $k \in \N$.
	According to \autoref{res: admissible sequence - global projection}, for every integer $n \geq k$, we have $\gro o{u_no}{u_ko} \leq 4 \alpha$.
	Thus any two accumulation points $c_1$ and $c_2$ of $(u_no)$, belong to $\mathcal O_o(u_ko,4\alpha)$.
	\autoref{res: admissible sequence - global projection} also tells us that $(o,u_ko)$ has an $(\alpha, L-18\alpha)$-contracting tail.
	Recall also that $\dist o{u_ko}$ diverges to infinity (\autoref{res: cylinder small gromov product - prelim}).
	It follows now from \cite[Lemma~5.7]{Coulon:2022tu} that $\norm[\infty]{c_1-c_2} \leq 20\alpha$.
	The above argument also proves that $c_i$ belongs to $\Lambda_U(o,4\alpha, L-18\alpha)$.
	The final conclusion is therefore a consequence of  \cite[Lemma~5.10]{Coulon:2022tu}.
\end{proof}


\begin{voca}
	In view of the above statement, there exists $c \in \partial X$, such that any accumulation point of $(u_no)$ is equivalent to $c$.
	This allows us to define the endpoints of an infinite $L$-admissible sequence $(u_n)$, provided $L > 200\alpha$.
	Its \emph{initial point} is $u_0$, while its \emph{terminal point} is the equivalence class $[c]$ of $c$.
	We sometimes make an abuse of vocabulary and say that $c$ is the endpoint of $(u_n)$.
	The initial and terminal points of a finite admissible sequence $(u_0, \dots, u_n)$ are simply $u_0$ and $u_n$.
	Note that the endpoints of an admissible sequence are either elements of $G$ or (equivalence classes of) cocycles in $\partial X$.
\end{voca}

The next lemma generalizes \autoref{res: admissible sequence - global projection}.

\begin{lemm}
\label{res: cylinder small gromov product}
	Let $L \in \R_+$ with $L > 200\alpha$.
	Let $(u_i)_{i \in \N}$ be an infinite $L$-admissible sequence starting at $1$.
	Let $c \in \partial X$ be an endpoint of $(u_io)$.
	Let $i \in \N \setminus\{0\}$.
	Let $\epsilon_i$ be the orientation of $u_i$ in the sequence.
	Then $o$ and $c$ are $(L, 36\alpha, \epsilon_i)$-separated at $u_i$.
	In particular, $\gro oc{u_io} \leq 4 \alpha$.
\end{lemm}

\begin{proof}
	Assume first that $c^* \in \partial X$ is an accumulation point of $(u_io)$.
	According to \autoref{res: admissible sequence - global projection}, $o$ and $u_jo$ are $(L, 16\alpha, \epsilon_i)$-separated at $u_i$, for every $j > i$.
	Passing to the limit, we get that $o$ and $c^*$ are still $(L, 16\alpha, \epsilon_i)$-separated at $u_i$.
	Let us now focus on $c$, which is by definition equivalent to $c^*$.
	Thus $\norm[\infty]{c-c^*} \leq 20\alpha$.
	It follows from (\ref{eqn: gromov product - lip}) that $o$ and $c$ are $(L, 36\alpha, \epsilon_i)$-separated at $u_i$.
	The final conclusion is now a consequence of \autoref{res: bi-gradient line}.
\end{proof}


We now prove the existence of infinite admissible sequences pointing to any point in $\Lambda$.
Our statement is actually slightly more general and provides a tool to extend admissible sequences.

\begin{prop}
\label{res: existence infinite admissible sequence - prelim}
	Let $r, L, M \in \R_+$ with $L > 200\alpha$ and $M > L + r + 100\alpha$.
	Let $\epsilon \in \{\pm 1\}$.
	Let $(u_0, \dots, u_n)$ be a finite $L$-admissible sequence.
	Let $c \in \Lambda$.
	If $u_{n-1}o$ and $c$ are $(M,r, \epsilon)$-separated at $u_n$ then the sequence $(u_0, \dots, u_n)$ extends to an infinite $L$-admissible sequence from $1$ to $[c]$.
\end{prop}

\begin{proof}
	If $c$ belongs to  $\Lambda_-$ (\resp $\Lambda_+$) we let $\epsilon_i = -1$ (\resp $\epsilon_i = +1$) for every integer $i \geq n+1$.
	For simplicity we let $L' = L + 13\alpha$.
	Using \autoref{res: separation triangle cocycles}, we build by induction a sequence $u_{n+1}, u_{n+2}, \dots$ with the following property: for every $i \geq n+1$, the points $u_{i-1}o$ and $c$ are $(L', 2\alpha, \epsilon_i)$-separated at $u_i$.
	Combining our assumption with \autoref{rem: separation at an element} we know that $u_{n-1}o$ and $c$ are also $(L', 2\alpha, \epsilon_n)$-separated at $u_n$, where $\epsilon_n = \epsilon$.
	We claim that $(u_i)$ is $L$-admissible.
	
	For simplicity, we now write $y_i = u_io$, for every $i \in \N$.
	If $i \geq n$ we also let $x_i = u_i \gamma(-\epsilon_i L)$ and $x'_i = u_i \gamma(\epsilon_i L)$.
	It follows from \autoref{res: bi-gradient line} that
	\begin{equation*}
		\gro{y_{i-1}}{x'_i}{x_i}  \leq 2 \alpha, \quad \forall i \geq n.
	\end{equation*}
	Since $(u_0, \dots, u_n)$ is already $L$-admissible, it suffices to prove that 
	\begin{equation*}
		\gro {y_{i+1}}{x_i}{x'_i} \leq 2\alpha, \quad \forall i \geq n.
	\end{equation*}
	Fix $i \geq n$.
	Denote by $\gamma_i \colon \R \to X$ the path given by $\gamma_i(t) = u_i \gamma(\epsilon_i t)$.
	Let $p = \gamma_i(s)$ and $q = \gamma_i(t)$ be respective projections of $y_{i+1}$ and $c$ onto $\gamma_i$ restricted to $\intval 0{L'}$.
	According to \autoref{res: wysiwyg shadow}, $\dist {y_i}q \geq L' - 9\alpha$.
	We claim that $\dist {y_i}p \geq L' - 13\alpha$.
	Without loss of generality, we can assume that $\dist pq > 4\alpha$.
	It follows from \autoref{res: proj cocycle contracting set} that $\gro {y_i}c{y_{i+1}} \geq \dist {y_{i+1}}p - 4\alpha$.
	Recall that $y_i$ and $c$ are $(L', 2\alpha, \epsilon_{i+1})$-separated at $u_{i+1}$.
	Thus
	\begin{equation*}
		\dist {y_i}{y_{i+1}} \geq \dist {y_i}{x_i'} - \gro {y_i}{y_{i+1}}{x'_i} \geq L' -2\alpha
	\end{equation*}
	and $\gro {y_i}c{y_{i+1}} \leq 4 \alpha$ (\autoref{res: bi-gradient line}).
	Consequently $ \dist {y_{i+1}}p \leq 8\alpha$.
	The triangle inequality now yields
	\begin{equation*}
		\dist {y_i}p \geq \dist {y_i}{y_{i+1}} - \dist {y_{i+1}}p \geq L'- 10\alpha.
	\end{equation*}
	which complete the proof of our claim.
	Recall that $L' - 13\alpha \geq L$.
	It follows from the previous discussion that $x_i$, $y_i$, $x'_i$ and $p$ are aligned in this way on $u_i\gamma$.
	Hence $\gro {x_i}{y_{i+1}}{x'_i} \leq 2\alpha$ (\autoref{res: proj contracting set}).
	
	We have shown that $(u_i)$ is an $L$-admissible sequence. 
	We are left to prove that its endpoint is indeed $[c]$.
	Let $j \geq n+1$.
	According to our construction, $c \in \mathcal O_{x_j}(x'_j,2\alpha)$ (see \autoref{rem: separation at an element}).
	Since $(u_i)$ is a $L$-admissible sequence, $(o,u_jo)$ has an $(\alpha, L-18\alpha)$-contracting tail and $o \in \mathcal O_{x'_j}(x_j, 16\alpha)$ (\autoref{res: admissible sequence - global projection}) 
	In other words, $o$ and $c$ are $(L, 16\alpha, \epsilon_j)$-separated at $u_j$. 
	In particular, $c \in \mathcal O_o(y_j, 4\alpha)$.
	As we noticed in the proof of \autoref{res: equivalent accumulation points}, any accumulation point $b \in \partial X$ of $(u_io)$ also belongs to $\mathcal O_o(y_j, 4\alpha)$.
	It follows from \cite[Lemma~5.7]{Coulon:2022tu} that $\norm[\infty]{c-b} \leq 20\alpha$, whence the result.
\end{proof}


\begin{coro}
	Let $L > 200\alpha$.
	For every $c \in \Lambda$, there is a an infinite $L$-admissible sequence from $1$ to $[c]$.
\end{coro}

\begin{proof}
	Fix $M > L + 202\alpha$.
	Let $u_0 = 1$.
	Recall that $c$ belongs to $\Lambda_+ \cup \Lambda_-$.
	According to \autoref{res: separation triangle cocycles} there is $u_1 \in G$ and $\epsilon \in \{\pm 1\}$ such that $o$ and $c$ are $(M,2\alpha, \epsilon)$-separated at $u_1$.
	Note the $(u_0, u_1)$ is an $L$-admissible sequence (the condition to check is void).
	Hence we can extend it to an infinite $L$-admissible sequence ending at $[c]$ (\autoref{res: existence infinite admissible sequence - prelim}).
\end{proof}


%
\subsection{Cylinders}
%

\begin{defi}
	Let $L > 200\alpha$.
	Let $c \in \Lambda$.
	The \emph{$L$-cylinder} of $c$ denoted by $C_L(c)$ is the set of all elements $u \in G$ which belong to an $L$-admissible sequence from $1$ to $[c]$.
\end{defi}

\begin{rema}
	Note that we have chosen $L$ so that all the above statements hold.
	In particular, the endpoints of any $L$-admissible sequence are well-defined.
\end{rema}

Now that we have defined our main tools, namely admissible chains and cylinders, let us comment a little on these definitions.
Recall that $B_G(R)$ stands for the ball of radius $R$ centered at $1$ in $G$ for the metric induced by $X$, see (\ref{eqn: induced ball}).
A crucial fact in our construction is that cylinders behave ``smoothly'' with respect to the topology of $\Lambda \qsim$.
More precisely if $c_1$ and $c_2$ are ``sufficiently close'' (typically if $o$ and $\{c_1,c_2\}$ are separated by a very long contracting segment) we want $C_L(c_1)$ and $C_L(c_2)$ to coincide on $B_G(R)$ for some large $R$ (see \autoref{res: cylinder coincide} below).
This motivates the local condition we used for admissible chains.

A first naive attempt could have been indeed to define the $L$-cylinder of $c_1$ as the set of all $u \in G$ such that $o$ and $c_1$ are  $(L,2\alpha, \epsilon)$-separated at $u$.
Nevertheless, with such a definition, even if $c_2$ is arbitrarily close to $c_1$, there will always be some minor side effects and the best we can hope for is that $o$ anc $c_2$ are $(L- C, 2\alpha, \epsilon)$-separated at $u$ (for some number $C$ depending only on $\alpha$).
Hence $u$ will not belong to the cylinder of $c_2$ for the same parameter $L$.

On the contrary, our definition of an admissible chain $(u_1,\dots, u_n)$ relies on a local condition: for every index $i$, we only require that (points closed to) $u_{i-1}o$ and $u_{i+1}o$ are $(L,2\alpha, \epsilon)$-separated at $u_i$.
This makes the following rerouting argument possible.
Suppose that $(u_n)$ is an $L$-admissible chain between $1$ and $[c_1]$. 
If $c_2$ is sufficiently close to $c_1$ then we can redirect the initial segment of $(u_n)$ to a admissible chain between $1$ and $[c_2]$, i.e. there exists $N \in \N$ and an $L$-admissible sequence $(u'_n)$ from $1$ to $[c_2]$ such that $u_i = u'_i$ for every $i \leq N$.
In particular, $u_1, \dots, u_N$ also belong to the $L$-cylinder of $c_2$.
The next statement formalizes this idea. 

\begin{prop}
\label{res: cylinder coincide}
	Let $r,L, M \in \R_+$ with $L > 200\alpha$ and $M > 2L + r + 500\alpha$.
	Let $\epsilon \in \{\pm 1\}$.
	Let $c_1,c_2 \in \Lambda$.
	Let $g \in G$.
	Suppose that $o$ and $\{c_1,c_2\}$ are $(M,r, \epsilon)$-separated at $g$.
	Then $C_L(c_1)$ and $C_L(c_2)$ coincide on $B_G(R)$ where $R = \dist o{go}$.
\end{prop}

This is probably the most technical proof of this section.
Hence let us give first an overview of the main ideas.
The cylinder $C_L(c)$ of a point $[c] \in \Lambda\qsim$ essentially records all the possible fellow-travels of a ray $\nu$ from $o$ to $c$ with the translates $u\gamma$ of $\gamma$.
If $u_1, u_2 \in C_L(c)$ are such that $u_1^{-1}u_2$ does not belong to $E(h)$, then $u_1\gamma$ and $u_2 \gamma$ cannot overlap much.
Thus if we only focus on maximal fellow-travels between $\nu$ and the translates of $\gamma$, we are inclined to think of those as a totally ordered set.
Consider now the situation handled in the statement.
Since $o$ and $\{c_1, c_2\}$ are separated at $g$ one expects that $g$ belongs to both $C_L(c_1)$ and $C_L(c_2)$, see \autoref{res: existence infinite admissible sequence - prelim}.
Consider now an $L$-admissible sequence  $(u_n)$ between $1$ and $[c_1]$.
Let $j$ be the largest index such that $\dist o{u_jo} \leq \dist o{go}$ (maybe up to some error in terms of $\alpha$).
We would like to say that $u_j$ comes ``before'' $g$.
However we need to distinguish two cases depending whether the Hausdorff distance between $u_j\gamma$ and $g \gamma$ is finite or not, i.e. depending whether $u_j^{-1}g$ belongs to $E(h)$ or not.
\begin{itemize}
	\item When this Hausdorff distance if infinite, then $u_j \gamma$ and $g \gamma$ cannot overlap too much so that $u_j$ truly comes ``before'' $g$ in $C_L(c_1)$.
It allows us to show that $(u_0, \dots, u_j, g)$ is an $L$-admissible sequence, and then extend it with \autoref{res: existence infinite admissible sequence - prelim} to an $L$-admissible sequence from $1$ to $[c_2]$ (see \autoref{fig: rerouting - infinite H dist}).
	We would like to bring the reader's attention to the fact that the definition of admissible sequences with its auxiliary points $z$ and $z'$ was precisely designed so that in this situation the value of the relevant Gromov products does not increase.

	\begin{figure}[htbp]
		\centering
		\labellist
			\small\hair 2pt
			 \pinlabel {$o$} [] at 63 232
			 \pinlabel {[...]} [] at 50 187
			 \pinlabel {$u_{j-1}o$} [l] at 82 144
			 \pinlabel {$u_jo$} [b] at 227 92
			 \pinlabel {$go$} [lb] at 365 35
			 \pinlabel {$c_1$} [] at 542 204
			 \pinlabel {$c_2$} [tl] at 569 210
			 \pinlabel {$g\gamma(\epsilon t_j)$} [t] at 270 17
			 \pinlabel {$g\gamma(\epsilon s_1)$} [tr] at 487 17
			 \pinlabel {$g\gamma(\epsilon s_2)$} [tl] at 492 15
			  \pinlabel {$u_{j-1}\gamma$} [t] at 9 146
			   \pinlabel {$u_j\gamma$} [b] at 303 65
			    \pinlabel {$g\gamma$} [b] at 552 21
		\endlabellist
		\includegraphics[width=0.98\linewidth]{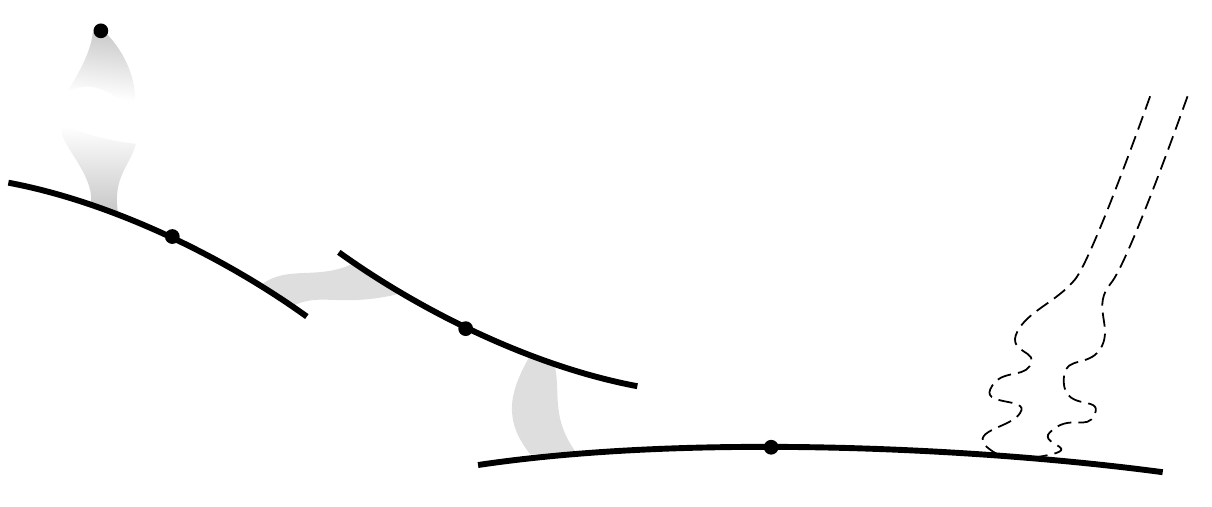}
		\caption{Rerouting an admissible sequence. The case where $u_j^{-1}g$ does not belong to $E(h)$. The thick lines are (portions of) the translates of $\gamma$.}
		\label{fig: rerouting - infinite H dist}
	\end{figure}
	
	\item If the Hausdorff distance between $u_j\gamma$ and $g \gamma$ is finite (hence at most $\alpha$) then it could happen that $u_jo$ and $go$ are too close so that $(u_0, \dots, u_j, g)$ cannot be an admissible $L$-sequence. 
	However our separation assumption is very generous.
	Therefore, we can find $g' \in g \group h$ such that $g'o$ is further away from $u_jo$ along $g \gamma$ and $(u_0, \dots, u_j, g')$ is now an $L$-admissible sequence. 
	We extend it then again to an admissible sequence from $1$ to $[c_2]$ using \autoref{res: existence infinite admissible sequence - prelim} (see \autoref{fig: rerouting - finite H dist}).
	
	\begin{figure}[htb!]
		\centering
		\labellist
			\small\hair 2pt
			 \pinlabel {$o$} [] at 73 242
			 \pinlabel {[...]} [] at 60 195
			 \pinlabel {$u_{j-1}o$} [l] at 82 154
			 \pinlabel {$u_jo$} [b] at 360 47
			 \pinlabel {$go$} [lb] at 378 45
			 \pinlabel {$g'o$} [lt] at 430 20			 
			 \pinlabel {$c_1$} [] at 552 214
			 \pinlabel {$c_2$} [tl] at 574 220
			 \pinlabel {$u_{j-1}\gamma$} [t] at 19 151
			 \pinlabel {$u_j\gamma$} [] at 587 22
			 \pinlabel {$g\gamma$} [b] at 567 36
			  \pinlabel {$g'\gamma$} [t] at 567 16
		\endlabellist
		\includegraphics[width=0.98\linewidth]{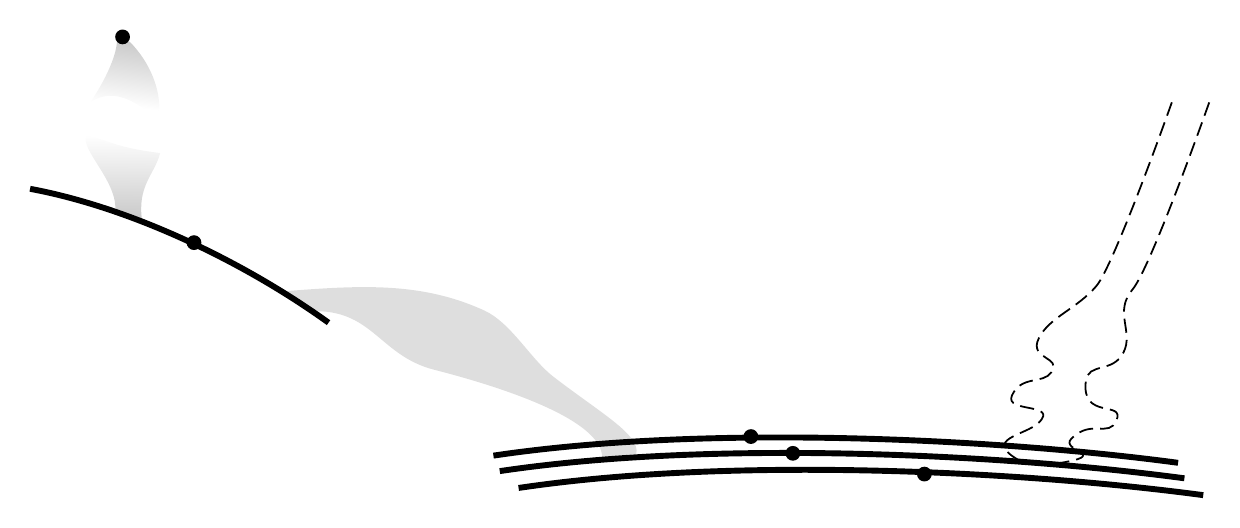}
		\caption{Rerouting an admissible sequence. The case where $u_j^{-1}g$ belongs to $E(h)$. The thick lines are (portions of) the translates of $\gamma$.}
		\label{fig: rerouting - finite H dist}
	\end{figure}
\end{itemize}

\begin{proof}
	We are going to show that $C_L(c_1) \cap B_G(R)$ is contained in $C_L(c_2)$.
	The other inclusion will follow by symmetry.
	Let $u \in C_L(c_1)$ such that $\dist o{uo} \leq \dist o{go}$.
	By definition, $u$ belongs to an infinite $L$-admissible sequence $(u_i)$ from $1$ to $[c_1]$.

	For every $i \in \N$, we write $p_i = g\gamma(\epsilon t_i)$ for a projection of $u_io$ on $g\gamma$.
	Since $o$ and $c$ are $(M, r, \epsilon)$-separated at $g$, we know that $t_0 \leq -M + (r + 2\alpha)$ (\autoref{res: proj contracting set}).
	For $k \in \{1, 2\}$, let $q_k = \gamma(\epsilon s_k)$ be a projection of $c_k$ on $\gamma$ restricted to $\intval{-M}M$.
	It follows from \autoref{res: bi-gradient line} that $s_k \geq M - (r + 7\alpha)$.
	
	\begin{clai}
		Let $i \in \N$.
		\begin{itemize}
			\item If $t_i > t_0 + \alpha$, then $\dist o{u_io} \geq \dist o{go} + t_i - 6\alpha$.
			\item If $t_i < s_1 - 4\alpha$, then $\dist o{u_io} \leq \dist o{go} + t_i + 34\alpha$.
		\end{itemize}
	\end{clai}

	Suppose first that $t_i > t_0 +  \alpha$.
	In particular, $\dist {p_0}{p_i} > \alpha$.
	It follows then from \autoref{res: proj contracting set} that 
	\begin{align*}
		\dist o{u_io} 
		\geq \dist o{p_0} + \dist {p_0}{p_i} - 6\alpha
		& \geq \dist o{p_0} + (t_i - t_0) - 6\alpha \\
		& \geq \dist o{p_0} + \dist {p_0}{go} + t_i - 6\alpha \\
		& \geq \dist o{go} + t_i - 6\alpha.
	\end{align*}
	Suppose now that $t_i < s_1 - 4\alpha$.
	In particular, $\dist {p_i}{q_1} > 4\alpha$.
	Note that 
	\begin{equation*}
		\gro o{p_i}{u_io} + \gro {c_1}{p_0}{p_i}  + \gro{c_1}o{p_0}= \gro o{c_1}{u_io} + \gro{u_io}{c_1}{p_i} + \gro o{p_i}{p_0},
	\end{equation*}	
	(It suffices to unwrap the Gromov products).
	However $\gro{c_1}o{p_0} \geq 0$ while $\gro{u_io}{c_1}{p_i} \leq 2\alpha$ and $ \gro o{p_i}{p_0}\leq 2 \alpha$ (\autoref{res: proj cocycle contracting set}).
	Note also that \autoref{res: wysiwyg shadow} yields
	\begin{equation*}
		 \gro {c_1}{p_0}{p_i} \geq  \max\{0, t_0 - t_i\} - 2\alpha
	\end{equation*}
	(which lower bound is optimal depends on whether $t_i \geq t_0$ or not).
	Consequently, 
	\begin{equation*}
		 \gro o{p_i}{u_io} +  \max\{0, t_0 - t_i\} \leq \gro o{c_1}{u_io} +  6\alpha.
	\end{equation*}
	Since $(u_i)$ is an $L$-admissible sequence ending at $[c_1]$,  the Gromov product on the right hand side is at most $4 \alpha$ (\autoref{res: admissible sequence - global projection}).
	Consequently $t_0 \leq t_i + 10\alpha$.
	Moreover, it follows from the triangle inequality that
	\begin{equation}
	\label{cla: lower bound ui}
		\dist o{u_io} - \dist o{p_i}  \leq \gro o{p_i}{u_io} \leq  10\alpha.
	\end{equation}
	Applying \autoref{res: proj contracting set}, we know that $\gro o{go}{p_0} \leq 2\alpha$.
	Consequently
	\begin{equation*}
		\dist o{p_i} 
		\leq \dist o{p_0} + \dist {p_0}{p_i}
		\leq \dist o{go} - \dist {p_0}{go} + \dist {p_0}{p_i} + 4\alpha,
	\end{equation*}
	Since $t_0 \leq t_i + 10\alpha$, it yields
	\begin{equation*}
		\dist o{p_i} 
		\leq \dist o{go} + t_i + 24\alpha.
	\end{equation*}
	Combined with (\ref{cla: lower bound ui}), it completes the proof of our claim.

	\medskip
	We denote by $j$ the largest integer such that $t_j \leq 40\alpha$.
	Since $(u_io)$ diverges to infinity (\autoref{res: cylinder small gromov product - prelim}) the existence of such an integer $j$ is ensured by our previous claim.
	It also implies that $u = u_i$ for some $i \leq j$.
	We are going to build an element $g' \in g \group h$ such that $u_jo$ and $c_2$ are ``sufficiently separated'' at $g'$ while $(u_0, \dots, u_j, g')$ is still an $L$-admissible sequence.
	Denote by $\epsilon_j$ the orientation of $u_j$ in $(u_i)$.
	For simplicity, we write $\gamma_j \colon \R \to X$ for the path defined by $\gamma_j(t) = u_j \gamma(\epsilon_j t)$.
	Set $x_j = \gamma_j(-L)$ and $x'_j = \gamma_j(L)$.
	We now distinguish two cases.
	
	\subparagraph{Case 1.}
	\emph{Suppose that $u_j^{-1}g$ belongs to $E(h)$.}
	By \ref{enu: alpha - fellow travel} there is $\eta \in \{\pm 1\}$ and $T_j \in \R_+$ such that 
	\begin{equation}
	\label{eqn: cylinder locally coincide - prelim 1-1 - fellow travel 1}
		\dist {\gamma_j(\eta t)}{g\gamma(\epsilon t + \epsilon T_j)} \leq \alpha, \quad \forall t \in \R.
	\end{equation}
	In particular, $\dist{u_jo}{p_j} \leq \alpha$ and thus $\abs{t_j - T_j} \leq 2\alpha$.
	Recall that $q_1 = \gamma(\epsilon s_1)$, which is a projection of $c_1$ on $\gamma$, satisfies $s_1 \geq M - (r +7\alpha)$ while $c_1 \in \mathcal O_{x_j}(x'_j, 36\alpha)$, see \autoref{res: cylinder small gromov product}.
	This forces $\eta = 1$.
	Set $T = t_j + L + 150\alpha$.
	In view of \ref{enu: alpha - translation}, there is $g' \in g\group h$ such that 
	\begin{equation}
	\label{eqn: cylinder locally coincide - prelim 1-1 - fellow travel 2}
		\dist{g'\gamma(t - \epsilon T)}{g\gamma(t)} \leq \alpha, \quad \forall t \in \R.
	\end{equation}
	Let 
	\begin{equation*}
		y = g'\gamma(\epsilon t_j - \epsilon T), 
		\quad \text{and} \quad 
		y' = g'\gamma(\epsilon s_2 - \epsilon T)
	\end{equation*}
	so that $\dist y{p_j} \leq \alpha$ and $\dist {y'}{q_2} \leq \alpha$.
	Recall that $q_2 = g\gamma(\epsilon s_2)$ is a projection of $c_2$ on $g\gamma$.	
	Combining \autoref{res: proj cocycle contracting set} with (\ref{eqn: cylinder locally coincide - prelim 1-1 - fellow travel 2}) we get
	\begin{equation*}
		\gro y{c_2}{y'} \leq
		\gro {p_j}{c_2}{q_2} + 2 \alpha
		\leq 7\alpha.
	\end{equation*}
	Similarly, the triangle inequality yields
	\begin{equation*}
		\gro {u_jo}{y'}y \leq \dist y{u_jo} \leq \dist y{p_j} + \dist{p_j}{u_jo} \leq 2\alpha.
	\end{equation*}
	Thus $u_jo$ and $c_2$ are $7\alpha$-separated by $(y,y')$.
	Note that
	\begin{equation*}
		s_2 - T \geq  M - (L +r +197\alpha)
		\quad \text{and} \quad
		T - t_j \geq L + 150\alpha
	\end{equation*}
	Hence $u_jo$ and $c_2$ are $(L + 120\alpha, 2\alpha, \epsilon)$-separated at $g'$ (\autoref{res: bi-gradient line}).
	Combining (\ref{eqn: cylinder locally coincide - prelim 1-1 - fellow travel 1}) and (\ref{eqn: cylinder locally coincide - prelim 1-1 - fellow travel 2}) we see that 
	\begin{equation*}
		\dist  {\gamma_j(t + T - T_j)}{g'\gamma(\epsilon t)} \leq 2\alpha, \quad \forall t \in \R.
	\end{equation*}
	Note that $T - T_j \geq L$.
	It follows directly that $g'o$ belongs to $\mathcal O_{x_j}(x'_j,2\alpha)$.
	Consequently $(u_0, \dots, u_j, g')$ is an $L$-admissible sequence.	
	
	\subparagraph{Case 2.}
	\emph{Suppose now that $u_j^{-1}g$ does not belong to $E(h)$.}
	According to \autoref{res: cylinder small gromov product}, $o$ and $c_1$ are $(L, 36\alpha, \epsilon_j)$-separated at $u_j$.
	It follows from \autoref{res: simultaneous separation} that $t_j \leq -M + (r + 22\alpha)$.
	Let $m = g\gamma(\epsilon \theta)$ be a projection of $x'_j$ on $g \gamma$.
	Since $u_j\gamma$ has a small projection on $g\gamma$ -- see \ref{enu: alpha - diam proj} -- we get 
	\begin{equation*}
		\theta \leq t_j + \alpha \leq - M + (r + 23\alpha)
	\end{equation*}
	By definition of admissible sequences, there is a point $z' \in B(u_{j+1}o, 10\alpha)$ such that $\gro {x_j}{z'}{x'_j} \leq 2 \alpha$.
	The projection on a contracting set being large scale Lipschitz (\autoref{rem: Lipschitz proj}), the projection $m' = g\gamma(\epsilon \theta')$ of $z'$ on $g\gamma$ satisfies
	\begin{equation*}
		\theta' \geq t_{j+1} - 14\alpha \geq 26\alpha.
	\end{equation*}
	In particular, $x'_j$ and $z'$ are $2\alpha$-separated by $(m,m')$.
	It follows from \autoref{res: bi-gradient line}~\ref{enu: bi-gradient line - neighborhood} that $go$ is $5\alpha$-close to a point $z$ on a geodesic $\nu$ joining $x'_j$ to $z'$.
	Hence the triangle inequality yields
	\begin{equation*}
		\gro {x_j}z{x'_j} \leq \gro {x_j}{z'}{x'_j}  + \gro{x'_j}{z'}z \leq 2 \alpha.
	\end{equation*}
	In other words $z$ belongs to $B(go, 10\alpha) \cap \mathcal O_{x_j}(x'_j, 2\alpha)$.
	Consequently the sequence $(u_0, \dots, u_j, g)$ is $L$-admissible.
	Let us prove now that $u_jo$ and $c_2$ are separated at $g$.
	Recall that $p_j = g\gamma(\epsilon t_j)$ and $q_2 = g\gamma(\epsilon s_2)$ are respective projections of $u_jo$ and $c_2$ on $g\gamma$.
	It follows \autoref{res: proj cocycle contracting set} that $u_jo$ and $c_2$ are $5\alpha$-separated by $(p_j,q_2)$.
	We already noticed that $t_j \leq -M + (r + 22\alpha)$ while $s_2 \geq M - (r + 7\alpha)$.
	Consequently $u_jo$ and $c_2$ are $(L+ 120\alpha, 2\alpha, \epsilon)$-separated at $g$ (\autoref{res: bi-gradient line}).

	\paragraph{Conclusion.}
	In both cases we have produced an element $g' \in g\group h$ such that $(u_0,\dots, u_j,g')$ is an $L$-admissible sequence while $u_jo$ and $c_2$ are $(L + 120\alpha, 2\alpha, \epsilon)$-separated at $g'$.
	According to \autoref{res: existence infinite admissible sequence - prelim}, we can extend $(u_0,\dots, u_j,g')$ to an infinite $L$-admissible sequence from $1$ to $[c_2]$.
	Thus $u$ belongs to $C_L(c_2)$.
\end{proof}


%
\subsection{Visual metric}
%

Let $L > 200\alpha$.
Given  $c_1,c_2 \in \Lambda$, we define an auxiliary quantity $r_L(c_1,c_2)$.
It is the largest number $r \in \R_+$ with the following properties:
\begin{enumerate}
	\item there exists $g \in C_L(c_1) \cap C_L(c_2)$ such that $\dist o{go} = r$, and
	\item $C_L(c_1) \cap B_G(r) = C_L(c_2) \cap B_G(r)$.
\end{enumerate}
If $C_L(c_1) = C_L(c_2)$ we adopt the convention that $r_L(c_1,c_2) = \infty$.
Otherwise, the existence of $r_L(c_1,c_2)$ follows from the fact that the action of $G$ on $X$ is proper.
Roughly speaking $r_L(c_1,c_2)$ is the radius of the largest ball on which $C_L(c_1)$ and $C_L(c_2)$ coincide.


\begin{defi}
\label{def: L-visual pseudo metric}
	Let $L > 200\alpha$.
	The \emph{$L$-visual pseudo-metric} between $c_1,c_2 \in \Lambda$ (seen from $o$) is 
	\begin{equation*}
		\dist[L] {c_1}{c_2} = e^{-r_L(c_1,c_2)}.
	\end{equation*}
\end{defi}

\begin{rema*}
	Observe that $\distV[L]$ satisfies an ultra-metric inequality, namely
	\begin{equation*}
		\dist[L]{c_1}{c_3} \leq \max\left\{ \dist[L]{c_1}{c_2}, \dist[L]{c_2}{c_3}\right\}, \quad \forall c_1,c_2,c_3 \in \Lambda. \qedhere
	\end{equation*}
\end{rema*}

\begin{lemm}
\label{res: metric quotient}
	Let $L > 200\alpha$.
	The visual pseudo-metric $\distV[L]$ on $\Lambda$ induces an ultra-metric distance on $\Lambda\qsim$.
\end{lemm}

\begin{proof}
	By construction the visual pseudo-distance between two cocycles in $\Lambda$ only depends on their respective equivalence classes.
	Hence $\distV[L]$ induces a pseudo-distance on $\Lambda\qsim$.
	The only point left to check is that $\distV[L]$ is positive.
	Let $c_1, c_2 \in \Lambda$ such that $\dist[L] {c_1}{c_2} = 0$.
	By definition this means that $C_L(c_1) = C_L(c_2)$.
	Let $R \in \R_+$ and $g \in C_L(c_1)$ such that $\dist o{go} > R + 17\alpha$.
	Lemmas~\ref{res: admissible sequence - global projection} and~\ref{res: cylinder small gromov product} tell us that $g$ belongs to $\mathcal T_\alpha(L- 18\alpha)$ while both $c_1$ and $c_2$ lie in $\mathcal O_o(go,4\alpha)$.
	It follows from \cite[Lemma~5.7]{Coulon:2022tu} that $\norm[K] {c_1-c_2} \leq 20\alpha$, where $K$ is the closed ball of radius $R$ centered at $o$.
	This fact holds for every $R \in \R_+$, thus $c_1 \sim c_2$.
\end{proof}


\begin{prop}
\label{res: separation for upper bound on distance}
	Let $r,L,M \in \R_+$,  with $L > 200\alpha$ and $M > 2L + r + 500\alpha$.
	Let $\epsilon \in \{\pm 1\}$.
	Let $c_1,c_2\in \Lambda$.
	Let $g \in G$.
	If $o$ and $\{c_1,c_2\}$ are $(M,r, \epsilon)$-separated at $g$, then
	\begin{equation*}
		\dist[L]{c_1}{c_2} \leq e^{-\dist o{go}}.
	\end{equation*}
\end{prop}

\begin{proof}
	Let $R = \dist o{go}$.
	Note that $(1,g)$ is an $L$-admissible sequence (the condition to check is void).
	Using \autoref{res: existence infinite admissible sequence - prelim}, we can extend $(1,g)$ to an infinity $L$-admissible sequence from $1$ to $[c_1]$ (\resp $[c_2]$).
	Thus $g$ belongs to both $C_L(c_1)$ and $C_L(c_2)$.
	
	Applying \autoref{res: cylinder coincide}, we see that $C_L(c_1)$ and $C_L(c_2)$ coincide on $B_G(R)$.
	Moreover this ball contains $g$.
	By definition of the  $L$-visual pseudo-metric we get
	\begin{equation*}
		\dist[L]{c_1}{c_2} \leq  e^{-\dist o{go}}. \qedhere
	\end{equation*}
\end{proof}


\begin{lemm}
\label{res: dense orbit geodesic}
	Let $L > 200\alpha$.
	Let $c \in \Lambda$ and $c' \in \Lambda_+$, which are not equivalent.
	For every $M \in \R$ and $\eta \in \R_+^*$, there is $g \in G$ such that $c'(o,g^{-1}o) \geq M$ and 
	\begin{equation*}
		\max \left\{ \dist[L]{gc}{c_\gamma}, \dist[L]{gc'}{c'_\gamma}\right\} < \eta.
	\end{equation*}
\end{lemm}

\begin{proof}
	Fix $T > 2L + 504\alpha$.
	By \ref{enu: alpha - translation}, there is $n \in \Z$ such that 
	\begin{equation}
	\label{eqn: dense orbit geodesic}
		\dist{h^n\gamma(t)}{\gamma(t + T)} \leq \alpha.
	\end{equation}
	In particular, $\dist o{h^no} \geq T - \alpha$.
	According to \autoref{res: separation triangle cocycles}, there is $g \in G$ such that $c'(o,g^{-1}o) \geq M$ while $c$ and $c'$ are $(2T, 2\alpha, +1)$-separated at $g^{-1}$.
	In particular, $gc'$ and  $c'_\gamma$ both belong to $\mathcal O_o(x',2\alpha)$ where $x' = \gamma(2T)$.
	Combined with (\ref{eqn: dense orbit geodesic}) we get that $o$ and $\{gc',c'_\gamma\}$ are $(T, 4\alpha, \epsilon)$-separated at $h^n$.
	Hence by \autoref{res: separation for upper bound on distance}, 
	\begin{equation*}
		\dist[L]{gc'}{c'_\gamma} \leq e^{-\dist o{h^no}} \leq e^\alpha e^{-T}.
	\end{equation*}
	The same proof gives an upper estimate of $\dist[L]{gc}{c_\gamma}$.
	We conclude by noticing that $T$ can be chosen arbitrary large.
\end{proof}


\begin{lemm}
\label{res: small distance with three cocycles}
	Let $L > 200\alpha$.
	Let $c_1, c_2, c' \in \Lambda$ such that $c_i$ is not equivalent to $c'$.
	For every $\eta \in \R_+^*$, there is $M \in \R_+$ such that for every $g \in G$, if $c'(o,g^{-1}o) \geq M$, then $\dist[L]{gc_1}{gc_2} < \eta$.
\end{lemm}

\begin{proof}
	Without loss of generality we can assume that $\eta \in (0,1)$.
	Fix $T > 2L + 524\alpha$.
	Recall that $\Lambda = \Lambda_- \cup \Lambda_+$.
	Let $\epsilon = -1$ (\resp $\epsilon = +1$)  if $c' \in \Lambda_-$ (\resp $c' \in \Lambda_+)$.
	According to \autoref{res: separation triangle cocycles}, there is $u \in G$ such that $\{c_1,c_2\}$ and $c'$ are $(T, 2\alpha, \epsilon)$-separated at $u$.
	Fix 
	\begin{equation*}
		M = - \ln \eta + c'(o,uo) + T + 10\alpha.
	\end{equation*}
	Let $g \in G$ such that $c'(o,g^{-1}o)\geq M$, that is $c'(uo,g^{-1}o) \geq  - \ln \eta + T +  10\alpha$.
	Denote by $p = u\gamma(\epsilon s)$ and $q = u\gamma(\epsilon t)$ be projection of $g^{-1}o$ and $c'$ on $u\gamma$ restricted to $\intval {-T}T$.
	We know by \autoref{res: bi-gradient line} that $t \geq T - 9\alpha$.
	We claim that $s \geq T - 13\alpha$.
	Indeed, if this is not the case then $\dist pq > 4\alpha$.
	Using \autoref{res: proj cocycle contracting set} we observe
	\begin{align*}
		c'(uo,g^{-1}o)
		\leq c'(uo,q) - c'(g^{-1}o,q)
		& \leq \dist {uo}q - \dist {g^{-1}o}q + 10\alpha\\
		& \leq \dist {uo}q + 10\alpha \\
		& \leq T + 10\alpha,
	\end{align*}
	which contradicts our assumption.
	Note that $\{c_1,c_2\}$  and $g^{-1}o$ are $2\alpha$-separated by $(y,p)$, where $y = u\gamma(-\epsilon T)$.
	Hence $\{c_1, c_2\}$ and $g^{-1}o$ are $(T- 22\alpha, 2\alpha, \epsilon)$-separated at $u$ (\autoref{res: bi-gradient line}).
	Or equivalently $\{gc_1, gc_2\}$ and $o$ are $(T - 22\alpha, 2\alpha, \epsilon)$-separated at $gu$.
	According to \autoref{res: separation for upper bound on distance} we get
	\begin{equation*}
		\dist[L]{gc_1}{gc_2}
		\leq e^{-\dist o{guo}}
	\end{equation*}
	However 
	\begin{equation*}
		\dist o{guo}
		\geq \dist {g^{-1}o}{uo}
		\geq c'(uo,g^{-1}o)
		> - \ln \eta,
	\end{equation*}
	whence the result.
\end{proof}


\begin{lemm}
\label{res: finite Gromov product from visual metric}
	Let $L > 200\alpha$.
	Let $b,b' \in \Lambda$ which are not equivalent.
	There are $r_1, r'_1, \beta \in \R_+^*$ with the following property.
	For every $c,c' \in \Lambda$ such that $\dist[L] bc < r_1$ and $\dist[L]{b'}{c'} < r'_1$, the point $o$ is $\beta$-close to any complete gradient line from $c$ to $c'$.
	Moreover, $\gro c{c'}o \leq \beta$.
\end{lemm}

\begin{proof}
	Fix $M > 2L + 504\alpha$.
	According to \autoref{res: separation triangle cocycles}, there are $\epsilon' \in \{\pm 1\}$ and  $g' \in G$ such that $\{b, o\}$ and $b'$ are $(M, 2\alpha, \epsilon')$-separated at $g'$.
	For simplicity, we write $\gamma_0$ for the path $\gamma$ restricted to $\intval{-M}M$.
	Similarly there are $\epsilon \in \{\pm 1\}$ and  $g \in G$ such that $b$ and $\{o, b'\} \cup g'\gamma_0$ are $(M, 2\alpha, \epsilon)$-separated at $g$.
	We show as in the proof of \autoref{res: separation for upper bound on distance} that $g$ and $g'$ respectively belong to $C_L(c)$ and $C_L(c')$.
	Set
	\begin{equation*}
		x = g\gamma(-\epsilon L), \quad
		y = g\gamma(\epsilon L), \quad
		x' = g'\gamma(-\epsilon'L), 
		\quad \text{and} \quad 
		y' = g'\gamma(\epsilon'L).
	\end{equation*}
	In addition, let by $m = g\gamma(0)$ ans $m' = g'\gamma(0)$.
	Fix also $r_1,r'_1 \in \R_+^*$ such that 
	\begin{equation*}
		r_1 < e^{-\dist o{go}}
		\quad \text{and} \quad 
		r'_1 < e^{-\dist o{g'o}}.
	\end{equation*}
	Consider now $c,c' \in \Lambda$ such that $\dist[L] bc < r_1$ and $\dist[L]{b'}{c'} < r'_1$.
	We claim that $c$ and $c'$ are $(L,36\alpha,\epsilon')$-separated at $g'$.
	It follows from the definition of the visual metric that $g$ belongs to $C_L(b) \cap C_L(c)$.
	Hence $o$ and $\{b,c\}$ are $(L,36\alpha, \epsilon)$-separated at $g$ (\autoref{res: cylinder small gromov product}).
	Thus $\{b, c\}$ and $\{x', y'\}$ are $(L,36\alpha, \epsilon)$-separated at $g$.
	According to \autoref{res: bi-gradient line}, $\gro c{y'}m \leq 4\alpha$ and $\gro b{x'}m \leq 4\alpha$.
	Observe also that 
	\begin{equation*}
		\gro c{y'}{x'} + \gro c{x'}m + \gro b{y'}m = \gro b{y'}{x'} + \gro b{x'}m +  \gro c{y'}m,
	\end{equation*}
	(it suffices to unwrap the Gromov product).
	By construction $b$ and $b'$ are $(M, 2\alpha, \epsilon')$-separated at $g'$ so that $ \gro b{y'}{x'} \leq 2\alpha$.
	It follows then from our previous discussion that $\gro c{y'}{x'} \leq 10\alpha$.
	Similarly, one shows that $g'$ belongs to $C_L(c')$ and therefore $\gro {x'}{c'}{y'} \leq 36\alpha$ (\autoref{res: cylinder small gromov product}).
	This completes the proof of our claim.
	It now follows from \autoref{res: bi-gradient line} that $m$ lies in the $5\alpha$-neighborhood of any complete gradient line $\nu$ from $c$ to $c'$.
	In addition $\gro c{c'}{m'} \leq 4\alpha$.
	Consequently,
	\begin{align*}
		d(o, \nu) & \leq \dist o{m'} + d(m',\nu) \leq \dist o{m'} + 5\alpha, \\
		\gro c{c'}o  &\leq \gro c{c'}{m'} + \dist o{m'} \leq \dist o{m'} + 4\alpha.
	\end{align*}
	Note that the upper bounds do not depend on $c$ nor on $c'$, whence the result.
\end{proof}

The various $L$-visual metric are related as follows.

\begin{lemm}
\label{res: comparing L-visual metric}
	Let $\ell, L \in \R_+$ with $\ell > 200\alpha$ and $L > 2\ell + 600\alpha$.
	For every $c_1,c_2 \in \Lambda$ we have $\dist[\ell] {c_1}{c_2} \leq \dist[L]{c_1}{c_2}$.
\end{lemm}

\begin{proof}
	By definition of the visual metric, there is $g \in C_L(c_1) \cap C_L(c_2)$ such that 
	\begin{equation*}
		\dist[L] {c_1}{c_2} = e^{-\dist o{go}}.
	\end{equation*}
	In particular, there is $\epsilon \in \{\pm 1\}$ such that $o$ and $\{c_1, c_2\}$ are $(L, 36\alpha, \epsilon)$-separated at $g$ (\autoref{res: cylinder small gromov product}).
	It follows then from \autoref{res: separation for upper bound on distance} that 
	\begin{equation*}
		\dist[\ell] {c_1}{c_2} \leq e^{-\dist o{go}} \leq \dist[L] {c_1}{c_2}.  \qedhere
	\end{equation*}
\end{proof}


In general, $(\Lambda\qsim, \distV[L])$ may not be a complete metric space, even if $L$ is sufficiently large.
To bypass this difficulty, we combine all the $L$-visual metrics together as follows.

\begin{defi}
\label{def: visual pseudo metric}
	The \emph{visual pseudo-metric} between $c_1,c_2 \in \Lambda$ (seen from $o$) is 
	\begin{equation*}
		\dist[\infty] {c_1}{c_2} = \sum_{k > 200} 2^{-k} \dist[k\alpha] {c_1}{c_2}
	\end{equation*}
\end{defi}

\begin{rema}
\label{rem: comparing visual and L-visual metrics}
	It follows from \autoref{res: comparing L-visual metric}, that there is $\lambda \in \R_+^*$ such that for every integer $k >  200$ we have
	\begin{equation*}
		2^{-k} \dist[k\alpha] {c_1}{c_2} \leq  \dist[\infty] {c_1}{c_2} \leq \dist[k\alpha] {c_1}{c_2}  + \lambda2^{-k/2}.
	\end{equation*}
	Moreover, according to \autoref{res: metric quotient},  $\distV[\infty]$ induces a metric on $\Lambda \qsim$.
	Note that $\distV[\infty]$ may not be an ultra-metric though.
\end{rema}


\begin{lemm}
\label{res: complete metric space - prelim}
	Let $(c_n)$ be a sequence of cocycles in $\Lambda$.
	Assume that $(c_n)$ is a Cauchy sequence for the visual pseudo-metric $\distV[\infty]$.
	Any accumulation point $c^*$ of $(c_n)$ for the topology of $\partial X$ belongs to $\Lambda$.
	Moreover 
	\begin{equation*}
		\lim_{n \to \infty} \dist[\infty]{c_n}{c^*} = 0.
	\end{equation*}
\end{lemm}

\begin{proof}
	Let $c^*$ be an accumulation point of $(c_n)$ in $\partial X$.
	Let $L > 200\alpha$ and $M > 2L + 600\alpha$.
	Let $T \in \R_+$.
	In view of \autoref{res: comparing L-visual metric}, $(c_n)$ is also a Cauchy sequence for $\distV[M]$.
	Hence there is $n \in \N$, such that
	\begin{equation*}
		\dist[M] {c_n}{c_m} < e^{-T}, \quad \forall m \geq n.
	\end{equation*}
	By definition of the visual metric, there is an element $g \in G$ with $\dist o{go}> T$ such that for every $m \geq n$ we have $g \in C_M(c_n) \cap C_M(c_m)$.
	In particular, $g$ belongs to $U(M - 18\alpha) \subset U(L)$ and $\gro o{c_m}{go} \leq 4 \alpha$, for every $m \geq n$ (Lemmas~\ref{res: admissible sequence - global projection} and~\ref{res: cylinder small gromov product}).
	Passing to the limit, as $m$ tends to infinity, we get $c^* \in \mathcal O_o(go, 4\alpha)$.
	Note that the argument works for every $T \in \R_+$ and every $L > 120\alpha$.
	Hence $c^*$ belongs to $\Lambda_U(G, o, 4\alpha)$ hence to $\Lambda$.
	
	We now prove the metric convergence.
	We keep the same notations as above.
	According to \autoref{res: cylinder small gromov product}, there is $\epsilon \in \{\pm1\}$ such that for all integers $m, m' \geq n$, the point $o$ and $\{c_m,c_{m'}\}$ are $(M, 36\alpha, \epsilon)$-separated at $g$.
	Recall that separation is a closed property, hence $o$ and $\{c_m,c^*\}$ are $(M, 36\alpha, \epsilon)$-separated at $g$.
	It follows now from \autoref{res: separation for upper bound on distance} that
	\begin{equation*}
		\dist[L]{c_m}{c^*} \leq e^{-\dist o{go}} < e^{-T}, \quad \forall m \geq n.
	\end{equation*}
	This argument holds for every $T \in \R_+$. 
	Hence $(c_n)$ converges to $c^*$ for the metric $\distV[L]$.
	This fact holds for every $L > 200\alpha$.
	We conclude from the dominated convergence theorem that $\dist[\infty]{c_n}{c^*}$ converges to zero.
\end{proof}


\begin{prop}
\label{res: complete metric space}
	The quotient $(\Lambda \qsim, \distV[\infty])$ is a complete separable metric space, hence a Polish space.
\end{prop}

\begin{proof}
	Recall that $\partial X$ is compact.
	The completeness of $(\Lambda \qsim, \distV[\infty])$ is thus a direct consequence of \autoref{res: complete metric space - prelim}.
	Let us prove now that this space is separable.
	For every $k \in \N$, let $M_k = (2k + 600)\alpha$.
	For every integer $k > 200$, $\epsilon \in \{\pm 1\}$, and  $g \in G$, we choose, if possible, a cocycle $c_{k,\epsilon, g} \in \Lambda$ such that $o$ and $c_{k,\epsilon, g} $ are $(M_k,2\alpha, \epsilon)$-separated at $g$.
	Denote by $Z$ the set of all cocycles obtained in this way.
	Since $G$ is countable, so is $Z$.
	We claim that the set $Z$ is dense in $(\Lambda \qsim, \distV[\infty])$.
	Denote by $\lambda$ the constant defined in \autoref{rem: comparing visual and L-visual metrics}.
	Consider $c \in \Lambda$ and $\eta \in (0, 1)$.
	Choose an integer $k > 200$ such that $\lambda 2^{-k/2} \leq \eta / 2$ and set $T = - \ln (\eta/2)$.
	Recall that $c \in \Lambda$.
	According to \autoref{res: separation triangle cocycles}, there is $g \in G$ with $\dist o{go} \geq T$ and $\epsilon \in \{\pm 1\}$ such that $o$ and $c$ are $(M_k, 2\alpha, \epsilon)$-separated at $g$.
	Hence $o$ and $\{c, c_{k,\epsilon, g}\}$ are $(M_k, 2\alpha, \epsilon)$-separated at $g$.
	It follows from \autoref{res: separation for upper bound on distance} that 
	\begin{equation*}
		\dist[k\alpha] c{c_{k,\epsilon, g}} \leq e^{-\dist o{go}} \leq \eta/2.
	\end{equation*}
	Combined with our choice of $k$, we get $\dist[\infty]c{c_{k,\epsilon, g}} \leq \eta$, whence the result.
\end{proof}


\begin{lemm}
\label{res: metric top vs quotient top - prelim}
	Let $L > 200\alpha$.
	Let $c \in \Lambda$.
	For every $\eta \in (0,1)$, there is an open neighborhood $U \subset \partial X$ of $c$ such that for every $b \in U \cap \Lambda$, we have $\dist[L] bc < \eta$.
\end{lemm}

\begin{proof}
	Fix $M > 2L + 503\alpha$.
	Recall that $c \in \Lambda_- \cup \Lambda_+$.
	According to \autoref{res: separation triangle cocycles}, there is $\epsilon \in \{\pm 1\}$ such that $o$ and $c$ are $(M, 2\alpha, \epsilon)$-separated at some element $g \in G$ such that $\dist o{go} > - \ln \eta$.
	Let $x = g \gamma(-\epsilon M)$ and $x' = g \gamma(\epsilon M)$.
	The set 
	\begin{equation*}
		U = \set{b \in \partial X}{ \gro xb{x'} < 3\alpha}.
	\end{equation*}
	is an open neighborhood of $c$.
	By construction, $o$ and $U$ are $(M, 3\alpha, \epsilon)$-separated at $g$.
	It follows from \autoref{res: separation for upper bound on distance} that $U \cap \Lambda$ is contained in the ball of radius $\eta$ centered at $c$ with respect to the metric $\distV[L]$.
\end{proof}


\begin{prop}
\label{res: metric top vs quotient top}
	Let $L > 200\alpha$.
	The distances $\distV[\infty]$ and $\distV[L]$ defines the same topology on $\Lambda\qsim$.
	Moreover, it coincides with the quotient topology inherited from the subspace topology on $\Lambda \subset \partial X$.
\end{prop}

\begin{proof}
	For simplicity, let us denote by $\mathcal T$ the quotient topology on $\Lambda \qsim$.
	According to \autoref{res: metric top vs quotient top - prelim}, any open ball for $\distV[L]$ (with $L > 200\alpha$)  is open in $\mathcal T$.
	Hence the natural map $(\Lambda \qsim, \mathcal T) \to (\Lambda \qsim, \distV[L])$ is continuous.
	Let $(z_n)$ be a sequence of $\Lambda\qsim$ converging to $z \in \Lambda\qsim$ for $\mathcal T$.
	We get by dominated convergence theorem that $\dist[\infty]{z_n}z$ converges to zero.
	Hence identity map $(\Lambda \qsim, \mathcal T) \to (\Lambda \qsim, \distV[\infty])$ is continuous.
	
	By construction the map $(\Lambda \qsim, \distV[\infty]) \to (\Lambda \qsim, \distV[L])$ is Lipschitz.
	Hence it suffices to prove that $(\Lambda \qsim, \distV[L]) \to (\Lambda \qsim, \mathcal T)$ is continuous.
	Let $(c_n)$ be a sequence of points in $\Lambda$ and $c \in \Lambda$ such that $\dist[L]{c_n}c$ converges to zero.
	We denote by $c^*$ an accumulation point of $c$ for the topology on $\partial X$.
	We claim that $c^*$ is equivalent to $c$.
	Let $R \in \R_+$.
	There is $N \in \N$ such that for every $n \geq N$, we have 
	\begin{equation*}
		\dist[L]{c_n}c < e^{-(R + 17\alpha)}.
	\end{equation*}
	In particular, there is $g \in C_L(o,c_n) \cap C_L(o,c)$ such that $\dist o{go} > R+17\alpha$.
	Thus $(o,go)$ as an $(\alpha, L-18\alpha)$-contracting tail (\autoref{res: admissible sequence - global projection}) while $c_n$ and $c$ belong to $\mathcal O_o(go,4\alpha)$ (\autoref{res: cylinder small gromov product}).
	Consequently $\norm[K]{c_n - c} \leq 20\alpha$, where $K$ stands for the closed ball of radius $R$ centered at $o$, see for instance \cite[Lemma~5.7]{Coulon:2022tu}.
	This estimate holds for every $n \geq N$.
	Passing to the limit we get $\norm[K]{c^* -c} \leq 20\alpha$.
	This computation does not depend on $R$, hence $\norm[\infty]{c^* - c} \leq 20\alpha$, which completes our claim.
	Since $\Lambda$ is saturated, $c^*$ belongs to $\Lambda$.
	Thus $[c^*] = [c]$ is an accumulation point of $[c_n]$ for the quotient topology $\mathcal T$.
	The above argument has two consequences.
	\begin{enumerate}
		\item \label{enu: metric top vs quotient top - existence}
		Every subsequence of $[c_n]$ admits an accumulation point (for $\mathcal T$).
		\item \label{enu: metric top vs quotient top - uniqueness}
		Every accumulation points of $[c_n]$ (for $\mathcal T$) equals $[c]$.
	\end{enumerate}
	Hence $[c_n]$ converges to $[c]$ in $ (\Lambda \qsim, \mathcal T)$.
\end{proof}


%
\section{Geodesic flow}
\label{sec: geodesic flow}
%

%
\subsection{The spaces}
\label{sec: spaces geodesic flow}
%

We first define the spaces used to study the geodesic flow on $X$.
As previously, $X$ is a proper, geodesic, metric space and $G$ a group acting properly, by isometries on $X$ with a contracting element.
We always assume that $G$ is not virtually cyclic.

\paragraph{The space of geodesics.}
Let $\mathcal GX$ be the space of bi-infinite geodesics $\gamma \colon \R \to X$ parametrized by arc length.
It comes with a flow $(\phi_t)_{t \in \R}$ defined as follows: $\phi_t \gamma$ is the path $s \mapsto \gamma(s +t)$.
The action of $G$ on $X$ induces an action by translation of $G$ on $\mathcal GX$.
One checks that the action of $G$ commutes with the flow.
The \emph{flip involution} $\sigma \colon \mathcal G X \to \mathcal GX$ sends the geodesic $\gamma$, to the geodesic $\bar \gamma$ obtained by reversing the orientation, that is $\bar \gamma(s) = \gamma(-s)$, for every $s \in \R$.
The flip involution commutes with the action of $G$ and anti-commutes with the flow, that is $\sigma \circ \phi_t = \phi_{-t} \circ \sigma$, for every $t \in \R_+$.

\paragraph{The coarse unit tangent bundle.} 
Let us now define a coarse version of the unit tangent bundle of $X$.
We define the ``diagonal'' $\Delta$ of $\partial X$ as
\begin{equation*}
	\Delta
	= \set{(c,c') \in \partial X \times \partial X}{ \gro c{c'}o = \infty}.
\end{equation*}

\begin{defi}
\label{defi: double boundary}
	The \emph{double boundary} of $X$ is the set $\partial^2 X = (\partial X \times \partial X) \setminus\Delta$.
\end{defi}

\begin{rema}
	Note that the definition does not depend on the base point $o \in X$.
	Recall that the map $\partial X \times \partial X \to \R_+ \cup \{\infty\}$ sending $(c,c')$ to $\gro c{c'}o$ is measurable.
	Hence $\Delta$ and $\partial^2X$ are a Borel sets.
\end{rema}

We define on $\partial X \times \partial X$ a \emph{flip involution} $\sigma$ sending $(c,c')$ to $(c',c)$.
It preserves $\partial^2X$.

\begin{defi}
\label{defi: unit tangent bundle}
	The \emph{coarse unit tangent bundle} of $X$, denoted by $SX$, is the product $\partial^2X \times \R$.
\end{defi}

\begin{itemize}
	\item The \emph{geodesic flow} $(\phi_t)_{t \in \R}$ is the flow on $SX$ defined as follows: $\phi_t$ sends $(c, c', s)$ to $(c, c', s + t)$.
	\item The action of $G$ on $\partial X$ induces a diagonal action of $G$ on $\partial ^2 X$.
	This action extends to a \emph{group action} of $G$ on $SX$ as follows:
	given $g \in G$ and $v = (c, c', s)$ in $SX$ we let
	\begin{equation*}
		gv = \left(gc, gc', s + \tau(g,v)\right), \quad \text{where} \quad
		\tau(g,v) = \frac 12 \left[c'(g^{-1}o, o) - c(g^{-1}o,o)\right].
	\end{equation*}
	\item The \emph{flip involution} $\sigma \colon SX \to SX$ sends $(c, c', s)$ to $(c', c,- s)$.
\end{itemize}
One checks that the action of $G$ commutes with the geodesic flow and the flip.
The flip anti-commutes with the geodesic flow.
These actions are continuous.
The projection $SX \to \partial^2X$ sending $(c,c',s)$ to $(c,c')$ is by construction equivariant for the action of $G$ and the flip.

\paragraph{Toward a proper action.}
We now relate the coarse unit tangent bundle $SX$ to the more concrete space of geodesics $\mathcal GX$.
Recall that any bi-infinite geodesic $\gamma \colon \R \to X$, defines two Busemann cocycles $c_\gamma, c'_\gamma \in \partial X$.
We define a map $\mathcal GX \to SX$ by sending $\gamma$ to $v = (c_\gamma, c'_\gamma, s)$ where
\begin{equation*}
	s = \frac 12 \left[c'_\gamma(o, \gamma(0)) - c_\gamma(o,\gamma(0))\right].
\end{equation*}
It follows from the construction that the map $\mathcal GX \to SX$ is equivariant for the actions of the group $G$, the flow and the flip.

If $X$ is CAT(-1), then $\mathcal GX \to SX$ is actually a homeomorphism.
When $X$ is the universal cover of a geodesically complete negatively curved manifold, $\mathcal GX$ can be seen as the unit tangent bundle of $X$ (each unit vector defines a unique bi-infinite geodesic) and the map $\mathcal GX \to SX$ is known as the Hopf coordinates.
Our study is motivated by this analogy. 
In particular, we call \emph{vectors} the elements of $SX$.
Given a vector $v = (c, c', s)$ in $SX$, we call $c$ (\resp $c'$) the \emph{past} (\resp \emph{future}) of $v$.

\begin{rema}
\label{rem: not proper action}
If $X$ is Gromov-hyperbolic, the map $\mathcal GX \to SX$ is not necessarily onto (not every cocycle $c \in \partial X$ is a Busemann cocycle given by a geodesic ray).
However it admits a coarse inverse, see \cite{Bader:2017te,Coulon:2018va}.
In our settings there is a priori no such coarse inverse.
If $X$ is CAT(0), then  $\mathcal GX \to SX$ is not even coarsely injective.

Indeed consider the space $X = \R^2$ endowed with the euclidean metric, so that $\partial X$ is homeomorphic to $S^1$.
Fix $m \in \R$.
For every $s \in \R$, we write $\gamma_s \colon \R \to X$ for the geodesic parametrized by 
\begin{equation*}
\gamma_s(t) = \left( \frac 1{\sqrt{1 + m^2}}t -ms,  \frac m{\sqrt{1 + m^2}}t + s\right).
\end{equation*}
All these geodesics are parallel.
Moreover, their origins lie on a line orthogonal to $\gamma_0$.
Thus the image in $SX$ of the set  $\set{\gamma_s}{s \in \R}$ is reduced to a point.

This observation is the source of another difficulty.
Consider the group $G = \Z^2$ acting by integer translations on $X$.
Let $v_0$ be the image of $\gamma_0$ in $SX$.
If $m$ is irrational, then the closure of $Gv_0$ is $\set{\phi_t(v_0)}{t \in \R}$.
Consequently the action of $G$ on $SX$ is not proper.
\end{rema}

The above remark motivates the following definitions.
A subset $U \subset \bar X \times \bar X$ is \emph{squeezing} if there exists a compact subset $K \subset X$ such that for every pair $(c,c') \in U$, any complete gradient arc from $c$ to $c'$ intersects $K$.
The compact $K$ is called the \emph{bottleneck} of $U$.
According to \autoref{res: continuous extension gromov product}, if $U$ is an open squeezing set, then $\gro c{c'}o$ is finite for every $(c,c')\in U$.

\begin{defi}
\label{def: squeezing}
	Denote by $\partial_0^2X$ the set of pairs $(c,c') \in \partial X \times \partial X$ which admit a squeezing neighborhood in $\bar X\times \bar X$.
	The subspace $S_0X$ consists of all vectors $v  = (c, c', s)$ with $(c,c') \in \partial_0^2X$.
\end{defi}

Note that $\partial^2_0X$, and thus $S_0X$, is non-empty.
Consider indeed a contracting element $h \in G$.
Denote by $c$ (\resp $c'$) an accumulation point of $(h^no)$ as $n$ tends to $\infty$ (\resp $-\infty$).
It follows from  \autoref{res: bi-gradient line} that $(c,c')$ admits a squeezing neighborhood.
We will actually prove later that $S_0X$ has positive measure (\autoref{res: proper subset positive measure}).

According to \autoref{res: continuous extension gromov product} the map $SX \to \R_+ \cup \{\infty\}$ sending $(c,c',s)$ to $\gro c{c'}o$ is continuous when restricted to $S_0X$.
In addition, it follows from the definition that $S_0X$ is an open subset of $\partial X \times \partial X \times \R$.
In particular, $S_0X$ is a Polish space.
It is also invariant under the actions of $G$, the flip map and the flow.

\begin{prop}
\label{res: proper action S0X}
	The action of $G$ on $S_0X$ is properly discontinuous.
\end{prop}

\begin{proof}
	Let $v_1 = (c_1, c'_1, s_1)$ and $v_2 = (c_2,c'_2,s_2)$ be two vectors in $S_0X$.
	One can find a compact subset $K \subset X$ and an open neighborhood $U_i\subset \bar X \times \bar X$ of $(c_i, c'_i)$ such that for every $(c,c') \in U_i$, any geodesic complete gradient arc from $c$ to $c'$ crosses $K$.
	Without loss of generality, we can assume that $K$ contains the base point $o$.
	We write $d$ for its diameter.
	Let $\eta \in \R_+^*$.
	We denote by $V_\eta(v_i) \subset SX$ the open neighborhood of $v_i$ which consists of all vectors $w = (b,b',t)$ in $U_i \times \R$ such that $\abs{s_i - t} < \eta$.
	We claim that the set
	\begin{equation*}
		S = \set{g \in G}{gV_\eta(v_1) \cap V_\eta(v_2) \neq \emptyset}
	\end{equation*}
	is finite.
	Consider $g \in S$.
	By definition there is $w = (b,b',t)$ in $V_\eta(v_1) $ such that $gw \in V_\eta(v_2)$.
	It follows from the triangle inequality that $\abs{\tau(g,w)} < \abs{s_2-s_1} +  2\eta$.
	
	Let $\gamma$ be a complete gradient line from $b$ to $b'$.
	By definition of $U_1$ and $U_2$, we can find two points $p$ and $q$ along $\gamma$ such that $p \in K$ and $q \in g^{-1}K$.
	Recall that $o$ belongs to $K$ whose diameter is $d$.
	In particular $\abs{\dist qp - \dist {g^{-1}o}o} \leq 2d$.
	Moreover, $b$ and $b'$ are $1$-Lipschitz, thus we get
	\begin{equation*}
		\abs{ \tau(g,w)}
		\geq \abs{\frac 12 \left[ b'(g^{-1}o, o) - b(g^{-1}o, o) \right]}
		\geq \abs{\frac 12 \left[ b'(q, p) - b(q, p) \right]} - 2d
	\end{equation*}
	However $p$ and $q$ lie on a complete gradient line from $b'$ to $b$.
	Thus
	\begin{equation*}
		\abs{\frac 12 \left[ b'(q, p) - b(q, p) \right]} = \dist qp.
	\end{equation*}
	Recall that $\abs{\tau(g,w)} < \abs{s_2-s_1} +  2\eta$.
	Combining our observations, we get $\dist o{go} \leq  \abs{s_2-s_1} + 4d + 2\eta$.
	This shows that $S$ is contained in $B_G(R)$ where $R =   \abs{s_2-s_1}+ 4d + 2\eta$.
	Since the action of $G$ on $X$ is proper, the latter set is finite, which complete the proof of our claim.
	Given $\eta \in \R_+^*$, any compact subset of $S_0X$ can be covered by finitely many open subsets of the form $V_\eta(v)$, whence the result.
\end{proof}
\subsection{The Bowen-Margulis measure}
%

Given $x \in X$ and $c,c' \in \partial X$, we define the quantity
\begin{equation*}
	D_x(c,c') = e^{-\gro c{c'}x}.
\end{equation*}
If $X$ is CAT(-1), then $D_x$ is a visual metric on $\partial X$.
However we will not use this property here.
It follows from the definition that for every $c,c' \in \partial X$, for every $x,y \in X$, and $g \in G$, we have 
\begin{align}
	\label{eqn: conformal Bourdon distance}
	D_x(c,c') & = \displaystyle e^{-\frac 12 \left[c(x,y) + c'(x,y)\right]}D_y(c,c'), \\
	\label{eqn: equivariant Bourdon distance}
	D_{gx}(gc,gc') & = D_x(c,c').
\end{align}
The first identity is a consequence of (\ref{eqn: gromov product - conf}).
The second one follows from the fact that the action of $G$ on $X$ preserves Gromov products.

\paragraph{Bowen-Margulis current and measure.}
We endow $\partial X \times \partial X$ with the Borel $\sigma$-algebra which is also the product $\mathfrak B \otimes \mathfrak B$.
We endow $SX$ with the Borel $\sigma$-algebra, denoted by $\mathfrak L$, which is a also the product of $\mathfrak B \otimes \mathfrak B$ restricted to $\partial^2 X$ with the Borel $\sigma$-algebra of $\R$.
Both are $G$-invariant.

\medskip
We fix once and for all $\omega \in \R_+$ and a $G$-invariant, $\omega$-conformal density $\nu = (\nu_x)$ supported on $\partial X$.
Recall that $\omega \geq \omega_G$.
Moreover if $\nu_o$ gives positive measure to the radial limit set, then the action of $G$ on $X$ is divergent and $\omega = \omega_G$ (\autoref{res: quasi-conf + ergo}).

\begin{lemm}
\label{res: measure diagonal}
	Assume that $\nu_o$ gives positive measure to the radial limit set. Then $(\nu_o \otimes \nu_o)(\Delta) = 0$.
\end{lemm}

\begin{proof}
	For every $c \in \partial X$, define the measurable set
	\begin{equation*}
		\Delta_c = \set{c' \in \partial X}{ \gro c{c'}o = \infty}.
	\end{equation*}
	It follows from Fubini theorem that 
	\begin{equation*}
		(\nu_o \otimes \nu_o)(\Delta) = \int \nu_o(\Delta_c) d\nu_o.
	\end{equation*}
	Hence it suffices to prove that $\nu_o(\Delta_c) = 0$, for $\nu_o$-almost every $c \in \partial X$.
	
	Since $\nu_o$ gives positive measure to the radial limit set, it actually gives full measure to the contracting limit set (\autoref{res: quasi-conf + ergo}).
	Let $c \in \Lambda_{\rm ctg}(G)$.
	We are going to prove that $\Delta_c$ is contained in the equivalence class $[c]$.
	According to  \cite[Proposition~5.9]{Coulon:2022tu} we can find $\alpha, r \in \R_+^*$ and $L > r + 32\alpha$ such that $c$ belongs to $\Lambda_{\mathcal T_\alpha}(G, o, r, L)$.
	In particular, there is a sequence of elements $(g_n)$ in $\mathcal T_\alpha(L)$ such that $(g_no)$ diverges to infinity and $c \in \mathcal O_o(g_no,r)$, for every $n \in \N$.
	
	Let $c' \in \Delta_c$.
	Let $n \in \N$.
	Denote by $\tau_n$ the $\alpha$-contracting tail associated to $g_n$.
	By definition of $\mathcal T_\alpha$, there is a projection $p_n$ of $o$ on $\tau_n$ such that $\dist {g_no}{p_n} \geq L$.
	Denote by $q_n$ and $q'_n$ respective projections of $c$ and $c'$ on $\tau_n$.
	According to \autoref{res: wysiwyg shadow}, $\dist {g_no}{q_n} \leq r + 7\alpha$.
	We claim that $\dist {g_no}{q'_n} \leq r + 14 \alpha$.
	Assume on the contrary that it is not the case.
	In particular $\dist {q_n}{q'_n} > 7\alpha$.
	It follows from \autoref{res: proj two cocycle contracting set} that $\gro c{c'}{q_n} \leq 5\alpha$. 
	In particular, $\gro c{c'}o < \infty$.
	This  contradicts our assumption and completes the proof of our claim.
	Note that $\dist {p_n}{q'_n} > 4\alpha$.
	According to \autoref{res: proj cocycle contracting set}, $q'_n$ is $5\alpha$-close to any complete gradient line from $o$ to $c'$, hence $\gro o{c'}{q'_n} \leq 5\alpha$.
	It follows from the triangle inequality that $\gro o{c'}{g_no} \leq r + 19\alpha$.
	
	We have proved that for every $n \in \N$, the cocycle $c'$ belongs to $\mathcal O_o(g_no, r+ 19\alpha)$, as does $c$.
	Since $\dist o{g_no}$ diverges to infinity, this forces that $\norm[\infty]{c - c'} <20\alpha$, see \cite[Lemma~5.7]{Coulon:2022tu}, as we announced.
	Recall that the measure $\nu_o$ restricted to $\mathfrak B^+$ has no atom (\autoref{res: quasi-conf + ergo}).
	Hence $\nu_o(\Delta_c) = \nu_o([c]) = 0$.
\end{proof}

\begin{defi}
\label{def: bm measure}
	The \emph{Bowen-Margulis current} (associated to $\nu$), which we denote by $\mu$, is the measure on $(\partial^2X, \mathfrak B \otimes \mathfrak B)$ defined by 
	\begin{equation*}
		\mu = \frac 1{D_o^{2\omega}} \nu_o \otimes \nu_o.
	\end{equation*}
	The \emph{Bowen-Margulis measure} is the measure $m$ on $(SX, \mathfrak L)$ defined by 
	\begin{equation*}
		m = \mu \otimes ds,
	\end{equation*}
	where $ds$ stands for the Lebesgue measure on the factor $\R$ of $SX$.
\end{defi}

\begin{lemm}
\label{res: bm current G invariant}
	The Bowen-Margulis current $\mu$ is $G$-invariant, for the diagonal action of $G$ on $(\partial^2X, \mathfrak B \otimes \mathfrak B)$.
	In particular, the measure $m$ on $(SX, \mathfrak L)$ is invariant under the action of $G$ as well as the flip map and the geodesic flow.
	Moreover $\mu$ and $m$ are $\sigma$-finite measures on $\mathfrak B \otimes \mathfrak B$ and $ \mathfrak L$ respectively.
\end{lemm}

\begin{proof}
	On the one hand, the measure
	\begin{equation*}
		\frac 1{D_x^{2\omega}} \nu_x \otimes \nu_x.
	\end{equation*}
	does not depend on the point $x \in X$.
	This is indeed of a consequence of (\ref{eqn: conformal Bourdon distance}) together with the conformality of $\nu$.
	On the other hand, the density $\nu$ is $G$-invariant, while $D_x$ is $G$-equivariant, see (\ref{eqn: equivariant Bourdon distance}).
	Consequently $\mu$ is $G$-invariant.
	
	By construction the map $D_o^{2\omega} \colon \partial^2 X \to \R_+$ is positive on $SX$.
	In addition, it belongs to $L^1(\partial^2 X, \mathfrak B \otimes\mathfrak B, \mu)$.
	Thus $\mu$ is $\sigma$-finite.
	The Lebesgue measure on $\R$ is $\sigma$-finite, hence so is $m$.
\end{proof}

As we noticed in \cite{Coulon:2022tu} the ergodic properties of the action of $G$ on $(\partial X, \mathfrak B, \nu_o)$ are not satisfactory.
In order to recover a behavior inspired by negative curvature, one needs to restrict ourselves to $(\partial X, \mathfrak B^+, \nu_o)$ where $\mathfrak B^+$ is the reduced sub-$\sigma$-algebra of $\mathfrak B$.
A similar shift is needed for $SX$.
This is a bit subtle though.
A first attempt would be to consider the sub-$\sigma$-algebra $\mathfrak L^0$ defined as the product $\mathfrak B^+ \otimes \mathfrak B^+$ with the Borel $\sigma$-algebra of $\R$.
Unfortunately, the map $c \mapsto c(g^{-1}o,o)$ is not $\mathfrak B^+$-measurable, hence the action of $G$ does not preserve $\mathfrak L^0$.
It motivates the next definition.

\begin{defi}
\label{def: admissible sigma-alg}
	We say that a sub-$\sigma$-algebra $\mathfrak L^+ \subset \mathfrak L$ is \emph{admissible} if it is $G$-invariant and satisfies the following properties: 
	for every $B \in \mathfrak B^+ \otimes \mathfrak B^+$ the set $B \times \R$ belongs to $\mathfrak L^+$;
	for any flow invariant set $A \in \mathfrak L^+$ there is $B \in \mathfrak B^+ \otimes \mathfrak B^+$ such that the symmetric difference between $A$ and $B \times \R$ has $m$-measure zero.
\end{defi}

\begin{exam}
\label{exa: admissible alg}
	Let us give some examples of admissible sub-$\sigma$-algebras.
	\begin{enumerate}
		\item The easiest (but not very useful) option is to take for $\mathfrak L^+$ all sets of the form $B \times \R$, where $B$ runs over $\mathfrak B^+ \otimes \mathfrak B^+$.
		
		\item A second choice is to consider all the sets $A \in \mathfrak L$ such that $gA \in \mathfrak L^0$, for every $g \in G$. 
		
		\item \label{enu: admissible alg - div}
		The last example is probably the most useful one. 
		However it has the disadvantage that the construction already assumes that the Poincaré series of $G$ diverges at $\omega$.
		In order to avoid any ambiguity, we denote by $\nu^+ = (\nu^+_x)$ the restriction of the density $\nu$ to $\mathfrak B^+$.
		Recall that $\nu^+_x \ll \nu^+_y$ for every $x,y \in X$.
		This allows us to define a $\mathfrak B^+$-measurable map $\beta \colon G \times \partial X \to \R$ by 
		\begin{equation}
		\label{eqn: def rn cocycle}
			\beta(g,c) = - \frac 1 \omega \ln \frac{d\nu^+_{g^{-1}o}}{d\nu^+_o} (c).
		\end{equation}	
		It satisfies the following cocycle relation 
		\begin{equation}
		\label{eqn: cocycle relation rn}
			\beta(g_1, g_2 c) - \beta(g_1g_2,c) + \beta(g_2, c) = 0, \quad \forall g_1,g_2 \in G,  \ \forall c \in \partial X \ \nu^+\text{-a.e.}
		\end{equation}
		Observe that the map $G \times \partial X \to \R$ sending $c$ to $c(g^{-1}o, o)$ also satisfies the cocycle relation (\ref{eqn: cocycle relation rn}).
		Suppose now that $\nu$ gives full measure to the contracting limit set.
		It follows from  \cite[Proposition~5.10]{Coulon:2022tu}, that there is $C$ such that for every $g \in G$, for $\nu$-a.e. $c \in \partial X$, 
		\begin{equation*}
			\abs{\beta(g,c) - c(g^{-1}o, o)} \leq C.
		\end{equation*}
		However any bounded measurable cocycle is a coboundary.
		More precisely there is a $\mathfrak B$-measurable function $f \colon \partial X \to \R$ such that for every $g \in G$ and $\nu$-a.e. $c \in \partial X$, we have 
		\begin{equation*}
			\beta(g,c) = c(g^{-1}o, o) + f(gc) - f(c).
		\end{equation*}
		Consider the map 
		\begin{equation*}
			\begin{array}{rccc}
				\chi \colon & SX& \to & SX \\
				& (c,c', s) & \mapsto & \left(c,c', s + f(c) \right)
			\end{array}
		\end{equation*}
		It is a flow equivariant, measure preserving automorphism of $(SX, \mathfrak L, m)$.
		Observe that for $m$-a.e. $v = (c,c',s)$ in $SX$, for every $g \in G$ we have
		\begin{equation*}
			\chi^{-1} \circ g \circ \chi(v) = \left(gc,gc', s + \beta(g,c')\right).
		\end{equation*}
		Since $\beta$ is $\mathfrak B^+$-measurable, the pull back of $\mathfrak L^0$ by $\chi$, which we denote by $\mathfrak L^1$, is $G$-invariant, up to measure zero.
		We define $\mathfrak L^+$ as the collection of all sets $A \in \mathfrak L$ whose symmetric difference with some element of $\mathfrak L^1$ has $m$-measure zero. 
		By construction it is admissible. \qedhere
	\end{enumerate}
\end{exam}

We now fix for the rest of the article an admissible sub-$\sigma$-algebra $\mathfrak L^+ \subset \mathcal L$.
In practice, it will not really play a role before the very end of the proof of the Hopf-Tsuji-Sullivan dichotomy (see the proof of \autoref{res: hopf-tsuji-sullivan}).

\begin{nota}
	For every $x,x' \in X$, for every $r \in \R_+$, we let 
	\begin{equation}
	\label{eqn: def double shadow contracting element}
		A(x,x',r) = \mathcal O_{x'}(x,r) \times \mathcal O_x(x', r).
	\end{equation}
	In plain words, $A(x,x',r) \subset \bar X \times \bar X$ is the set of pairs $(c,c')$ such that $c$ and $c'$ are $r$-separated by $(x,x')$.
\end{nota}

\begin{lemm}
\label{res: measure product shadows}
	Let $h \in G$ be a contracting element.
	There is $r_0 \in \R_+$ with the following property.
	For every $r \geq r_0$, there is $N \in \N$ such that for every integer $n \geq N$, the following holds:
 	\begin{itemize}
		\item $A(h^{-n}o, h^no,r)$ is contained in an open squeezing subset;
		\item  $0 < \mu\left(A(h^{-n}o, h^no,r) \cap \partial^2 X\right) < \infty$.
	\end{itemize}
\end{lemm}

\begin{proof}
	Let $\epsilon, r_0 \in \R_+^*$  be the parameters given by the Shadow Lemma (\autoref{res: shadow lemma}).
	Since $h$ is contracting, there is $\alpha \in \R_+$ such that for every $n \in \N$, any geodesic from $h^{-n}o$ to $h^no$ is $\alpha$-contracting.
		Choose $r \geq r_0$.
	We fix $N \in \N$ such that $\dist o{h^no} > 2r + 100\alpha$, for every $n \geq N$.
	Let $n \geq N$.
	According to \autoref{res: bi-gradient line},   $A(h^{-n}o, h^no,r)$ is contained in an open squeezing subset.
	If $K$ stands for the corresponding bottleneck, then for every $(c,c') \in A(h^{-n}o, h^no,r)$, there is $z \in K$ such that $\gro c{c'}z = 0$ (\autoref{res: continuous extension gromov product}).
	Thus, $\gro c{c'}o$ is uniformly bounded from above on $A(h^{-n}o, h^no,r) \cap \partial^2 X$, which therefore has finite $\mu$-measure.
	It remains to prove that $A(h^{-n}o, h^no,r)$ has positive measure.
	It follows from the Shadow Lemma that 
	\begin{equation*}
		\nu_o \left(\mathcal O_{h^no}(h^{-n}o,r) \right)
		\geq e^{-\omega \dist o{h^no}} \nu_{h^no} \left(\mathcal O_{h^no}(h^{-n}o,r) \right) \geq \epsilon e^{-3\omega \dist o{h^no}} > 0.
	\end{equation*}
	Similarly 
	\begin{equation*}
		\nu_o \left(\mathcal O_{h^{-n}o}(h^no,r) \right)
		\geq \epsilon e^{-3\omega \dist o{h^no}} > 0,
	\end{equation*}
	whence the result.
\end{proof}


\begin{lemm}
\label{res: proper subset positive measure}
	The set $S_0X$ has positive $m$-measure.
	If in addition, $\nu_o$ gives full measure to the radial limit set, then $S_0X$ contains a subset $B \in \mathfrak L^+$ with full $m$-measure wich is flow, group and flip invariant.
\end{lemm}

\begin{proof}
	Let $h \in G$ be a contracting element.
	According to  \autoref{res: measure product shadows}, there are $r, N \in \R_+$ such that for every integer $n \geq N$, the set $A(h^{-n}o, h^no,r)$ defined in (\ref{eqn: def double shadow contracting element}) is contained in $\partial^2_0X$ and has positive $\mu$-measure.
	Hence $S_0X$ has positive $m$-measure.
	
	Suppose now that $\nu_o$ gives full measure to the radial limit set.
	In particular, $\omega = \omega_G$ and the action of $G$ on $X$ is divergent.
	Denote by $\Lambda = \Lambda_h$ the associated radial limit set studied in \autoref{res: radial limit set contracting elt}.
	For simplicity we let $\Lambda^{(2)} = (\Lambda \times \Lambda) \cap \partial^2 X$.
	According to \autoref{res: preferred limit set full measure}, $\Lambda$ has full $\nu_o$-measure.
	Recall that $\Lambda$ is saturated.
	Hence $B = \Lambda^{(2)} \times \R$ is a subset in $\mathfrak L^+$ with full $m$-measure.
	By construction it is flow, group and flip invariant.
	Combining Propositions~\ref{res: separation triangle cocycles} and \ref{res: bi-gradient line} we observe that $B$ is contained in $S_0X$, whence the result.
\end{proof}


%
\subsection{Quotient space}
\label{sec: quotient space}
%

Denote by $\mathfrak L_G$, (\resp $\mathfrak L^+_G$) the sub-$\sigma$-algebra of $\mathfrak L$ (\resp $\mathfrak L^+$) which consists of all $G$-invariant subsets of $\mathfrak L$ (\resp $\mathfrak L^+$).
Note that $\mathfrak L^+_G \subset \mathfrak L_G$.
Define a $G$-invariant map $\xi \colon SX \to \intval 01$ as follows
\begin{equation}
\label{eqn: def stab map}
	\xi(v) = \card{G_v}^{-1}, \quad \forall v \in SX,
\end{equation}
where $G_v$ is the set of elements in $G$ fixing $v$.
As the group $G$ is countable and acts by homeomorphisms on $SX$, the map $\xi$ is $\mathfrak L_G$- measurable.
Since the action of $G$ on $S_0X$ is proper, $\xi$ does not vanish on $S_0X$.
For the same reason, every $G$-orbit of a vector $v \in S_0X$ is closed (for the induced topology on $S_0X$).
Thus we can find a measurable fondamental domain for the action of $G$ on $S_0X$, that is a set $D \in \mathfrak L$, contained in $S_0X$, whose intersection with any $G$-orbit in $S_0X$ contains exactly one point, see Bourbaki \cite[Chapitre~IX, §6.9, Théorème~5]{Bourbaki:1974wv}.
Note that in general $D$ does not belong to $\mathfrak L^+$.
We now endow $(SX, \mathfrak L_G)$ with a measure $m_G$ defined as
\begin{equation*}
	m_G =  \xi \mathbb 1_D m.
\end{equation*}
The computation shows that 
\begin{equation*}
	\mathbb 1_{S_0X} = \xi \sum_{g \in G} \mathbb 1_{gD}.
\end{equation*}
Together with the fact that $m$ is $G$-invariant, it implies that $m_G$ does not depend on the choice of $D$.
It follows also from this relation that 
\begin{equation*}
	m_G(GB) \leq m(B \cap S_0X) \leq m(B), \quad \forall B \in \mathfrak L.
\end{equation*}
Since $m$ is a $\sigma$-finite measure on $\mathfrak L$, this implies that $m_G$ is $\sigma$-finite on $\mathfrak L_G$ as well.
Recall that $G$ is countable.
Hence 
\begin{equation*}
	m_G(B) = 0  \iff  m(B\cap S_0X) = 0, \quad \forall B \in \mathfrak L_G.
\end{equation*}
In particular, $m_G$ is not the zero measure (\autoref{res: proper subset positive measure}).

One checks that if $\psi \colon SX \to SX$ is homeomorphism commuting with $G$ and preserving $S_0X$ as well as $m$, then it preserves also $m_G$.
In particular, $m_G$ is both flip and flow invariant.
We now think of the action of $(\phi_t)$ on $(SX, \mathfrak L^+_G, m_G)$ as the geodesic flow on the quotient $SX/G$.

\begin{rema}
\label{rem: mG supported on S0X}
	By construction $m_G$ is supported on $S_0X$.
	We have to keep in mind though that $S_0X$ does not a priori belong to $\mathfrak L^+_G$.
	However if the action of $G$ on $X$ is divergent, then $S_0X$ contains a flow/flip invariant subset $B \in \mathfrak L^+_G$ with full $m$-, hence $m_G$-measure (\autoref{res: proper subset positive measure}).
\end{rema}


%
\subsection{Conservativity of the geodesic flow}
%

\paragraph{Conservative system.}
We make here a little digression about conservative / dissipative systems.
Let $(\Omega, \mathfrak A, \mu)$ be a measure space.
In this paragraph, we denote by $(H, dh)$ either a discrete countable group endowed with the counting measure, or $\R$ endowed with the Lebesgue measure.
We assume that $H$ acts on $(\Omega, \mathfrak A)$ preserving the measure class of $\mu$.
In particular, the action is mesurable in the sense that the map $H \times \Omega \to \Omega$ sending $(h, \omega)$ to $h\omega$ is measurable.
A subset $W \subset \Omega$ is called \emph{wandering} if for $\mu$-almost every $\omega \in W$, we have
\begin{equation*}
	\int_H \mathbb 1_W \left(h \omega\right) dh < \infty.
\end{equation*}
The action of $H$ is \emph{conservative} if there is no measurable wandering set with positive measure.
In this case, for every  $A\in \mathfrak A$, for $\mu$-almost every $\omega \in A$,
\begin{equation*}
	\int_H \mathbb 1_A (h\omega) dh = \infty.
\end{equation*}
The action of $H$ is \emph{dissipative} if $\Omega$ is, up to measure zero, a countable union of wandering sets.

\begin{lemm}
\label{res: criterion conservativity}
	Assume that the action of $H$ preserves the measure $\mu$.
	If there is summable map $f \colon \Omega \to \R_+$ such that 
	\begin{equation*}
		\int_H f(h \omega)dh= \infty, \quad \mu\text{-a.e.},
	\end{equation*}
	then the action of $H$ is conservative.
\end{lemm}

\begin{proof}
	Let $f_1, f_2 \colon \Omega \to \R_+$ be two funtions.
	Since $H$ is unimodular and preserves the measure $\mu$, we have
	\begin{equation*}
		\int_\Omega \left( \int_H f_1 \circ h  dh \right)  f_2 d\mu
		= \int _\Omega f_1 \left( \int_H  f_2 \circ h dh \right)d\mu.
	\end{equation*}
	Let $W \in \mathfrak A$ be a wandering set.
	We use the above relation with $f_1 = f$ and
	\begin{equation*}
		f_2(\omega) = \frac 1{c(\omega)} \mathbb 1_W(\omega)
		\quad \text{where} \quad 
		c(\omega) = \max \left\{ 1, \int_H \mathbb 1_W \left(h \omega\right) dh \right\}
	\end{equation*}
	Since $W$ is wandering $f_2(\omega) > 0$ for $\mu$-almost every $\omega \in W$.
	Observe also that $c \colon \Omega \to \R_+$ is $H$-invariant, thus 
	\begin{equation*}
		\int_H  f_2 (h\omega) dh \leq 1, \quad \forall \omega \in \Omega.
	\end{equation*}
	Consequently we get 
	\begin{equation*}
		\int_\Omega \left( \int_H f_1 \circ h  dh \right) f_2 d\mu
		\leq \int_\Omega f_1 d\mu 
		<\infty.
	\end{equation*}
	It follows from our assumption that $f_2$ vanishes $\mu$-a.e., hence $\mu(W) = 0$.
\end{proof}

If the measure $\mu$ is $\sigma$-finite, then $\Omega$ admits a \emph{Hopf decomposition}, that is a partition $\Omega = \Omega_C \sqcup \Omega_D$ which is $H$-invariant (up to measure zero) so that the action of $H$ on $\Omega_C$ (\resp $\Omega_D$) is conservative (\resp dissipative) \cite{Hopf:1937kk}.
This partition is unique up to measure zero.
The sets $\Omega_C$ and $\Omega_D$ are respectively called the \emph{conservative} and \emph{dissipative parts} of $\Omega$.

\paragraph{Application to the geodesic flow.}
We now come back to the study of the geodesic flow.

\begin{prop}
\label{res: conservativity geodesic flow}
	If $\nu_o$ gives full measure to the radial limit set, then
	\begin{enumerate}
		\item the diagonal action of $G$ on $(\partial X \times \partial X, \mathfrak B \otimes \mathfrak B, \nu_o \otimes \nu_o)$ is conservative, 
		\item
		the geodesic flow on $(SX, \mathfrak L_G, m_G)$ is conservative.
	\end{enumerate}
\end{prop}

\begin{rema}
	Recall that $\mathfrak L^+_G$ is a sub-$\sigma$-algebra of $\mathfrak L_G$.
	Hence if the geodesic flow on $(SX, \mathfrak L_G, m_G)$ is conservative, then so is the one on $(SX, \mathfrak L^+_G, m_G)$.
	We will not use this property though.
\end{rema}

\begin{proof}
	Let $h \in G$ be a contracting element.
	Denote by $\Lambda = \Lambda_h$ the corresponding limit set studied in \autoref{res: radial limit set contracting elt}.
	According to \autoref{res: measure product shadows} there is $r_0 \in \R_+^*$ with the following property: for every $r \geq r_0$, there is $N(r) \in \N$, such that for every integer $n \geq N(r)$, the set $A(h^{-n}o, h^no,r)$ is squeezing, contained in $\partial^2_0X$, and has positive, finite $\mu$-measure.
	
	In view of \autoref{res: separation triangle cocycles}, we can find $r \geq r_0$, such that the following holds: if $c \in \partial X$ and $c' \in \Lambda$ are two non-equivalent cocycles, then for every $T \in \R_+$, for every $n \in \N$,  there is $u \in G$ such that $c(o,uo) \geq T$ and $(c,c') \in uA(h^{-n}o, h^no,r)$.
	We now fix $n \geq N(r)$ and let
	\begin{equation*}
		A = A(h^{-n}o, h^no,r)	
	\end{equation*}
	Since $A(h^{-n}o, h^no,r)$ is squeezing, there is $\beta \in \R_+$ such that for every $(c,c') \in A$, we have $\gro c{c'}o \leq \beta$.
	Let $B = A \times \intval {-\beta}{\beta}$.
	By construction, $A$ belongs to $\mathfrak B \otimes \mathfrak B$ and $GB \in \mathfrak L_G$.
	Since $\mu(A)$ is finite, so are $m(B)$ and $m_G(GB)$.
	
	\begin{clai}
	\label{cla: conservativity prelim}
		Let $v = (c,c',s)$ be a vector in $SX$.
		If $c,c' \in \Lambda$, then
		\begin{equation*}
		\label{eqn: conservativity prelim}
			\sum_{g \in G} \mathbb 1_A(gc,gc') = \infty
			\quad \text{and} \quad
			\int_\R \mathbb 1_{G B} \circ \phi_t(v) dt = \infty.
		\end{equation*}
	\end{clai}

	The value of the above integral is flow-invariant, hence we can assume that $s = -\gro c{c'}o$.
	Thus, for every $g \in G$, we have $g\phi_t(v) = (gc,gc',t + \delta(g))$ where 
	\begin{equation*}
		\delta(g) = - \gro c{c'}o + \tau(g,v) = -c'(o,g^{-1}o) - \gro c{c'}{g^{-1}o}.
	\end{equation*}
	Let $T \in \R_+$.	
	According to our choice of $r$, there is $u \in G$ such that $c'(o, uo) \geq T$ and  $(c, c') \in u A$.
	In particular, $\gro c{c'}{uo} \leq \beta$.
	Hence, $u^{-1}\phi_t(v)$ belongs to $B$, whenever
	\begin{equation*}
		0 \leq t -  c'(o,uo) \leq \beta.
	\end{equation*}
	Consequently, 
	\begin{equation*}
		\mathbb 1_A(u^{-1}c, u^{-1}c') = 1 
		\quad \text{and} \quad 
		\int_{c'(o,uo)}^{c'(o,uo)+\beta} \mathbb 1_{G B} \circ \phi_t(v) dt  = \beta.
	\end{equation*}
	The construction holds for arbitrary $T$, hence $c'(o,uo)$  can be chosen arbitrarily large in the above identity.
	Consequently, 
	\begin{equation*}
		\sum_{g \in G} \mathbb 1_A(gc,gc') = \infty
			\quad \text{and} \quad
		\int_\R \mathbb 1_{G B} \circ \phi_t(v) dt = \infty.
	\end{equation*}
	This proves our claim.

	Assume now that $\nu_o$ gives full measure to the radial limit set.
	It forces that $\omega = \omega_G$.
	Moreover the action of $G$ on $X$ is divergent.
	Consequently, the set $\Lambda$ has full $\nu_0$-measure (\autoref{res: preferred limit set full measure}).
	According to previous claim,
	\begin{enumerate}
		\item \label{enu: conservativity prelim - bdy}
		$\displaystyle \sum_{g \in G} \mathbb 1_A(gc,gc') = \infty$, for $\mu$-a.e. $(c,c') \in \partial^2X$.
		\item $\displaystyle\int_\R \mathbb 1_{G B} \circ \phi_t(v) dt = \infty$, for $m_G$-a.e. $v \in SX$.
	\end{enumerate}
	It follows from the conservativity criterion proved in \autoref{res: criterion conservativity} that the flot $\phi_t$ on $(SX, \mathfrak L, m_G)$ and the action of $G$ on $(\partial^2X, \mathfrak B \otimes \mathfrak B, \mu)$ are conservative.
	Note that $\mu$ and $\nu_o \otimes \nu_o$ restricted to $\partial^2X$ are in the same measure class.
	Hence the action of $G$ on $(\partial^2X, \mathfrak B \otimes \mathfrak B, \nu_o \otimes \nu_o)$ is conservative as well.
	Recall that the ``diagonal'' $\Delta = (\partial X \times \partial X) \setminus \partial ^2 X$ has zero measure (\autoref{res: measure diagonal}).
	Therefore the action of $G$ on $(\partial X \times \partial X, \mathfrak B \otimes \mathfrak B, \nu_o \otimes \nu_o)$ is conservative.
\end{proof}


\begin{lemm}
\label{res: dissipativity prelim}
	Let $A$ be a squeezing subset of $\partial^2 X$ and $B = A \times I$ where $I \subset \R$ is a bounded interval.
	Let $v = (c,c',s)$ be a vector in $SX$.
	Assume that one of the following holds:
	\begin{equation*}
		\sum_{g \in G} \mathbb 1_A(gc,gc') = \infty 
		\quad \text{or} \quad
		\int_\R \mathbb 1_{GB}\circ \phi_t(v) = \infty.
	\end{equation*}
	Then $c$ or $c'$ belong to the radial limit set $\Lambda_{\rm rad}(G)$ of $G$.
\end{lemm}

\begin{proof}
	Without loss of generality, we can suppose that $o$ belongs to the bottleneck $K$ of $A$.
	We write $d$ for its diameter.
	Let $\nu \colon \R \to X$ be a complete gradient line from $c$ to $c'$.
	We choose our parametrization of $\nu$ such that $c'(o,x) = 0$, where $x = \nu(0)$.
	If
	\begin{equation*}
		\sum_{g \in G} \mathbb 1_A(gc,gc') = \infty 
		\quad \text{or} \quad
		\int_\R \mathbb 1_{GB}\circ \phi_t(v) = \infty,
	\end{equation*}
	then there is an infinite sequence $(g_n)$ of pairwise distinct elements of $G$ such that for every $n \in \N$, the path $\nu$ crosses $g_nK$, at the point $q_n = \nu(t_n)$ say.
	Since the action of $G$ on $X$ is proper, the sequence $(g_no)$ diverges to infinity.
	Recall that both $o$ and $g_n^{-1}q_n$ belong to $K$, whose diameter is bounded.
	Thus, up to passing to a subsequence, we can assume that $(t_n)$ diverges to $\pm \infty$.
	Assume that it diverges to $+ \infty$.
	Since $\nu$ is a gradient line for $c'$, and $c'$ is $1$-Lipschitz, we get
	\begin{align*}
		c'(o, g_no)
		\geq c'(x, q_n)  -\dist {q_n}{g_no}
		& \geq \dist x{q_n} -\dist {q_n}{g_no} \\
		& \geq \dist o{g_no} - [\dist ox + 2\dist {q_n}{g_no}] \\
		& \geq \dist o{g_no} - [\dist ox + 2d].
	\end{align*}
	In other words, there is $R \in \R_+$ such that $c' \in \mathcal O_o(g_no,R)$, for every sufficiently large $n \in \N$.
	This implies that $c'$ belongs to the radial limit set.
	If $(t_n)$ diverges to $- \infty$, we prove in the same way that $c$ lies in the radial limit set.
\end{proof}


\begin{prop}
\label{res: dissipativity geodesic flow}
	If $\nu_o$ gives zero measure to the radial limit set, then
	\begin{enumerate}
		\item the diagonal action of $G$ on $(\partial_0^2X, \mathfrak B \otimes \mathfrak B, \nu_o \otimes \nu_o)$ is dissipative, 
		\item
		the geodesic flow on $(SX, \mathfrak L_G, m_G)$ is dissipative.
	\end{enumerate}
\end{prop}

\begin{proof}
	Let $\mathcal U$ be the collection of open squeezing subsets of $\bar X \times \bar X$.
	It follows from the definition that $\mathcal U$ covers $\partial^2_0X$.
	Recall that $\partial X$ is compact metrizable, hence second-countable.
	In particular, so is $\partial^2_0X$.
	Consequently, it is covered by a countable subset $\mathcal U_0 \subset \mathcal U$.
	We now write $\mathcal B_0$ for the collection of all subets $B = (U \cap \partial^2X) \times \intval s{s'}$ of $S_0X$ where $U \in \mathcal U_0$ and $s,s' \in \Q$.
	According to the above discussion, $\mathcal B_0$ is a countable collection covering $S_0X$.

	We now prove that the geodesic flow on $(SX, \mathfrak L_G, m_G)$ is dissipative. 
	Denote by $\Omega_C$ the conservative part of the geodesic flow.
	Let $B \in \mathcal B_0$.
	By recurrence, we know that for $m_G$-almost every vector $v = (c,c',s)$ in $GB \cap \Omega_C$ we have
	\begin{equation*}
		\int_\R \mathbb 1_{GB}\circ \phi(t) dt = \infty.
	\end{equation*}
	Hence for $m_G$-almost every vector $v = (c,c',s)$ in $GB \cap \Omega_C$ either $c \in \Lambda$ or $c' \in \Lambda$ (\autoref{res: dissipativity prelim}).
	However according to our assumption $\partial X \times \Lambda$ and $\Lambda \times \partial X$ have zero $(\nu_0 \otimes \nu_0)$-measure.
	Hence $m_G(GB \cap \Omega_C) = 0$.
	The sets $GB$ as above are countably many and cover $S_0X$.
	Moreover, $m_G$ is supported on $S_0X$.
	Consequently $\Omega_C$ has zero $m_G$-measure.
	The geodesic flow is therefore dissipative.	
	The same argument as above shows that the diagonal action of $G$ on $(\partial^2_0X, \mathfrak B \otimes \mathfrak B, \nu_o \otimes \nu_o)$ is dissipative.
\end{proof}

The next statement is a variation of the previous one for the saturated $\sigma$-algebras.
As we noticed earlier, if $\nu_o$ gives full measure to the radial limit, then the geodesic flow on $(SX, \mathfrak L_G,m_G)$ is conservative, hence so is the geodesic flow on  $(SX, \mathfrak L^+_G,m_G)$.
However \autoref{res: dissipativity geodesic flow} does not automatically provide the converse statement.
This is purpose of the next result.

\begin{prop}
\label{res: non conservative geodesic flow}
	If $\nu_o$ gives zero measure to the radial limit set, then 
	\begin{enumerate}
		\item \label{enu: non conservative geodesic flow - bdy} 
		the diagonal action of $G$ on $(\partial^2X, \mathfrak B^+ \otimes \mathfrak B^+, \nu_o \otimes \nu_o)$ is not conservative;
		\item \label{enu: non conservative geodesic flow - flow}
		the geodesic flow on $(SX, \mathfrak L^+_G,m_G)$ is not conservative. 
	\end{enumerate}		
\end{prop}

\begin{rema}
	Item \ref{enu: non conservative geodesic flow - flow} is given for the benefit of the reader.
	It will not be used in the proof of the Hopf-Tsuji-Sullivan dichotomy.
\end{rema}

\begin{proof}
	Let $h \in G$ be a contracting element.
	Combining Lemmas~\ref{res: measure product shadows} and \ref{res: saturation shadow} we can find $r_0, r_1 \in \R_+$ with $r_0 \leq r_1$ and $n \in \N$ with the following properties.
	\begin{itemize}
		\item The set $A_i = A(h^{-n}o, h^no, r_i)$ is squeezing with finite positive measure.
		\item The set 
		\begin{equation*}
			A^+ = \mathcal O^+_{h^no}(h^{-n}o, r_0) \times \mathcal O^+_{h^{-n}o}(h^no, r_0) 
		\end{equation*}
		belongs to $\mathfrak B^+ \otimes \mathfrak B^+$ and satisfies $A_0 \subset A^+ \subset A_1$.
	\end{itemize}
	Let $S_1X = GA^+ \times \R$.
	We consider now the restriction of the geodesic flow to the space $(S_1X, \mathfrak L^+_G, m_G)$.
	By construction $S_1X$ has positive $m_G$-measure.
	Moreover $m_G$ restricted to $(S_1X, \mathfrak L^+_G)$ is $\sigma$-finite.
	In particular it admits a Hopf decomposition.
	Denote by $\Omega_C$ its conservative part.
	
	Let $I$ be a compact interval and $B = A^+ \times I$.
	As a subset of $A_1$, the set $A^+$ is squeezing.
	Proceeding as in the proof of \autoref{res: dissipativity geodesic flow}, we prove that $m_G(GB \cap \Omega_C) = 0$.
	However $S_1X$ can be covered by countably many sets $GB$ as above.
	Therefore $m_G(\Omega_C) = 0$ and the geodesic flow on $(S_1X, \mathfrak L^+_G, m_G)$ is dissipative.
	In particular, $S_1X$ contains a wandering set $W \in \mathfrak L^+_G$ with positive measure.
	Observe that $W$ is also a wandering set for the geodesic flow on $(SX, \mathfrak L^+_G,m_G)$.
	Thus the latter is not conservative.
	
	Proceeding in the same way, we prove first that the diagonal action of $G$ on $(GA^+, \mathfrak B \otimes \mathfrak B^+, \nu_o \otimes \nu_o)$ is dissipative, and then that the diagonal action of $G$ on $(\partial^2X, \mathfrak B^+ \otimes \mathfrak B^+, \nu_o \otimes \nu_o)$ is not conservative.
\end{proof}


%
\subsection{Ergodicity of the geodesic flow}
\label{sec: ergodicity}
%

The remainder of this section is mainly dedicated to the proof of the following statement.

\begin{theo}
\label{res: ergodicity geodesic flow}
	Assume that the geodesic flow on $(SX, \mathfrak L_G,m_G)$ is conservative.
	Then the diagonal action of $G$ on $(\partial^2X, \mathfrak B^+ \otimes \mathfrak B^+, \mu)$ is ergodic.
\end{theo}

For this proof, we will only work with $(SX, \mathfrak L_G, m_G)$ and not $(SX, \mathfrak L^+_G, m_G)$.
In particular, we will apply the Hopf ergodic theorem in $L^1(SX, \mathfrak L_G, m_G)$.
We assume that the geodesic flow on $(SX, \mathfrak L_G,m_G)$ is conservative.
It follows from \autoref{res: dissipativity geodesic flow} that $\nu_o$ gives positive measure to the radial limit set.
Hence $\omega = \omega_G$ and the action of $G$ on $X$ is divergent (\autoref{res: quasi-conf + ergo}).
Following Bader and Furman \cite{Bader:2017te}, we define the following two operations.

\paragraph{From the boundary to the unit tangent bundle.}
We define a map
\begin{equation}
\label{eqn : ergo2 - tensor product}
	\begin{array}{ccc}
		L^1(\partial^2X, \mathfrak B \otimes \mathfrak B, \mu) \times L^1(\R,ds) & \to & L^1(SX,\mathfrak L, m) \\
		(f,\theta)                             & \mapsto & f_\theta
	\end{array}
\end{equation}
where 
\begin{equation*}
	f_\theta(v) = f(c,c') \theta(s), \quad \text{for every $v = (c,c',s)$ in $SX$}.
\end{equation*}
For every $f \in L^1(\partial^2X, \mathfrak B \otimes \mathfrak B, \mu)$ and $\theta \in L^1(\R,ds)$, we have
\begin{equation*}
	\int_{SX} f_\theta dm = \int_{\partial^2X} f d\mu \int_\R \theta(s) ds.
\end{equation*}
If $\theta$ is symmetric, then the operation $f \mapsto f_\theta$ is equivariant with respect to the actions induced by the flip involutions.

\paragraph{From the cover to the quotient.}
Given a \emph{non-negative} summable function $f \in L^1_+(SX,\mathfrak L, m)$, we define a $G$-invariant function
\begin{equation*}
\label{eqn: ergo2 - averaging function}
	\begin{array}{rccc}
		\hat f \colon & SX & \to & \R_+ \cup \{\infty\} \\
		 & v & \mapsto & \displaystyle \sum_{g \in G} f\left(gv\right).
	\end{array}
\end{equation*}
Recall that $\xi \colon SX \to \intval 01$ is the $G$-invariant map defined by (\ref{eqn: def stab map}) while $D \in \mathfrak L$ stands for a Borel fundamental domain for the action of $G$ on $S_0X$.
Moreover, we have the relation
\begin{equation*}
	\mathbb 1_{S_0X} = \xi \sum_{g \in G} \mathbb 1_{gD}.
\end{equation*}
Combined with the fact that $m$ is $G$-invariant, we get
\begin{equation}
	\label{eqn: ergo2 - averaging vs integration}
	\int_{SX} \hat   f d  m_G
	=  \int \xi \left(\sum_{g \in G} f\circ g\right) \mathbb 1_D dm
	= \int f\xi   \left(\sum_{g \in G}\mathbb 1_{gD} \right) dm
	= \int f \mathbb 1_{S_0X} dm.
\end{equation}
In particular, $\hat f\in L^1_+(SX,\mathfrak L_G, m_G)$.

We can now extend this operation to any function in $L^1(SX,\mathfrak L_G, m_G)$ (not just the non-negative ones).
Consider such a function $f$.
The above discussion applied to $f' = \abs f$ tells us that $\hat f'$ is summable, hence finite $m_G$-almost everywhere.
Thus the series
\begin{equation*}
	\sum_{g \in G} f\left(gv\right)
\end{equation*}
is absolutely convergent, $m_G$-almost everywhere, and can be used to define a function $\hat f$ as above.
Observe also that 
\begin{equation*}
	\abs{\hat f} \leq \widehat{\abs f}.
\end{equation*}
so that $\hat f$ belongs to $L^1(SX,\mathfrak L_G, m_G)$.
Recall that the action of $G$ commutes with the flip $\sigma$.
Hence the operation $f \mapsto \hat f$ commutes with the action induced by the flip.

\paragraph{A particular space of functions.}
For the remainder, we fix once and for all the function $\theta \colon \R \to \R$ we are going to work with, namely 
\begin{equation*}
	\theta(t) = e^{-2\abs t}.
\end{equation*}
For every $T_1, T_2 \in \R$ with $T_1 \leq T_2$ we define
\begin{equation*}
	\Theta_{T_1}^{T_2}(u) = \int_{T_1}^{T_2}\theta(t + u) dt.
\end{equation*}
This function is \og almost constant \fg on $[-T_2,-T_1]$ and decays exponentially outside this interval.
More precisely, we have the following useful estimates.
\begin{enumerate}
	\item For every $u \in \R$,
	      \begin{equation}
		      \label{eqn : ergo2 - tail Theta}
		      \Theta_{T_1}^{T_2}(u) \leq \frac 12\min \left\{ e^{2(T_2+u)}, e^{-2(T_1+u)} \right\}.
	      \end{equation}
	\item For every $u \in [-T_2,-T_1]$,
	      \begin{equation*}
		      \Theta_{T_1}^{T_2}(u) = 1 - \frac 12\left[ e^{2(T_1+u)} + e^{-2(T_2+u)}\right].
	      \end{equation*}
	      Consequently, for every $u, u' \in [-T_2, -T_1]$,
	      \begin{equation}
		      \label{eqn : ergo2 - belly Theta}
		      \abs{\Theta_{T_1}^{T_2}(u) - \Theta_{T_1}^{T_2}(u')}
		      \leq \left[ e^{\left(2T_1 + u+u'\right)} +e^{-\left(2T_2 + u+u'\right)} \right]\sinh \left(\abs{u - u'}\right).
	      \end{equation}
	      Heuristically, the difference on the left hand side almost vanishes whenever $-T_2 \ll u, u' \ll -T_1$.
\end{enumerate}
See \autoref{fig: graph Theta} for a sketch of the graph of $\Theta_{T_1}^{T_2}$.

\begin{figure}[htbp]
\centering
	\includegraphics[page=1, width = \linewidth]{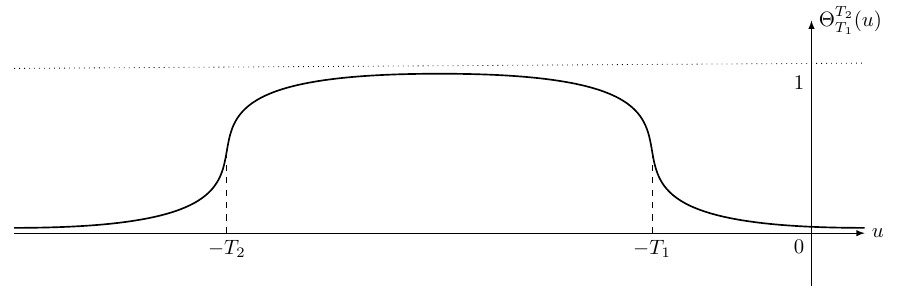}
	\caption{Graph of the map $\Theta_{T_1}^{T_2}$.}
	\label{fig: graph Theta}
\end{figure}

\medskip
Denote by $\Lambda$ the preferred limit set studied in \autoref{res: reduced bdy standard}.
The geodesic $\gamma \colon \R \to X$ and the parameter $\alpha$ are the same as the one fixed in this section.
Let $L > 200\alpha$.
We consider here the $L$-visual pseudo-metric $\distV[L]$ on $\Lambda$ introduced in \autoref{def: L-visual pseudo metric}.
Recall that, by construction, the pseudo-distance $\dist[L]c{c'}$ between two points $c,c' \in \Lambda$ only depends on their equivalence classes for $\sim$.
Hence balls in $(\Lambda, \distV[L])$ belong to $\mathfrak B^+$.

Let $b, b' \in \Lambda$ which are not equivalent.
We write $r_1,r_1', \beta$ for the parameters given by \autoref{res: finite Gromov product from visual metric} applied with $b$ and $b'$.
Let $(r,r') \in (0, r_1) \times (0, r'_1)$.
We consider the function $f \colon \partial^2X \to \R$ defined by $f = \mathbb 1_{A \times A'}$, where $A$ (\resp $A'$) is the open ball in $(\Lambda, \distV[L])$ of radius $r$ centered at $b$ (\resp of radius $r'$ centered at $b'$) seen as a subset of $\partial X$.
A function of this form is called \emph{elementary}.
We denote by $\mathcal E$ the collection of all elementary functions.
Note that $\mathcal E$ is invariant under the action induced by $G$ and the flip involution on $\partial^2X$.
We summarize in the next statement the properties of elementary functions.

\begin{lemm}
\label{res: product of balls in E}
	Let $f \in \mathcal E$.
	There is $\beta \in \R_+$ such that the following holds:
	\begin{enumerate}
		\item \label{enu: product of balls in F - gromov}
		For every $(c,c') \in \partial^2X$, if $f(c,c') \neq 0$, then $\gro c{c'}o \leq \beta$.
		\item \label{enu: product of balls in F - summable}
		The map $f$ belongs to $L^1(\partial^2X, \mathfrak B^+ \otimes \mathfrak B^+, \mu)$.
		\item \label{enu: product of balls in F - bounded}
		The map $\hat f_\theta$ is bounded.
		\item \label{enu: product of balls in F - locally cst}
		Let $c_1, c_2, c' \in \Lambda$ such that $(c_1,c')$ and $(c_2, c')$ belongs to $\partial^2X$.
		There is $T \in \R$, such that
		\begin{equation*}
			f(gc_1,gc') = f(gc_2,gc'), \quad \text{whenever} \quad 
			c'(o,g^{-1}o) \geq T.
		\end{equation*}
		\item \label{enu: product of balls in F - counting}
		For every $a \in \R_+$, there is $N\in \N$ such that for every $(c,c') \in \partial^2X$, for every $T \in \R_+$ the set
		\begin{equation*}
			\set{g \in G}{ f(gc,gc') \neq 0, \abs{c'(o, g^{-1}o) - T} \leq a}
		\end{equation*}
		contains at most $N$ elements.
	\end{enumerate}
\end{lemm}

\begin{proof}
	Let $b, b' \in \Lambda$ which are not equivalent.
	We write $r_1,r_1', \beta$ for the parameters given by \autoref{res: finite Gromov product from visual metric} applied with $b$ and $b'$.
	Let $(r,r') \in (0, r_1) \times (0, r'_1)$.
	Denote by $A$ (\resp $A'$) the open ball of $(\Lambda, \distV[L])$ of radius $r$ (\resp $r'$) and consider the elementary function $f = \mathbb 1_{A \times A'}$.
	It follows from \autoref{res: finite Gromov product from visual metric}, that the map $\partial^2X \to \R_+\cup \{\infty\}$ sending $(c,c')$ to $\gro c{c'}o$ is bounded above by $\beta$, when restricted to $A \times A'$, whence \ref{enu: product of balls in F - gromov}.
	Consequently $f$ belongs to $L^1(\partial^2X, \mathfrak B^+ \otimes \mathfrak B^+, \mu)$, which proves \ref{enu: product of balls in F - summable}.
	
	Let $c_1, c_2, c' \in \Lambda$ such that $(c_1,c')$ and $(c_2, c')$ belongs to $\partial^2X$.
	According to \autoref{res: small distance with three cocycles}, there is $T \in \R$ such that for every $g \in G$,
	\begin{equation*}
		\dist[L]{gc_1}{gc_2} < r \quad \text{provided} \quad c'(o,g^{-1}o) \geq T.
	\end{equation*}
	Consider such an element $g$.
	Recall that $\distV[L]$ is an ultra-metric. 
	Hence $gc_1 \in A$ if and only if $gc_2 \in A$.
	Thus $f(gc_1,gc') = f(gc_2,gc')$, whence \ref{enu: product of balls in F - locally cst}.
	
	We now prove \ref{enu: product of balls in F - counting}.
	Recall that the action of $G$ on $X$ is proper.
	Hence there is a map $N \colon \R_+ \to \N$, such that for every $\ell \in \R_+$ and $y \in X$, the set 
	\begin{equation*}
		F(y, \ell)  = \set{g \in G}{\dist y{go} \leq \ell}
	\end{equation*}
	contains at most $N(\ell)$ elements.
	Indeed, if $F(y, \ell)$ is non-empty, then it is contained in $F(go, 2\ell)$ for some $g \in G$.
	Let $(c,c') \in \partial^2 X$ and $T \in \R$.
	Up to replacing $(c,c', T)$ by $(uc,uc', T - c'(o, u^{-1}o))$, we can assume without loss of generality that $f(c,c') \neq 0$.
	Fix a complete gradient line $\nu \colon \R \to X$ from $c$ to $c'$ parametrized such that $c'(o, \nu(0)) = 0$.
	Consider an element $g \in G$ such that $f(gc,gc') \neq 0$, i.e. $(gc,gc') \in A \times A'$.
	It follows from \autoref{res: finite Gromov product from visual metric} that there is $y = \nu(t)$ such that $\dist y{g^{-1}o} \leq \beta$.
	According to our choice of parametrization, we have
	\begin{equation*}
		c'(o,g^{-1}o) - t
		=  c'(\nu(t), \nu(0)) +  c'(\nu(0), o) + c'(o,g^{-1}o)
		= c'(y, g^{-1}o).
	\end{equation*}
	As $c'$ is $1$-Lipschitz, we get $\abs{c'(o,g^{-1}o) - t} \leq \beta$.
	It follows from this discussion that for every $a, T \in \R$, the set
	\begin{equation*}
		\set{g \in G}{ f(gc,gc') \neq 0, \abs{c'(o,g^{-1}o) - T} \leq a} \subset F\left( \nu(T), a + 2\beta\right),
	\end{equation*}
	which completes the proof of \ref{enu: product of balls in F - counting}.

	We are left to prove \ref{enu: product of balls in F - bounded}, i.e. that $\hat f_\theta$ is bounded.
	Let $v = (c, c',s)$ be a vector of $SX$.
	Observe that if $(c,c') \notin \Lambda \times \Lambda$, then $\hat f_\theta(v) = 0$.
	Assume now that $(c,c') \in \Lambda \times \Lambda$.
	In particular, there exists a bi-infinite gradient line $\nu \colon \R \to X$ from $c$ to $c'$ (Propositions~\ref{res: separation triangle cocycles} and~\ref{res: bi-gradient line}).
	We choose the parametrization of $\nu$ such that 
	\begin{equation*}
		\frac 12 \left[ c'(o, x) - c(o, x)\right] = s, \quad \text{where} \quad x = \nu(0).
	\end{equation*}
	The computation yields,
	\begin{equation*}
		\hat f_\theta (v) = \sum_{g \in G} \mathbb 1_{A\times A'} (gc, gc') \theta\left(s + \tau(g,v)\right).
	\end{equation*}
	
	Consider an element $g \in G$ such that the corresponding term in the above sum does not vanish, i.e. $(gc,gc') \in A \times A'$.
	It follows from \autoref{res: finite Gromov product from visual metric} that there is $y = \nu(t)$ such that $\dist y{g^{-1}o} \leq \beta$.
	According to our choice of parametrization of $\nu$ we have
	\begin{equation*}
		s + \tau(g,v) = \frac 12 \left[c'(g^{-1}o,x) - c(g^{-1}o,x)\right].
	\end{equation*}
	Since $c$ and $c'$ are $1$-Lipschitz, this quantity differs from 
	\begin{equation*}
		\frac 12 \left[c'(y,x)  - c(y,x)\right] = -t = \pm \dist xy.
	\end{equation*}
	by at most $\beta$.
	Combined with the triangle inequality, it yields
	\begin{equation*}
		\abs{s + \tau(g,v)} \geq \dist x{g^{-1}o}  - 2\beta.
	\end{equation*}
	Coming back the computation of $\hat f_\theta$, the above discussion implies that
	\begin{equation*}
		\hat f_\theta (v) \leq e^{4\beta} \sum_{g \in E} e^{-2\dist x{go}}
	\end{equation*}
	where $E$ consists of all elements $g \in G$ such that $go$ is $\beta$-close to $\nu$.
	Hence 
	\begin{equation*}
		\hat f_\theta (v) 
		\leq e^{4\beta} \sum_{k \in \Z} \sum_{g \in F(\nu(k \beta), 2\beta)}  e^{-2\dist x{go}}
		\leq N(2\beta)e^{8\beta}  \sum_{k \in \Z} e^{-2\abs{k}\beta} 
		< \infty.
	\end{equation*}
	The above estimate does not depend on $v$, hence $\hat f_\theta$ is bounded.
\end{proof}

\begin{prop}
\label{res: dense subspace for Hopf}
	The linear space spanned by $\mathcal E$ is dense in $L^1(\partial^2X, \mathfrak B^+ \otimes \mathfrak B^+, \mu)$.
\end{prop}

\begin{proof}
	Simple functions are dense in $L^1(\partial^2X, \mathfrak B^+ \otimes \mathfrak B^+, \mu)$.
	Thus we just have to prove any function of the form $\mathbb 1_B$ with $B \in \mathfrak B^+ \otimes \mathfrak B^+$ can be approximated by a linear combination of elements of $\mathcal E$, provided $\mu(B) < \infty$.
	However $\mathfrak B^+ \otimes \mathfrak B^+$ is a product $\sigma$-algebra, thus it suffices to consider subsets $B \in \mathfrak B^+ \otimes \mathfrak B^+$ of the form $B = A \times A'$ with $A,A' \in \mathfrak B^+$.
	Since $\Lambda$ has full $\nu_o$-measure, we can even assume that $A$ and $A'$ are contained in $\Lambda$.
	Recall that $\Lambda\qsim$ endowed with the Borel $\sigma$-algebra induced by $\distV[L]$ is isomorphic as measurable space to $(\Lambda, \mathfrak B^+)$.
	A Vitali type argument shows that any subset $A \in \mathfrak B^+$ can be approximated by a finite disjoint union of arbitrarily small balls for the distance $\distV[L]$, whence the result.
\end{proof}


\begin{lemm}
\label{res: existence positive and bounded function}
	There exists $\rho \in \mathcal E$, such that
	\begin{itemize}
		\item $\hat \rho_\theta$ is positive $m_G$-almost everywhere;
		\item the map $P \colon \partial^2X \to \R_+$ given by
		\begin{equation*}
			P(c,c') = \int_{\R}\hat \rho_\theta(c,c',t) \theta(t) dt,
		\end{equation*}
		is bounded and positive $\mu$-almost everywhere.
	\end{itemize}
\end{lemm}


\begin{proof}
	Recall that $\gamma \colon \R \to X$ is the path we used to study the preferred limit set $\Lambda$ while $c_\gamma$ and $c'_\gamma$ are the Busemann cocycles associated to $\gamma$ as in \autoref{exa: busemann cocycle}.
	Note that $c_\gamma$ and $c'_\gamma$ are not equivalent.
	We write $r_1,r_1',\beta$ for the parameters given by \autoref{res: finite Gromov product from visual metric} applied with $c_\gamma$ and $c'_\gamma$.
	Let $(r,r') \in (0, r_1) \times (0, r'_1)$.
	Let $A$ (\resp $A'$) be the ball in $(\Lambda, \distV[L])$ of radius $r$ (\resp $r'$) centered at $c_\gamma$ (\resp $c'_\gamma$) seen as a subset of $\partial X$.
	We choose $\rho = \mathbb 1_{A \times A'}$, which is elementary by construction.
	In particular $\hat \rho_\theta$ is bounded, hence so is $P$.

	We now prove that $\hat \rho_\theta$ and $P$ are positive almost everywhere.
	Recall that $\Lambda$ decomposes as a union $\Lambda = \Lambda_- \cup \Lambda_+$ where $\Lambda_-$ and $\Lambda_+$ are two $G$-invariant, saturated Borel sets with full $\nu_o$-measure (see the beginning of \autoref{sec: initial data}).
	Thus it suffices to show that $\hat \rho_\theta(c,c',s) > 0$ and $P(c,c') > 0$, whenever $(c,c') \in \Lambda \times \Lambda_+$.
	Consider a vector $v = (c,c',s)$ for some $(c,c') \in \Lambda \times \Lambda_+$ and $s \in \R$.
	According to \autoref{res: dense orbit geodesic}, there is $u \in G$ such that
	\begin{equation*}
		\max \left\{\dist[L] {uc}{c_\gamma} ,  \dist[L] {uc'}{c'_\gamma} \right\} < \min\{r,r'\}.
	\end{equation*}
	Hence $(uc, uc') \in A \times A'$.
	Consequently
	\begin{equation*}
		\hat \rho_\theta (v)
		\geq \rho_\theta(uv)
		\geq \theta(s+ \tau(u,v))
		> 0.
	\end{equation*}
	Moreover
	\begin{align*}
		P(c,c')
		& \geq \sum_{g \in G}  \mathbb 1_{A \times A'}(gc,gc')  \int_\R \theta\left(t + \tau(g,v)\right) \theta(t) dt \\
		& \geq  \int_\R \theta\left(t + \tau(u,v)\right) \theta(t) dt > 0. \qedhere
	\end{align*}
\end{proof}


\paragraph{Running the Hopf argument.}

\begin{prop}
\label{res: exploiting contraction}
	Let $f \in \mathcal E$.
	There is a flow/flip invariant subset $Y \in \mathfrak L_G$ with full $m_G$-measure, such that for all vectors $v_1, v_2 \in Y$ of the form $v_i = (c_i, c', s_i)$ (i.e. with the same futur) the map 
	\begin{equation*}
		\begin{array}{ccc}
			\R_+ & \to & \R_+ \\
			T & \mapsto & \displaystyle \int_0^T \left[ \hat f_\theta \circ \phi_t(v_1) -  \hat f_\theta \circ \phi_t(v_2) \right] dt
		\end{array}
	\end{equation*}
	is bounded.
\end{prop}

\begin{proof}
	Let $\Lambda^{(2)} = (\Lambda \times \Lambda) \cap \partial^2 X$.
	The set $Y = \Lambda^{(2)}  \times \R$ has full $m_G$-measure.
	It is invariant by $G$, the flow, and the flip.
	Consider now $v_1 = (c_1, c',s_1)$ and $v_2 = (c_2, c', s_2)$ in $Y$.
	According to \autoref{res: product of balls in E}, $\hat f_\theta$ is bounded.
	Consequently, for every $t_1, t_2 \in \R$, the map sending $T \in \R_+$ to 
	\begin{equation*}
		 \abs{
			 \int_0^T \left[ \hat f_\theta \circ \phi_t(v_1) -  \hat f_\theta \circ \phi_t(v_2) \right] dt
			 - 
			  \int_0^T \left[ \hat f_\theta \circ \phi_{t+t_1}(v_1) -  \hat f_\theta \circ \phi_{t+t_2}(v_2) \right] dt
		  }
	\end{equation*}
	is bounded.
	Therefore, we can advantageously replace $v_i$ by $w_i = \phi_{t_i}(v_i)$ for a suitable value of $t_i$.
	In practice, we take $t_i = -s_i - \gro {c_i}{c'}o$.
	For simplicity we let
	\begin{equation*}
		F_i(T) =  \int_0^T \hat f_\theta \circ \phi_t(w_i)dt, \quad \forall T \geq 0.
	\end{equation*} 
	The computation yields
	\begin{equation*}
		F_i(T) =  \sum_{g \in G} f(gc_i, gc') \Theta_0^T (\delta_i(g)),
		\quad \text{where} \quad
		\delta_i(g) = - c'(o,g^{-1}o) - \gro {c_i}{c'}{g^{-1}o}.
	\end{equation*}
	
	Before starting our estimations, we define some other auxiliary objets / quantities.
	It follows from \autoref{res: product of balls in E} that there is $\beta, T_0 \in \R_+^*$ such that for every $g \in G$, we have
	\begin{equation*}
		f(gc_1,gc') = f(gc_2,gc'),
		\quad \text{whenever} \quad 
		c'(o,g^{-1}o) \geq T_0.
	\end{equation*}
	Moreover $\gro b{b'}o \leq \beta$, as soon as $f(b,b') \neq 0$.
	We split the group $G$ as follows: for every $k \in \Z$ and $i \in \{1,2\}$, we write $G_{i,k}$ for the set of all $g \in G$, such that 
	\begin{itemize}
		\item $k \beta \leq c'(o, g^{-1}o) < (k+1)\beta$, and
		\item $f(gc_i, gc') \neq 0$.
	\end{itemize}
	According to \autoref{res: product of balls in E}, the cardinality of $G_{i,k}$ is bounded above by a number $N$ that does not depend on $k$ nor on $i$.
	In addition, we set $G_k = G_{1,k} \cup G_{2,k}$.
	
	\medskip
	We now decompose the work as follows.
	\begin{equation*}
		\abs{F_1(T) - F_2(T)}
		\leq \sum_{k \in \Z} S_k,
	\end{equation*}
	where
	\begin{equation*}
		S_k =  \sum_{g \in G_k}\abs{f(gc_1,gc')\Theta_0^T(\delta_1(g)) - f(gc_2,gc')\Theta_0^T(\delta_2(g))}.	
	\end{equation*}
	We are going to estimate each $S_k$ depending on the value of $k$.
	
	\begin{clai}[Estimate for the tails]
		For every integer $k \in \Z$,
		\begin{equation*}
			S_k \leq Ne^{4\beta} \min \left\{ e^{2(T-k \beta)}, e^{2k \beta} \right\}. 
		\end{equation*}
	\end{clai}
	
	Let $k \in \Z$.
	A (brutal) upper bound for $S_k$ using (\ref{eqn : ergo2 - tail Theta}) is 
	\begin{equation*}
		S_k 
		\leq \sum_{i \in \{1,2\}} \sum_{g\in G_{i,k}} \Theta_0^T(\delta_i(g))
		\leq \frac 12\sum_{i \in \{1,2\}} \sum_{g\in G_{i,k}} \min \left\{ e^{2T+2\delta_i(g)}, e^{-2\delta_i(g)} \right\}. 
	\end{equation*}
	Let $i \in \{1, 2\}$.
	Let  $g \in G_{i,k}$.
	Since $f(gc_i, gc') \neq 0$, we know from \autoref{res: product of balls in E} that $\gro{c_i}{c'}{g^{-1}o} \leq \beta$, hence
	\begin{equation*}
		 - (k+2)\beta < \delta_i(g) \leq -k\beta.
	\end{equation*}
	Recall also that $G_{i,k}$ contains at most $N$ elements.
	Our claim now follows from a direct computation.
	
	\begin{clai}[Estimate for the center]
		Let $k \in \Z$.
		If $T_0 \leq k \beta \leq T - 2\beta$, then
		\begin{equation*}
			S_k \leq Ne^{4\beta} \sinh(\beta) \left[ e^{-2k\beta} +e^{-2(T -k\beta)} \right]
		\end{equation*}
	\end{clai}
	
	Let $g \in G_k$.
	According to our hypothesis, $c'(o, g^{-1}o) \geq T_0$.
	Consequently $f(gc_1, gc') = f(gc_2,gc')$.
	In other words $G_{1,k} = G_{2,k} = G_k$.
	Thus,
	\begin{equation*}
		S_k =  \sum_{g \in G_k}\abs{\Theta_0^T(\delta_1(g)) - \Theta_0^T(\delta_2(g))}.	
	\end{equation*}
	Reasoning as in our previous claim, we observe that $\gro{c_i}{c'}{g^{-1}o} \leq \beta$.
	Consequently
	\begin{equation*}
		 -T \leq - (k+2)\beta < \delta_i(g) \leq -k\beta \leq - T_0 \leq 0.
	\end{equation*}
	Moreover $\abs{\delta_1(g) - \delta_2(g)} \leq \beta$.
	Recall that $G_k$ contains at most $N$ elements.
	The claim now follows from a direct computation using (\ref{eqn : ergo2 - belly Theta}).
	
	\medskip
	Set $C = Ne^{4\beta} \max \{ 1, \sinh \beta\}$.
	Combining our estimates and reindexing the sums, we get
	\begin{itemize}
		\item
		\begin{math}
			\displaystyle
			\sum_{k\beta < T_0} S_k  
			\leq C  \sum_{k\beta < T_0} e^{2k\beta}
			\leq C  \sum_{k > - T_0/ \beta} e^{-2k\beta},
		\end{math}
		
		\item
		\begin{math}
			\displaystyle
			\sum_{T_0 \leq k\beta \leq T - 2\beta} S_k 
			\leq C  \sum_{T_0 \leq k\beta \leq T - 2\beta}  \left(e^{-2k\beta} + e^{-2(T-k\beta)}\right)
			\leq 2 C\sum_{k \geq 0}e^{-2k\beta},
		\end{math}
		
		\item
		\begin{math}
			\displaystyle
			\sum_{k\beta > T-2\beta} S_k 
		 	\leq C \sum_{k\beta > T-2\beta}   e^{2(T-k\beta)} 
			\leq C \sum_{k \geq -2} e^{-2k\beta}.
		\end{math}
	\end{itemize}
	Observe that all these upper bounds are finite and do not depend on $T$.
	Hence the result is proven.
\end{proof}

We now fix a function $\rho$ as given by \autoref{res: existence positive and bounded function}.
Up to multiplying $\rho$ by a positive scalar, we may assume that 
\begin{equation*}
	\int_{SX} \hat \rho_\theta dm_G = \int_{S_0X} \rho_\theta dm = \int_{\partial^2X} \rho d\mu = 1.
\end{equation*}
As in \autoref{res: existence positive and bounded function}, we write $P \colon \partial^2X \to \R_+$ for the map defined by
\begin{equation*}
	P(c,c') = \int_{\R}\hat \rho_\theta(c,c',t) \theta(t) dt.
\end{equation*}
It is bounded and positive $\mu$-almost everywhere.

\begin{prop}
\label{res: hopf argument}
	For every function $f \in L^1(\partial^2X, \mathfrak B^+ \otimes \mathfrak B^+, \mu)$, for almost every vector $v \in (SX, \mathfrak L_G, m_G)$ we have
	\begin{equation*}
		\lim_{T \to \pm\infty} 
		\frac{\displaystyle\int_0^T \hat f_\theta \circ \phi_t(v)dt}{\displaystyle \int_0^T \hat \rho_\theta \circ \phi_t(v) dt} 
		= \int_{\partial^2X} fP d\mu
	\end{equation*}
\end{prop}

\begin{proof}
	On the one hand, $\hat f_\theta \colon SX \to \R_+$ belongs to $L^1(SX, \mathfrak L_G, m_G)$. 
	On the other hand,  the function $\hat \rho_\theta \in L^1(SX, \mathfrak L_G, m_G)$ is positive almost everywhere.
	Since the geodesic flow on $(SX, \mathfrak L_G, m_G)$ is conservative, the Hopf ergodic theorem \cite{Hopf:1937kk} tells us that for almost every vector $v \in (SX, \mathfrak L_G, m_G)$,
	\begin{equation}
	\label{eqn: hopf ergodic thm}
		\lim_{T \to \pm\infty} 
		\frac{\displaystyle\int_0^T \hat f_\theta \circ \phi_t(v)dt}{\displaystyle \int_0^T \hat \rho_\theta \circ \phi_t(v) dt} 
		= f_\infty(v),
		\quad \text{where} \quad
		f_\infty = \mathbb E_{\hat \rho_\theta m_G}\left( \hat f_\theta \middle| \mathfrak I\right)
	\end{equation}
	is the conditional expectation of $\hat f_\theta$ with respect to the sub-$\sigma$-algebra $\mathfrak I$ which consists of all flow-invariant subsets of $\mathfrak L_G$ and the probability measure $\hat \rho_\theta m_G$.
	
	Suppose now that $f$ belongs to $\mathcal E$.
	Denote by $Y \in \mathfrak L_G$ the flow/flip invariant subset given by \autoref{res: exploiting contraction}.
	Let $v_1 = (c_1, c', s_1)$ and $v_2 = (c_2, c', s_2)$ be two vectors in $Y$ with the same futur.
	It follows from \autoref{res: exploiting contraction} that $f_\infty(v_1) = f_\infty(v_2)$.
	In other words $f_\infty$ only depends on the future $m_G$-almost surely.
	Recall that the flow anti-commutes with the flip involution $\sigma$.
	Applying the same argument with $f \circ \sigma$, which belongs to $\mathcal E$ as well, we obtain that $f_\infty$ only depends on the past $m_G$-almost surely.
	As $\mu$ is equivalent to a product measure, one proves using Fubini's theorem that $f_\infty$ is constant $m_G$-almost everywhere.
	By definition of the conditional expectation, the almost sure value $M$ of $f_\infty$ satisfies
	\begin{equation*}
		M 
		= \int_{SX} f_\infty \hat \rho_\theta dm_G
		= \int_{SX} \hat f_\theta \hat \rho_\theta dm_G.
	\end{equation*}
	Since $\hat \rho_\theta$ is $G$-invariant, the previous inequality yields
	\begin{equation*}
		M 
		= \int_{SX} \hat f_\theta \hat \rho_\theta dm_G
		= \int_{SX} \widehat{ f_\theta \hat \rho_\theta} dm_G
		= \int_{S_0X} f_\theta \hat \rho_\theta dm
		= \int_{\partial^2 X} f P d\mu.
	\end{equation*}
	Since $P$ is bounded, both maps
	\begin{equation*}
		f \mapsto \mathbb E_{\hat \rho_\theta m_G}\left( \hat f_\theta \middle| \mathfrak I\right)
		\quad \text{and} \quad
		f \mapsto \int_{\partial^2 X} f P d\mu
	\end{equation*}
	are countinuous linear maps on $L^1(\partial^2X, \mathfrak B^+ \otimes \mathfrak B^+, \mu)$.
	The above discussion shows that they agree on $\mathcal E$.
	According to \autoref{res: dense subspace for Hopf}, the linear space spanned by $\mathcal E$ is dense in $L^1(\partial^2X, \mathfrak B^+ \otimes \mathfrak B^+, \mu)$.
	Hence the two above functions agree on $L^1(\partial^2X, \mathfrak B^+ \otimes \mathfrak B^+, \mu)$, which completes the proof.
\end{proof}

Although we followed the classical Hopf argument, we have not quite proved yet that the geodesic flow is ergodic.
Indeed our strategy applies only to functions of the form $\hat f_\theta$ for some $f \in L^1(\partial^2X, \mathfrak B^+ \otimes \mathfrak B^+, \mu)$.
Nevertheless this is enough to show that the diagonal action of $G$ on the double boundary is ergodic.
We first need an auxiliary lemma, telling us that the measure $\mu$ restricted to $\mathfrak B^+ \otimes \mathfrak B^+$ is $\sigma$-finite.

\begin{lemm}
\label{res: double boundary sigma finite}
	Consider the map $\psi \colon \partial^2 X \to \R^*_+$ given by
	\begin{equation*}
		\psi(c,c') = \sup e^{-2\omega\gro b{b'}o}.
	\end{equation*}
	where the supremum runs over all pairs $(b,b') \in \partial^2X$ such that $b \sim c$, and $b' \sim c'$.
	Then $\psi \in L^1(\partial^2 X, \mathfrak B^+ \otimes \mathfrak B^+, \mu)$.
\end{lemm}

\begin{proof}
	By construction this is a $\mathfrak B^+ \otimes \mathfrak B^+$-measurable function.
	Recall that $\Lambda$ has full $\nu_o$-measure.
	Moreover for every $b,c \in \Lambda$ in the same equivalence class, we have $\norm[\infty]{b-c} \leq 20\alpha$. 
	See for instance \cite[Proposition~5.10]{Coulon:2022tu}.
	Consequently, $\mu$-almost everywhere, we have
	\begin{equation*}
		\psi(c,c') \leq e^{40\omega \alpha} e^{-2\omega\gro c{c'}o} \leq e^{40\omega \alpha} D^{2\omega}_o(c,c').
	\end{equation*}
	Hence $\psi$ is summable.
\end{proof}

\begin{proof}[Proof of \autoref{res: ergodicity geodesic flow}]
	Let $B  \in \mathfrak B^+ \otimes \mathfrak B^+$ be a $G$-invariant subset of $\partial^2 X$ with positive $\mu$-measure.
	Fix a positive function $\psi \in L^1(\partial^2X, \mathfrak B^+ \otimes \mathfrak B^+, \mu)$ as given by \autoref{res: double boundary sigma finite}.
	Applying \autoref{res: hopf argument} with $f = \mathbb 1_B \psi$, there is a subset $A \subset \partial^2X$ in $\mathfrak B \otimes \mathfrak B$ with full $\mu$-measure, such that for every $(c,c') \in A$, for every sufficiently large $T$ we have
	\begin{equation*}
		\int_0^T \hat f_\theta \circ \phi_t(v)dt > 0,
		\quad \text{where} \quad
		v = (c,c',0),
	\end{equation*}
	(note that $A$ is not necessarily saturated).
	Hence for every $(c,c') \in A$, there is $g \in G$, such that $f(gc,gc') >0$, that is $(c,c') \in g^{-1}B$.
	Since $B$ is $G$-invariant, $A \subset B$, thus $B$ has full $\mu$-measure.
\end{proof}

We complete this section with the following statement that will be useful to prove the converse of \autoref{res: ergodicity geodesic flow}.
In particular, we do not assume anymore that the geodesic flow on $(SX, \mathfrak L_G, m_G)$ is conservative.

\begin{prop}
\label{res: ergodicity implies no atom}
	Assume that the diagonal action of $G$ on $(\partial^2_0 X, \mathfrak B^+ \otimes \mathfrak B^+, \nu_o \otimes \nu_o)$ is ergodic, then $(\partial X, \mathfrak B^+, \nu_o)$ has no atom.
\end{prop}

\begin{proof}
	Assume on the contrary that $(\partial X, \mathfrak B^+, \nu_o)$ has an atom, say $A$.
	Let $h$ be contracting element and $Y$ an axis of $h$.
	Since $\nu$ is $G$-invariant, $gA$ is an atom for every $g \in G$.
	According to \autoref{res: subatom with bded proj}, there is $d \in \R_+$ such that for every $g \in G$, one can find an atom $B_g \subset A$ such that $gB_g \cap \partial^+Y$ is empty and the projection of $gB_g$ on $Y$ has diameter at most $d$.
	We let 
	\begin{equation*}
		B = \bigcap_{g \in G} B_g.
	\end{equation*}
	Since $G$ is countable, it is an atom.
	Moreover for every $g \in G$, the projection of $gB$ on $Y$ has diameter at most $d$.
	In particular, we can find $h_1 \in \group h$ such that the projections of $B$ and $h_1B$ on $Y$ are sufficiently far so that $B \times h_1B$ is contained in $\partial^2_0X$ (see \autoref{res: bi-gradient line}).
	According to \autoref{res: no contracting stabilizing atom}, $\stab B$ does not contain a contracting element.
	It follows from \autoref{res: proj orbit subgroup without contracting} that the projection on $Y$ of 
	\begin{equation*}
		Z = \bigcup_{g \in \stab B} gh_1B
	\end{equation*}
	is bounded.
	As previously, we can find $h_2 \in \group h$ such that $\pi_Y(Z)$ and $\pi_Y(h_2B)$ are disjoint while $B \times h_2B$ is contained in $\partial^2_0X$.
	
	As $B$ is an atom both $B \times h_1B$ and $B \times h_2 B$ have positive measure.
	By ergodicity $G(B \times h_1B)$ has full measure.
	Since $G$ is countable, there is $g \in G$ such that $g(B \times h_1B) \cap (B \times h_2B)$ has positive measure.
	In particular $\nu_o(B \cap gB) > 0$ while $gh_1 B \cap h_2B$ is non-empty.
	The former says that $g$ belongs to $\stab B$, while the latter implies the $Z \cap h_2B$ is non-empty.
	This contradicts our choice of $h_2$ and completes the proof.
\end{proof}


%
\subsection{The Hopf-Tsuji-Sullivan dichotomy}
%

We can now summarize the previous study in a single statement.

\begin{theo}[Hopf-Tsuji-Sullivan theorem]
\label{res: hopf-tsuji-sullivan}
	Let $X$ be a proper, geodesic, metric space.
	Fix a base point $o \in X$.
	Let $G$ be a group acting properly, by isometries on $X$.
	We assume that $G$ is not virtually cyclic and contains a contracting element.
	Let $\omega \in \R_+$.
	Let $\nu  = (\nu_x)$ be a $G$-invariant, $\omega$-conformal density supported on $\partial X$.
	Let $m_G$ be the associated Bowen-Margulis measure on $(SX, \mathfrak L_G)$.
	Let $\mathfrak L^+$ be an admissible sub-$\sigma$-algebra of $\mathfrak L$.
	Then one of the following two cases holds.
	Moreover, for each one, all the stated properties are equivalent.
	\begin{itemize}

		\item\emph{\bfseries Convergent case.}
		\begin{enumerate}[label=(c\arabic{*}), ref=(c\arabic{*})]
			\item \label{eqn: hopf-tsuji-sullivan - conv - def}
			The Poincaré series $\mathcal P_G(s)$ converges at $s = \omega$.
			\item \label{eqn: hopf-tsuji-sullivan - conv - radial}
			The measure $\nu_o$ gives zero measure to the radial limit set.
			\item \label{eqn: hopf-tsuji-sullivan - conv - bdy}
			The diagonal action of $G$ on $(\partial_0^2 X, \mathfrak B \otimes \mathfrak B, \nu_o \otimes \nu_o)$ is dissipative.
			\item \label{eqn: hopf-tsuji-sullivan - conv - dissipative}
			The geodesic flow on $(SX, \mathfrak L_G, m_G)$ is dissipative.
		\end{enumerate}
		
		\item\emph{\bfseries Divergent case.}
		\begin{enumerate}[label=(d\arabic{*}), ref=(d\arabic{*})]
			\item \label{eqn: hopf-tsuji-sullivan - div - def}
			The Poincaré series $\mathcal P_G(s)$ diverges at $s = \omega$, hence $\omega = \omega_G$.
			\item \label{eqn: hopf-tsuji-sullivan - div - radial}
			The measure $\nu_o$ gives full measure to the contracting limit set, hence to the radial limit set.
			\item \label{eqn: hopf-tsuji-sullivan - div - bdy}
			The diagonal action of $G$ on $(\partial X \times \partial X, \mathfrak B \otimes \mathfrak B, \nu_o \otimes \nu_o)$ is conservative.
			\item \label{eqn: hopf-tsuji-sullivan - div - conservative}
			The geodesic flow on $(SX, \mathfrak L_G, m_G)$ is conservative.
			\item \label{eqn: hopf-tsuji-sullivan - div - double ergo}
			The diagonal action of $G$ on $(\partial X \times \partial X, \mathfrak B^+ \otimes \mathfrak B^+, \nu_o \otimes \nu_o)$ is ergodic.
			\item \label{eqn: hopf-tsuji-sullivan - div - ergo flow}
			The geodesic flow on $(SX, \mathfrak L^+_G, m_G)$ is ergodic.
		\end{enumerate}
		
	\end{itemize}
\end{theo}

\begin{rema}
\label{rem: admissible algebra - matin theo}
	Before proving the statement, let us comment on the role of the admissible sub-$\sigma$-algebra $\mathfrak L^+$ (see \autoref{def: admissible sigma-alg}).
	Note first that $\mathfrak L^+$ only appears (indirectly) in the ergodicity of the geodesic flow on $(SX, \mathfrak L^+_G, m_G)$.
	As we observed in \autoref{exa: admissible alg}, there are many possible choices for $\mathfrak L^+$.
	In practice, this choice should be guided by the implication the reader is interested in.
	Indeed the smaller $\mathfrak L^+$ is, the more likely the geodesic flow on $(SX, \mathfrak L^+_G, m_G)$ is ergodic.
	Therefore, if one tries to use \ref{eqn: hopf-tsuji-sullivan - div - ergo flow} $\Rightarrow$ \ref{eqn: hopf-tsuji-sullivan - div - def}, choosing a ``small'' sub-$\sigma$-algebra $\mathfrak L^+$ will make it easier to check that \ref{eqn: hopf-tsuji-sullivan - div - ergo flow} holds true.
	That being said, the converse direction -- namely \ref{eqn: hopf-tsuji-sullivan - div - def} $\Rightarrow$ \ref{eqn: hopf-tsuji-sullivan - div - ergo flow} -- is probably the more useful.
	In this situation, since the action of $G$ is divergent, one can make use of \autoref{exa: admissible alg}\ref{enu: admissible alg - div} that provides a  sub-$\sigma$-algebra $\mathfrak L^+$ carrying more informations.
\end{rema}

\begin{proof}
The equivalences \ref{eqn: hopf-tsuji-sullivan - div - def} $\Leftrightarrow$ \ref{eqn: hopf-tsuji-sullivan - div - radial} and \ref{eqn: hopf-tsuji-sullivan - conv - def} $\Leftrightarrow$ \ref{eqn: hopf-tsuji-sullivan - conv - radial} are proven in \cite{Coulon:2022tu}.

\paragraph{Divergent case.}
We can endow $(SX, \mathfrak L, m)$ with an action of $G \times \R$, where the action $\R$ corresponds to the geodesic flow.
Let us state a couple of additional properties.
\begin{enumerate}[label=(d\arabic{*}), ref=(d\arabic{*})]
	\setcounter{enumi}{6}
	\item \label{eqn: hopf-tsuji-sullivan - div - bdy saturated}
	The diagonal action of $G$ on $(\partial^2 X, \mathfrak B^+ \otimes \mathfrak B^+, \nu_o \otimes \nu_o)$ is conservative.
	\item \label{eqn: hopf-tsuji-sullivan - div - double ergo minus diag}
	The diagonal action of $G$ on $(\partial^2X, \mathfrak B^+ \otimes \mathfrak B^+, \nu_o \otimes \nu_o)$ is ergodic.
\end{enumerate}
The proof now follows the plan below.

\begin{center}
	\begin{tikzpicture}
		\matrix (m) [matrix of math nodes, row sep=0.5em, column sep=2.5em, text height=1.5ex, text depth=0.25ex] 
		{ 
				&& \mathrm{ \ref{eqn: hopf-tsuji-sullivan - div - bdy}}&& \\
				\mathrm{\ref{eqn: hopf-tsuji-sullivan - div - def}}  & \mathrm{\ref{eqn: hopf-tsuji-sullivan - div - radial}}  && \mathrm{\ref{eqn: hopf-tsuji-sullivan - div - bdy saturated}} & \\
			&&&&	 \mathrm{\ref{eqn: hopf-tsuji-sullivan - div - double ergo}} \\
				&\mathrm{ \ref{eqn: hopf-tsuji-sullivan - div - conservative}} && \mathrm{\ref{eqn: hopf-tsuji-sullivan - div - double ergo minus diag}} &\\	
				& &&& \mathrm{\ref{eqn: hopf-tsuji-sullivan - div - ergo flow}} \\		
		}; 
		\draw[ {Straight Barb[length=4pt,width=5pt]}-{Straight Barb[length=4pt,width=5pt]}, double] (m-2-1) -- (m-2-2) ;
		\draw[ {Straight Barb[length=4pt,width=5pt]}-{Straight Barb[length=4pt,width=5pt]}, double] (m-4-4) -- (m-5-5) ;
		\draw[ {Straight Barb[length=4pt,width=5pt]}-{Straight Barb[length=4pt,width=5pt]}, double] (m-4-4) -- (m-3-5) ;
		
		\draw[-{Straight Barb[length=4pt,width=5pt]}, double] (m-2-2) -- (m-1-3) ;
		\draw[-{Straight Barb[length=4pt,width=5pt]}, double] (m-1-3) -- (m-2-4) ;
		\draw[-{Straight Barb[length=4pt,width=5pt]}, double] (m-2-4) -- (m-2-2) ;
		\draw[-{Straight Barb[length=4pt,width=5pt]}, double] (m-2-2) -- (m-4-2) ;
		\draw[-{Straight Barb[length=4pt,width=5pt]}, double] (m-4-2) -- (m-4-4) ;
		\draw[-{Straight Barb[length=4pt,width=5pt]}, double] (m-4-4) -- (m-2-4) ;
	\end{tikzpicture}
\end{center}

The implications \ref{eqn: hopf-tsuji-sullivan - div - bdy} $\Rightarrow$ \ref{eqn: hopf-tsuji-sullivan - div - bdy saturated} and \ref{eqn: hopf-tsuji-sullivan - div - double ergo} $\Rightarrow$ \ref{eqn: hopf-tsuji-sullivan - div - double ergo minus diag} are straightforward.
\autoref{res: conservativity geodesic flow} shows \ref{eqn: hopf-tsuji-sullivan - div - radial} $\Rightarrow$ \ref{eqn: hopf-tsuji-sullivan - div - bdy} and \ref{eqn: hopf-tsuji-sullivan - div - radial} $\Rightarrow$ \ref{eqn: hopf-tsuji-sullivan - div - conservative}, while \autoref{res: non conservative geodesic flow}\ref{enu: non conservative geodesic flow - bdy}  gives
\ref{eqn: hopf-tsuji-sullivan - div - bdy saturated} $\Rightarrow$ \ref{eqn: hopf-tsuji-sullivan - div - radial}.
Recall that $\mu$ and $\nu_o \otimes \nu_o$ restricted to $\partial^2X$ belong to the same measure class, hence the implication \ref{eqn: hopf-tsuji-sullivan - div - conservative} $\Rightarrow$  \ref{eqn: hopf-tsuji-sullivan - div - double ergo minus diag} follows from \autoref{res: ergodicity geodesic flow}.

We now prove \ref{eqn: hopf-tsuji-sullivan - div - double ergo minus diag} $\Rightarrow$ \ref{eqn: hopf-tsuji-sullivan - div - bdy saturated}.
Suppose that the diagonal action of $G$ on $(\partial^2 X,  \mathfrak B^+ \otimes \mathfrak B^+, \nu_o \otimes \nu_o)$ is ergodic.
Hence so is the one on $(\partial^2_0 X, \mathfrak B^+ \otimes \mathfrak B^+, \nu_o \otimes \nu_o)$.
It follows from \autoref{res: ergodicity implies no atom} that $(\partial X, \mathfrak B^+, \nu_o)$ has no atom.
Hence neither has $(\partial^2X, \mathfrak B^+ \otimes \mathfrak B^+, \nu_o \otimes \nu_o)$.
However since $G$ is countable, it implies that the action of $G$ on $(\partial^2X, \mathfrak B^+ \otimes \mathfrak B^+, \nu_o \otimes \nu_o)$ is conservative.

Let us focus on \ref{eqn: hopf-tsuji-sullivan - div - double ergo minus diag} $\Rightarrow$  \ref{eqn: hopf-tsuji-sullivan - div - double ergo}.
Suppose that the diagonal action of $G$ on $(\partial^2 X,  \mathfrak B^+ \otimes \mathfrak B^+, \nu_o \otimes \nu_o)$ is ergodic.
In view of the previous study we know that $\nu$ gives full measure to the contracting limit set, hence to the radial limit set.
It follows that the diagonal $\Delta = (\partial X \times \partial X) \setminus \partial^2 X$ has zero $(\nu_o \otimes \nu_o)$-measure (\autoref{res: measure diagonal}).
Hence the action of $G$ on $(\partial X \times \partial X,  \mathfrak B^+ \otimes \mathfrak B^+, \nu_o \otimes \nu_o)$ is ergodic as well.

Recall that $\mathfrak L^+$ stands for an admissible sub-$\sigma$-algebra.
Unwrapping the definitions, one observes that \ref{eqn: hopf-tsuji-sullivan - div - ergo flow} and \ref{eqn: hopf-tsuji-sullivan - div - double ergo minus diag} are equivalent to the ergodicity of the action of $G \times \R$ on $(SX, \mathfrak L^+, m)$.
This proves \ref{eqn: hopf-tsuji-sullivan - div - ergo flow} $\Leftrightarrow$ \ref{eqn: hopf-tsuji-sullivan - div - double ergo minus diag}.

\paragraph{Convergent case.}
\autoref{res: dissipativity geodesic flow} gives \ref{eqn: hopf-tsuji-sullivan - conv - radial} $\Rightarrow$ \ref{eqn: hopf-tsuji-sullivan - conv - bdy} and \ref{eqn: hopf-tsuji-sullivan - conv - radial} $\Rightarrow$ \ref{eqn: hopf-tsuji-sullivan - conv - dissipative}.
The converse implications were already proved during the divergent case.
\end{proof}


\bigskip
\noindent
\emph{R\'emi Coulon} \\
Université Bourgogne Europe, CNRS \\
IMB - UMR 5584 \\
F-21000 Dijon, France\\
\\
CRM, CNRS\\
IRL 3457 \\
Montréal (Québec) H3C 3J7, Canada\\
 \\
\texttt{remi.coulon@cnrs.fr} \\
\texttt{http://rcoulon.perso.math.cnrs.fr}

\end{document}